\documentclass[a4paper,leqno]{article}
\usepackage[latin1]{inputenc}
\usepackage{amsfonts}
\usepackage{enumitem}
\usepackage{color}
\usepackage{textcomp}
\usepackage [english]{babel}
\usepackage{graphicx}
\usepackage{amssymb,amsmath,amsthm}
\newtheorem{define}{Definition}[section]
\newtheorem{pro}{Proposition}[section]
\newtheorem{Lem}{Lemma}[section]
\newtheorem{cor}{Corollary}[section]

\newtheorem{remark}{Remark}[section]
\theoremstyle{plain} \newtheorem{thm}{Theorem}[section]
\newcommand{\eqdef}{\stackrel{\mathrm{def}}{=}}   % Composed symbols

\newcommand{\dotV}{\overset{\bullet}{V}}

\let\div\undefined
\DeclareMathOperator{\div}{div}
\DeclareMathOperator{\dist}{dist}

\catcode`\@=11
\makeatletter

\@addtoreset{equation}{section}
\def\R{\mathbb R}
\def\N{\mathbb N}
\begin{document}
\title{A Picone identity for variable exponent operators and applications.}
\author{
{ \sc Rakesh Arora, Jacques Giacomoni, Guillaume Warnault} \\
{\sc\footnotesize  LMAP (UMR E2S-UPPA CNRS 5142) Bat. IPRA,}\\
{\sc\footnotesize    Avenue de l'Universit\'e }\\
{\sc\footnotesize   64013 Pau, France}
\\
 {\sc\footnotesize  e-mail: rakesh.arora@univ-pau.fr},
{\sc\footnotesize  jacques.giacomoni@univ-pau.fr},\\
{\sc\footnotesize  guillaume.warnault@univ-pau.fr}\\
} 
\maketitle
\begin{abstract} In this work, we establish a new Picone identity for anisotropic quasilinear operators, such as the $p(x)$-Laplacian defined as $\mbox{div}(|\nabla u|^{p(x)-2} \nabla u).$ Our extension  provides a new version of the  Diaz-Saa inequality and new uniqueness results to some quasilinear elliptic equations with variable exponents. This new Picone identity can be also used to prove some accretivity property to a class of fast diffusion equations involving variable exponents.  Using this, we prove for this class of parabolic equations a new weak comparison principle.
\end{abstract}

%%%%%%%%%%%%%%%%%%%%%%%%%with variable exponents%%%%%%%%%%%%% Introduction and Main Results %%%%%%%%%%%%%%%%%%%%%%%%%%%%%%%%%%%%%%%%%%%%%%%%%%

\vspace{1cm}
\section{Introduction and main results}
The main aim of this paper is to prove a new version of the Picone identity involving quasilinear elliptic operators with variable exponent. The Picone identity is already known for homogeneous quasilinear elliptic as $p$-Laplacian with $1<p<\infty$.\\
In \cite{*11}, M. Picone considers the homogeneous second order linear differential system 
\begin{equation*}
     \left\{
         \begin{alignedat}{2} 
             {} (a_1(x) u')' + a_2(x) u
             & {}= 0
             \\
            (b_1(x) v')' + b_2(x) v & {}= 0
          \end{alignedat}
     \right.
\end{equation*}
and proved for differentiable functions $u ,v \neq 0$ the pointwise relation:
\begin{equation}\label{piconeder}
    \bigg(\dfrac{u}{v}(a_1 u'v- b_1 uv')\bigg)'= (b_2-a_2) u^2 + (a_1-b_1)u'^2+ b_1 \bigg(u'- \dfrac{v'u}{v}\bigg)^2
\end{equation}
and in \cite{*12}, extended \eqref{piconeder} to the Laplace operator,
{\it i.e.} for differentiable functions $u \geq 0 , v>0$ one has
\begin{equation}\label{piconelap}
    |\nabla u|^2 + \dfrac{u^2}{v^2} |\nabla v|^2 - 2 \dfrac{u}{v} \nabla u. \nabla v = |\nabla u|^2 - \nabla \bigg(\dfrac{u^2}{v}\bigg). \nabla v \geq 0.
\end{equation}
In \cite{*13}, Allegretto and Huang extended \eqref{piconelap} to the $p$-Laplacian operator with $1<p<\infty$. Precisely,  for differentiable functions  $v >0$ and $u \geq 0$ we have
\begin{equation*}%\label{piconeplap}
%|\nabla u|^p +(p-1) \dfrac{u^p}{v^p} |\nabla v|^p -p \dfrac{u^{p-1}}{v^{p-1}} \nabla u |\nabla v|^{p-2} \nabla v =
|\nabla u|^p -  |\nabla v|^{p-2} \nabla v. \nabla\bigg(\dfrac{u^p}{v^{p-1}} \bigg) \geq 0.
\end{equation*}
Picone identity plays an important role for proving qualitative properties of differential operators. In this regard, various attempts have been made to generalize Picone identity for different types of differential equations. At the same time, the study of differential equations and variational problems with variable exponents are getting more and more attention. Indeed, the mathematical problems related to nonstandard $p(x)$-growth conditions are connected to many different areas as the nonlinear elasticity theory and non-Newtonian fluids models (see \cite{*22,*15}). In particular the importance of investigating these kinds of problems lies in modelling various anisotropic features that occur in electrorheological fluids models, image restoration \cite{*14}, filtration process in complex media, stratigraphy problems \cite{*28} and heterogeneous biological interactions \cite{*27}. The mathematical framework to deal with these problems are the generalized Orlicz Space $ L^{p(x)}(\Omega)$ and the generalized Orlicz-Sobolev Space $W^{1,p(x)}(\Omega)$. We refer to \cite{*6,*2,*3,*10, VR, KR} for the existence and regularity of minimizers in variational problems. 

In \cite{*19,*17}, several applications of Picone-type identity for $p(\cdot)=\mathrm{constant}$ case have been obtained. This original identity is not further applicable for differential equations with $p(x)$-growth conditions. So, it is relevant to establish a new version of the Picone identity to include a large class of nonstandard $p(x)$-growth problems. In \cite{*10,GTW,*15}  convexity arguments to homogeneous functionals have been used to deal with quasilinear elliptic and parabolic problems with variable exponents. In the present paper,  taking advantage of our new  Picone pointwise  identity, we give further applications in the context of elliptic and parabolic problems.

%%%%%%%%%%%%%%%%%%%%%%%%%%%%%%%%%%%%%%%%%%%%              Notation     %%%%%%%%%%%%%%%%%%%%%%%%%%%%%%%%%%%%%%%%%%%%%%%%%%%%%%%%%%%
Before giving the statement of our main results, we first introduce notations and function spaces. Let $\Omega \subset \mathbb{R}^N, N \geq 1$. 
%and $p \in C(\overline{\Omega})$ such that  $$1< p_{-}\eqdef \min_{\overline{\Omega}} p(x) \leq p(x) \leq p_{+}\eqdef \max_{\overline{\Omega}} p(x) < \infty.$$\\
We recall some definitions of variable exponent Lebesgue and Sobolev spaces. Let $\mathcal{P}(\Omega)$ be the set of all measurable function $p: \Omega \rightarrow [1,\infty[$ in $N$-dimensional Lebesgue measure. Define
$$\rho_p(u) \eqdef \int_{\Omega}|u|^{p(x)}~dx.$$
$$L^{p(x)}(\Omega) = \{u:\Omega \to \mathbb{R} \ | \ u \text{\ is measurable and } \rho_p(u) < \infty \}$$ endowed with the norm
$$\|u\|_{L^{p(x)}} =\inf\{\sigma >0 : \rho_p\bigg(\frac{u}{\sigma}\bigg)\leq 1 \}.$$
The corresponding Sobolev space is defined as follows:
 $$W^{1,p(x)}(\Omega) = \{ u \in L^{p(x)}(\Omega) : |\nabla u| \in L^{p(x)}(\Omega)\}$$ endowed with the norm
$$\|u\|_{W^{1,p(x)}} = \|u\|_{L^{p(x)}} + \|\nabla u\|_{L^{p(x)}}$$ and $ W_0^{1,p(x)}(\Omega) = W^{1,1}_0(\Omega)\cap{W^{1,p(x)}(\Omega)}.$\\
In the sequel, we assume that $\Omega$ satisfies:
\begin{itemize}
\item[${(\mathbf \Omega)}$] For $N=1$, $\Omega$ is a bounded open interval and for $N \geq 2$, $\Omega$ is a bounded domain whose the boundary $\partial \Omega$ is a compact manifold of class $C^{1, \gamma}$ for some $\gamma \in (0,1)$ and satisfies the interior sphere condition at every point of $\partial \Omega$.
\end{itemize}
Throughout the paper, we also assume that $p \in C^1(\overline{\Omega})$.  In addition, we suppose that $$1< p_{-} \eqdef  \min_{\overline{\Omega}} p(x) \leq p(x) \leq p_{+}\eqdef \max_{\overline{\Omega}} p(x) < \infty.$$ Then, $W_0^{1,p(x)}(\Omega) =\overline{C^\infty_0(\Omega)}^{W^{1,p(x)}(\Omega)}$.
%where the $\mathcal{P}^{log}(\Omega)$ class of variable exponents is defined by
%$$\mathcal{P}^{log}(\Omega)= \{p \in \mathcal{P}(\Omega) : \dfrac{1}{p} \ \ \text{is globally log-H\"older continuous}\}.$$
%In particular, for any $p \in \mathcal{P}^{log}(\Omega),$ there exists a function $k$ such that
%$$\forall (x,y)\in \Omega \times \Omega,\ |p(x)-p(y)| \leq k(|x-y|)\ \text{and} \ \limsup_{t \rightarrow 0+} (-k(t)\ln t) < +\infty.$$
We also recall some well-known properties on $L^{p(x)}$ spaces (see \cite{KR}).
\begin{pro}\label{P1}
Let $p\in L^\infty(\Omega)$. 
Then for any $u\in L^{p(x)}(\Omega)$ we have:
\begin{enumerate}[label={\it(\roman*)}]
\item $\rho_p(u/\|u\|_{L^{p(x)}})=1$.
\item $\|u\|_{L^{p(x)}}\rightarrow 0$\; \text{if and only if}\;  $\rho_p(u)\rightarrow 0$.
\item $L^{p_c(x)}(\Omega)$ is the dual space of $L^{p(x)}(\Omega)$ where we denote by $p_c$ the conjugate exponent of $p$ defined as
$$p_c(x)=\frac{p(x)}{p(x)-1}.$$
 \end{enumerate}
\end{pro}
\noindent Proposition \ref{P1} {\it (i)} implies that: if $\|u\|_{L^{p(x)}}\geq 1$,
\begin{equation}\label{a}
\|u\|_{L^{p(x)}}^{p_{-}}\leq \rho_p(u)\leq\|u\|_{L^{p(x)}}^{p_{+}}
\end{equation}
and if $\|u\|_{L^{p(x)}}\leq 1$
\begin{equation}\label{b}
\|u\|_{L^{p(x)}}^{p_{+}}\leq \rho_p(u)\leq\|u\|_{L^{p(x)}}^{p_{-}}.
\end{equation}
Moreover, we have also the generalized H\"older inequality: for $p$ measurable function in $\Omega$, there exists a constant $C=C(p_+,p_-)\geq 1$ such that for any $f\in L^{p(x)}(\Omega)$ and $g\in L^{p_c(x)}(\Omega)$
\begin{equation}\label{hi}
 \int_\Omega |f(x)g(x)|~dx\leq C\|f\|_{L^{p(x)}}\|g\|_{L^{p_c(x)}}.
 \end{equation}
 
 In Section 2, we prove the Picone identity for a general class of nonlinear operator. More precisely, we consider a continuous operator $A:\  \Omega \times \mathbb{R}^N \to \mathbb{R}$ such that $(x,\xi) \rightarrow A(x, \xi)$ is differentiable with respect to variable $\xi $ and satisfies:
\begin{itemize}%\label{A1 and A2}
     \item[$(A1)$] $\xi \to A(x,\xi)$ is positively $p(x)$-homogeneous {\it i.e.} $A(x,t \xi)= t^{p(x)} A(x,\xi),\ \forall\ t \\ \in \mathbb{R}^+,\  \xi \in \mathbb{R}^N \text{\ and a.e.\ } x \in \Omega.$
      \item[$(A2)$]   $\xi \to A(x,\xi) \ \text{is strictly convex}$ for any $x \in \Omega.$
\end{itemize}
\begin{remark}
From the assumptions of $A$, we deduce $A(x,\xi) > 0$ for $\xi \neq 0$ and the symmetry $A(x, \xi)= A(x, -\xi)$ for any $x \in \Omega$ and $\xi \in \mathbb{R}^N.$ 
\end{remark} 
By using  the convexity and  the $p(x)$-homogeneity of the operator $A$, we prove the following extension of the Picone identity: 

 %%%%%%%%%%%%%%%%%%%%%%%%%%%%%%%%%%   Statement of Results%%%%%%%%%%%%%%%%%%%%%%%%%%%%%%%
 \begin{thm}[Picone identity]\label{picone}
Let $A: \Omega \times \mathbb{R}^N \to \mathbb{R}$ is a continuous and differentiable function satisfying $(A1)$ and $(A2)$. Let $v_0,v \in L^\infty(\Omega)$ belonging to $\dot{V}_+^r \eqdef \{v:\Omega \to (0,+\infty) \ | \ v^{\frac1r} \in W^{1,p(x)}_0(\Omega)\}$ for  some $r \geq 1$. Then
\begin{equation*}
\dfrac{1}{p(x)} \bigg< \partial_\xi A(x,\nabla v_0^{1/r}),\nabla\bigg({\dfrac{v}{v_0^{(r-1)/r}}}\bigg) \bigg> \leq A^{\frac{r}{p(x)}}(x,\nabla v^{1/r}) \  A^{\frac{(p(x)-r)}{p(x)}}(x,\nabla v_0^{1/r})
\end{equation*}
where $\langle .,. \rangle$ is the inner scalar product and the above inequality is strict if $r >1$ or $\frac{v}{v_0} \not\equiv \mathrm{Const}>0.$ 
 \end{thm}
\noindent From the above Picone identity, we can show an extension of the famous Diaz-Saa inequality to the class of operators with variable exponent. This inequality is strongly linked to the strict convexity of some associated homogeneous energy type functional. 
 \begin{thm}[Diaz-Saa inequality]{\label{Diaz-Saa}}
Let $A: \Omega \times \mathbb{R}^N \to \mathbb{R}$ is a continuous and differentiable function satisfying $(A1)$ and $(A2)$ and define  $a(x,\xi)=(a_i(x,\xi))_i \eqdef  \left(\dfrac{1}{p(x)} \partial_{\xi_i} A(x,\xi)\right)_i$. Assume in addition that there exists $\Lambda >0$ such that \\ 
\begin{equation*}
\begin{split}
a \in C^1(\Omega\times (\R^N\backslash\{0\}))^N\ \  \text{and} \ \ \sum_{i,j=1}^N \left|\dfrac{\partial a_i(x,\xi)}{\partial \xi_j}\right| \leq \Lambda |\xi|^{p(x)-2}
\end{split}
\end{equation*}
%and 
%\begin{equation}{ellipticity}
%\end{equation}
for all $(x,\xi)\in \Omega\times\R^N\backslash\{0\}$. Then, we have in the sense of distributions, for any $r \in [1,p_-]$ 
\begin{equation}\label{*}
   \int_\Omega  \bigg(- \frac{\div(a(x,\nabla w_1))}{w_1^{r-1}} + \frac{\div (a(x,\nabla w_2))}{w_2^{r-1}}\bigg) (w_1^r-w_2^r) ~dx \geq 0
\end{equation}
 for any $w_1,w_2 \in W_0^{1,p(x)}(\Omega)$, positive in $\Omega$ such that $\dfrac{w_1}{w_2},\dfrac{w_2}{w_1} \in L^{\infty}(\Omega).$
 Moreover, if the equality occurs in \eqref{*}, then  $w_1/w_2$ is constant in $\Omega$.
 If $p(x) \not\equiv r$ in  $\Omega$ then even $w_1=w_2$ holds in $\Omega.$
 \end{thm}
%\begin{remark}\label{diazremark}
%The distributional inequality \eqref{*} is understood in the following way that 
%$$\textcolor{red}{\int_{\Omega} a(x, \nabla w_1). \nabla \left(w_1- \frac{w_2^r}{w_1^{r-1}}\right) ~dx \geq %\int_{\Omega} a(x, \nabla w_2). \nabla \left(\frac{w_1^r}{w_2^{r-1}}-w_2\right) ~dx} $$
%for any $w_1,w_2 \in W_0^{1,p(x)}(\Omega)$, positive in $\Omega$ such that $\dfrac{w_1}{w_2},\dfrac{w_2}{w_1} \in L^{\infty}(\Omega).$
%\end{remark}
%\begin{remark}
%We can replace the hypothesis \eqref{sense} by assuming $a(x, \nabla w) \in L^{p_c(x)}$ for any $w \in W_0^{1,p(x)}(\Omega).$ Condition \eqref{sense} ensures that  \eqref{*} is well-defined.  
%\end{remark}
\noindent 
%Picone identity is a suitable tool in the study of qualitative properties of solutions of nonlinear problems. 
In sections 3, 4 and 5, we derive some applications of the new Picone identity. Precisely, we investigate the solvability of some boundary problems involving quasilinear elliptic operators with variable exponent.\\ 
In section 3, we consider the following nonlinear problem: 
\begin{equation}\label{10}
     \left\{
         \begin{alignedat}{2} 
             {} -\Delta_{p(x)} u+ g(x,u)
             & {}= f(x,u)  
             && \quad\mbox{ in }\,  \Omega;
             \\
             u & {}> 0
             && \quad\mbox{ in }\,  \Omega;
             \\
             u & {}= 0
             && \quad\mbox{ on }\,  \partial\Omega .
          \end{alignedat}
     \right.
\end{equation}
The extended Picone identity can be reformulated as in  Lemma \ref{implem} below.  Together with the strong maximum principle and elliptic regularity, this identity can be used to prove the uniqueness of weak solutions to elliptic equations as \eqref{10}. 
%For this we adapt strong maximum principle and Hopf Lemma for nonlinear problems with nonstandard $p(x)$-%growth conditions.
 In particular, we establish the following result:
\begin{thm}\label{Application 2}
Let $f, g: \Omega\times [0,\infty)\to \mathbb{R}^+$ be defined as $f(x,t)= h(x)t^{q(x)-1}$ and $g(x,t)= l(x)t^{s(x)-1}$ with $1\leq q,s\in C(\overline{\Omega})$ such that
\begin{itemize}
\item $q_+ < p_- <s_-$ and $q_- \geq 1$;\vspace{-0.3cm}
\item $h,l \in L^{\infty}(\Omega)$, positive functions such that $x\to \dfrac{h(x)}{l(x)} \in L^{\infty}(\Omega).$
\end{itemize}
Then, there exists a weak solution $u$ to \eqref{10}, i.e. $u$ belongs to $W^{1,p(x)}_0(\Omega)\cap L^{s(x)}(\Omega)$ and satisfies for any $\phi\in W^{1,p(x)}_0(\Omega)\cap L^{s(x)}(\Omega)$:
$$\int_\Omega |\nabla u|^{p(x)-2}\nabla u.\nabla \phi ~dx=\int_\Omega (f(x,u)-g(x,u))\phi ~dx.$$ 
Furthermore $u\in C^{1,\alpha}(\overline{\Omega})$ for some $\alpha \in (0,1)$ and $0 \leq u^{s_--q_+}\leq \max\{\|\frac{h}{l}\|_{L^{\infty}},1\}$ {\it a.e.} in $\Omega$.\\
Assume in addition that $x\to \dfrac{l(x)}{h(x)}$ belongs to $L^\infty(\Omega)$, then $u \in C^0_d(\overline\Omega)^+ \eqdef \{ v\in C_0(\overline\Omega)\ | \ \exists \ c_1,\ c_2\in \R^+_*:\ c_1 \leq \dfrac{v}{\dist(x, \partial \Omega)} \leq c_2\}$ and is the unique weak solution to \eqref{10}.
\end{thm}
Regarding the current literature, Theorem \ref{Application 2} does not require any subcritical growth condition for $g$ to establish existence and uniqueness of the weak solution to \eqref{10}.

In section $4$, we study a nonlinear fast diffusion equation (F.D.E. for short) driven by $p(x)$-Laplacian.  From the physical Fick's law, the diffusion coefficient of our problem is then proportional to $|\nabla u(x,t)|^{p(x)-2}$. 
%Hence the nonlinear diffusion driven by our problem strongly depends on the gradient of the density $u(x,t)$ at %position $x$ and time $t.$ \\ 
It naturally leads to investigate the following F.D.E. type problem:
\begin{equation}\label{P}
     \left\{
         \begin{alignedat}{2} 
             {} \frac{q}{2q-1} \partial_t (u^{2q-1}) -\Delta_{p(x)} u
             & {}= f(x,u) + h(t,x) u^{q-1}  
             && \quad\mbox{ in } \, Q_T;
             \\
             u & {}> 0
             && \quad\mbox{ in }\,  Q_T;
             \\
             u & {}= 0
             && \quad\mbox{ on }\, \Gamma;
             \\
             u(0,.) & {}= u_0
             && \quad\mbox{ in }\, \Omega 
          \end{alignedat}
     \right.
\end{equation}
where $q\in (1,p_-)$, $Q_T = (0,T) \times \Omega$ and $\Gamma = (0,T) \times \partial\Omega$ for some $T>0$. We suppose that $h\in L^\infty(Q_T)$ and nonnegative. The assumptions on  $f$ are given by
\begin{enumerate}
 \item[$(f_1)$] $f : \Omega \times \mathbb{R^+} \to \mathbb{R^+}$ is a function such that $f(x,0)=0$ for all $x \in \Omega$ and $f \not\equiv 0$;
% \item[$(f_3)$] $\lim_{s  \to +\infty} \dfrac{f(x,s)}{s^{p_- -1}} =0 $ uniformly in $x$;
 \item[$(f_2)$] $\lim_{s \to 0^+}\dfrac{f(x,s)}{s^{2q -1}} = +\infty $ uniformly in $x$;
 \item[$(f_3)$] for any $x \in \Omega, s\to \dfrac{f(x,s)}{s^{q-1}}$ is nonincreasing in $\mathbb{R^+}\backslash \{0\}$.
\end{enumerate}
%\begin{remark}
%We need $u_0 >0$ to ensure that $ u(t)>0 $. This is guaranteed by the existence of positive subsolution and by strong maximum principle. 
%\end{remark}
\begin{remark}
\noindent Conditions $(f_1)$ and $(f_3)$ imply 
there exist positive constant $C_1, C_2$ such that for any  $(x,s) \in \Omega \times \mathbb{R}^+$:
\begin{equation*}
0 \leq f(x,s) \leq C_1 + C_2 s^{q -1},
\end{equation*}
{ \it i.e.} $f$ has a strict subhomogeneous growth.
\end{remark}
%\begin{remark}
 %We can also generalize the condition $(f_3)$ to more extent as in \cite{*10}.\\
 %$(f_{2'}) :$ $\lim\sup_{s  \to +\infty} \dfrac{f(x,s)}{s^{p_- -1}} \leq C$ uniformly in $x$ where C is a suitable constant which depends upon $p$ and embeddings.
%\end{remark}
\noindent We set $\mathcal R $ the operator defined by $\mathcal Rv= \dfrac{-\Delta_{p(x)}(v^{1/q})}{v^{(q-1)/q}}- \dfrac{f(x,v^{1/q})}{v^{(q-1)/q}}$ and the associated domain
$$\mathcal{D}(\mathcal R)= \{v: \Omega \to (0,\infty): v^{1/q} \in W_0^{1,p(x)}(\Omega), v \in L^{2}(\Omega), \mathcal Rv \in L^2(\Omega)\}.$$
%$$C_d^0(\overline \Omega)^+ =\{ u\in C_0(\overline\Omega)\ | \ \exists \ c_1,\ c_2\in \R^+_*:\ c_1 \leq \dfrac{u}%{dist(x, \partial \Omega)} \leq c_2\}.$$
Note that $\mathcal{D}(\mathcal R)$ contains for instance solutions to \eqref{subsolu}. One can also easily check that solutions to \eqref{supprob} belong to $\overline{\mathcal{D}(\mathcal R)}^{L^2(\Omega)}$.
In the sequel, we denote $X^+\eqdef \{x\in X\,|\, x\geq 0\}$  the associated positive cone of a given real vector space $X$.\\
%\begin{thm}\label{main2}
%Let $q\in (1,min\{\frac{N}{2}, p^-\})$. Assume the initial data $u_0$ satisfies $u_0^q\in C^0_d(\overline{\Omega})^+, h \in L^\infty(Q_T)^+, h \geq m_h >0$ and $f$ satisfies $(f_1)$-$(f_3)$. Then there exists a unique positive weak solution to $\eqref{P}.$\\ 
%Furthermore, let $u$ and $v$ the weak solutions with initial data $u_0$, $v_0$ where $ u_0^q, v_0^q \in C^0_d(\overline{\Omega})^+$ and  $h,\ g\in L^\infty(Q_T)^+ , h \geq m_h >0, g\geq m_g >0$ respectively. Then, for any $0\leq t\leq T$, 
%\begin{equation}\label{contraction}
%\Vert (u^q(t)-v^q(t))^+\Vert_{L^2(\Omega)}\leq \Vert (u_0^q-v_0^q)^+\Vert_{L^2(\Omega)}+\Vert (h-g)^+\Vert_{L^2(Q_T)}.
%\end{equation}
%\end{thm}
%We establish the result regarding the weak solution of \eqref{P}. Since $\frac{q}{2q-1} \partial_t(v^{2q-1}) = v^{q-1} \partial_t(v^q)$ in weak sense under some conditions on the solutions. Therefore 
In order to establish existence and properties of weak solutions to \eqref{P}, we investigate the following related parabolic problem:
 \begin{equation}\label{WeakS1}
     \left\{
         \begin{alignedat}{2} 
             {} v^{q-1} \partial_t (v^q) - \Delta_{p(x)} v
             & {}= h(t,x) v^{q-1} + f(x,v)
             && \quad\mbox{ in }\, Q_T\,;
             \\
             v & {}>0
             && \quad\mbox{ in }\, Q_T\,;
             \\
             v & {}= 0
             && \quad\mbox{ on }\, \Gamma \,;
             \\
             v(0,.) & {}= v_0(x) > 0
             && \quad\mbox{ in }\, \Omega \,.
          \end{alignedat}
     \right.
\end{equation}
 The notion of weak solution for \eqref{WeakS1} is given as follows:
\begin{define}
A weak solution to \eqref{WeakS1} is any positive function \newline $v \in L^{\infty}(0,T;W_0^{1,p(x)}(\Omega))\cap L^{\infty}(Q_T) \cap C(0,T; L^r(\Omega))$ for any $r \geq 1$ such that $\partial_t(v^q) \in L^{2}(Q_T)$ and for any $\phi \in C_0^{\infty}(Q_T)$ satisfies
\begin{equation}\label{weaksodef}
\begin{split}
\int_0^T \int_{\Omega} \partial_t(v^q) v^{q-1} \phi ~dxdt &+ \int_0^T \int_{\Omega} |\nabla v|^{p(x)-2} \nabla v. \nabla \phi ~dxdt \\
&= \int_0^T \int_{\Omega} h(t,x) v^{q-1} \phi ~dxdt + \int_0^T \int_{\Omega} f(x,v) \phi ~dxdt.
\end{split}
\end{equation} 
\end{define}
\noindent Concerning \eqref{WeakS1},  we prove the following results:
\begin{thm}\label{wea}
Let $T > 0, v_0 \in C^{0}_{d}(\overline\Omega)^+ \cap W_{0}^{1,p(x)}(\Omega)$. In addition, there exists $h_0 \in L^{\infty}(\Omega)$, $h_0\not\equiv 0$ and $h(t,x) \geq h_0(x)\geq0$ for a.e $x \in \Omega$, for a.e. $t\geq 0$. Assume in addition $q \leq \min\{\frac{N}{2}+1, p_-\}$ and $f$ satisfies $(f_1)$-$(f_3)$. Then there exists a weak solution to \eqref{WeakS1}. 
%Moreover $v\in L^{\infty}(Q_T)^+.$
\end{thm}
Based on the accretivity of ${\mathcal R}$ with domain ${\mathcal D}({\mathcal R})$, we show the following result providing a contraction property for weak solutions to \eqref{WeakS1} under suitable conditions on initial data:
\begin{thm}\label{contr}
Let $v_1$ and $v_2$ are weak solutions of \eqref{WeakS1} with initial data $ u_0, v_0 \in C^0_d(\overline\Omega)^+ \cap W_{0}^{1,p(x)}(\Omega)$ and such that $u_0^q$, $v_0^q\in\overline{\mathcal{D}(\mathcal R)}^{L^2(\Omega)} $ and $h, g\in L^\infty(Q_T)$, such that $h\geq h_0$, $g\geq g_0$ with $h_0, g_0$ as in Theorem \ref{wea}. Then, for any $0\leq t\leq T$, 
\begin{equation}\label{ntk}
\Vert (v_1^q(t)-v^q_2(t))^+\Vert_{L^2}\leq \Vert (u_0^q-v_0^q)^+\Vert_{L^2}+\int_0^t\Vert (h(s)-g(s))^+\Vert_{L^2}\,ds.
\end{equation}
\end{thm}
Furthermore, using a similar approach as in \cite{*27}, we consider for $\epsilon>0$ the perturbed operator 
$\mathcal R_\epsilon  v= \dfrac{-\Delta_{p(x)}(v^{1/q})}{(v+\epsilon)^{(q-1)/q}}- \dfrac{f(x,v^{1/q})}{(v+\epsilon)^{(q-1)/q}}$. If $p_-\geq 2$, we can prove (as in Proposition 2.6 in \cite{*27}) that 
$$\displaystyle\overline{\mathcal{D}(\mathcal R_\epsilon)}^{L^2(\Omega)}\supset \dot{V}_+^q\cap C^{0}_{d^q}(\overline\Omega)^+ .$$ Arguing as in Theorem \ref{contr} with the operator ${\mathcal R}_\epsilon$  instead of ${\mathcal R}$ and passing to the limit as $\epsilon\to 0^+$, we get:
\begin{cor}\label{pg2}
Assume $p_-\geq 2$. Let $v_1$ and $v_2$ are weak solutions of \eqref{WeakS1} with initial data $ u_0, v_0 \in C^0_d(\overline\Omega)^+ \cap W_{0}^{1,p(x)}(\Omega)$. Then Theorem \ref{contr} holds.
\end{cor}
\noindent From Theorem \ref{contr}, we derive the following comparison principle from which uniqueness of the weak solution to problem \eqref{WeakS1} follows:
\begin{cor}\label{compari}
Let $u$ and $v$ are the weak solutions of \eqref{WeakS1} with initial data $ u_0, v_0$ satisfying conditions in Theorem \ref{contr} or Corollary \ref{pg2}. Assume $u_0 \leq v_0$  and $h,\ g\in L^\infty(Q_T) $, $h_0 \in L^\infty(\Omega)$ such that  and $0 < h_0 \leq h \leq g$. Then $u \leq v.$ 
\end{cor}
\begin{remark}\label{conversion}
If $v \in L^{\infty}(Q_T)^+$ then from Proposition $9.5$ in \normalfont{\cite{*4}} we obtain $\frac{q}{2q-1} \partial_t(v^{2q-1}) = v^{q-1} \partial_t(v^q)= q v^{2q-2} \partial_t\,v$ in weak sense. 
\end{remark}
\begin{remark}
From Theorem \ref{contr}, we can derive stabilization  results for the evolution equation \eqref{P} in  $L^q(\Omega)$ with $q\in [2,\infty)$ (see \cite{*10} in this regard).
\end{remark}
From the above remark, under assumptions given in Theorem \ref{wea}, we obtain the existence of weak solutions to \eqref{P} satisfying the monotonicity properties in Theorem \ref{contr} and Corollaries \ref{pg2}, \ref{compari}.
We highlight that in our knowledge there is no result available in the current literature  about F.D.E. with variable exponent. In this regard our results are completely new. 

In the previous applications, the condition $(A1)$ plays a crucial role to get suitable convexity property of energy functionals. In section 5, we study a quasilinear elliptic problem where this condition is not satisfied. Precisely, given $\epsilon>0$, we study 
%an application of strict convexity of a $p(x)$ powertype energy functional defined by 
%$$\mathcal{J}_m(v) \eqdef  \int_{\Omega} \dfrac{1}{p(x)} A(x,v^{1/m},\nabla v^{1/m})^{p(x)} ~dx$$
%for $m\in [1,p_-]$ and for all positive functions such that $|v|^{1/m} \in W_0^{1,p(x)}(\Omega) \cap L^{\infty}%(\Omega)$. In this way, we study 
the following nonhomogeneous quasilinear elliptic problem: 
\begin{equation}\label{2}
     \left\{
         \begin{alignedat}{2} 
             {} -\div((|\nabla u|^2+ \epsilon u^2)^{\frac{p(x)-2}{2}} \nabla u) - (|\nabla u|^2+ \epsilon u^2)^{\frac{p(x)-2}{2}} \epsilon u 
             & {}= g(x,u)
             && \quad\mbox{ in }\, \Omega \,;
             \\
             u & {}= 0
             && \quad\mbox{ on }\, \partial\Omega \,;
             \\
             u & {}>  0
             && \quad\mbox{ in }\, \Omega \,
          \end{alignedat}
     \right.
\end{equation}
%% \begin{equation}\label{2}
%\left\{\begin{array}{l}
 %     -\nabla.((|\nabla u|^2+ \epsilon u^2)^{\frac{p(x)-2}{2}} \nabla u) - (|\nabla u|^2+ \epsilon u^2)^{\frac{p(x)-2}{2}} \epsilon u =f(x,u) \ 
%			\mbox{in } \Omega,\\
%			u>0 \quad \mbox{in } \Omega, \quad u=0 \  \mbox{on } \partial\Omega
%			\end{array}
%			\right.
 % \end{equation}
where  $g$ satisfies $(f_1)$ and $(\tilde{g})$ for some $m \in [1, p_-]$:
\begin{enumerate}
 \item[$(\tilde{g})$] For any $x \in \Omega,\ s\to \dfrac{g(x,s)}{s^{m-1}}$ is decreasing in $\mathbb{R^+}\backslash \{0\}$ and a.e. in $\Omega$.
\end{enumerate}
 Then we prove the following result:
 \begin{thm}\label{ray-strict}
Assume that $g$ satisfies $(f_1)$ and $(\tilde{g})$. Then for any $\epsilon$, \eqref{2} admits one and only one positive weak solution. Furthermore, $u \in C^1(\overline{\Omega})$, $u > 0$ in $\Omega$ and $\dfrac{\partial u}{\partial \vec n} <0$ on $\partial \Omega.$
 \end{thm}
To get the uniqueness result contained in Theorem \ref{ray-strict}, we exploit the hidden convexity property of the associated energy functional in the interior of positive cone of $C^1(\overline{\Omega})$.
%%%%%%%%%%%%%%%%%%%%%%%%%%%%%%%%%%    Picone identity and its proof%%%%%%%%%%%%%%%%%%%%%%%%%%%%%%%%%%
\section{Picone identity and Diaz-Saa inequality }\label{section2}
\subsection{Picone identity}
First we recall the notion of strict ray-convexity.
\begin{define}
Let $X$ be a real vector space.  Let $\dotV$ be a non empty cone in $X$. A function $J: \dotV \to \mathbb{R}$ is ray-strictly convex if for all $v_1,v_2 \in \dotV$ and for all $\theta \in (0,1)$
$$J((1-\theta) v_1+ \theta v_2)\leq (1-\theta) J(v_1)+\theta J(v_2)$$  
where the inequality is always strict unless $v_1= Cv_2$ for some $C>0$.
\end{define}
\noindent Then we have the following result:
\begin{pro}\label{1} Let $A$ satisfying $(A1)$ and $(A2)$ and let $r\geq 1$. Then, for any $x\in \Omega$ the map $\xi\to N_r(x,\xi) \eqdef  A(x,\xi)^{r/p(x)}$ is positively $r$-homogeneous and ray-strictly convex. For $r >1$, $\xi \to N_r(x,\xi)$ is even strictly convex.
\end{pro}
\begin{proof}
We begin by the case $r=1$. For any $t \in \mathbb{R}^+$, we have $N_1(x,t \xi)= t N_1(x,\xi)$. Furthermore,
$$A(x,(1-t)\xi_1+t \xi_2) \leq (1-t) A(x,\xi_1) + t A(x,\xi_2) \leq \max\{A(x,\xi_1),A(x,\xi_2)\}$$
for any $x\in \Omega,\, \xi_1,\,\xi_2 \in \mathbb{R}^N$ and $t \in [0,1]$. Therefore 
\begin{equation}\label{**}
    N_1(x,(1-t)\xi_1+t \xi_2) \leq \max\{N_1(x,\xi_1),N_1(x,\xi_2)\}
\end{equation}
and this inequality is always strict unless $\xi_1 = \lambda \xi_2$, for some $\lambda >0.$\\
Now we prove that $N_1$ is subadditive. \\
Without loss of generality, we can assume that $ \xi_1 \neq 0$ and  $ \xi_2 \neq 0 $. Then we have $N_1(x,\xi_1) >0$ and $N_1(x,\xi_2) >0$.
Therefore, from \eqref{**} and $1$-homogeneity of $N_1(x, \xi)$ we obtain for any $t \in (0,1)$:
$$N_1\bigg(x,(1-t)\dfrac{\xi_1}{N_1(x,\xi_1)} + t \dfrac{\xi_2}{N_1(x,\xi_2)}\bigg) \leq 1. $$
We now fix $t$ such that  $$\dfrac{1-t}{N_1(x, \xi_1)}=\dfrac{t}{N_1(x, \xi_2)}\ \  {\it i.e.} \ \ t=\dfrac{N_1(x, \xi_2)}{N_1(x, \xi_1)+ N_1(x, \xi_2)}  \leq 1.$$
Then we get
$$N_1\bigg(x,\dfrac{\xi_1+\xi_2}{N_1(x,\xi_1)+N_1(x, \xi_2)}\bigg) \leq 1$$
and by $1$-homogeneity of $N_1$, we obtain  $$N_1(x, \xi_1+\xi_2) \leq N_1(x,\xi_1)+ N_1(x,\xi_2),\mbox{ {\it i.e.} $N_1$ is subadditive.}$$
Finally for $t \in (0,1),\ \xi_1 \neq \lambda \xi_2, \ \forall \lambda>0$ 
$$N_1(x, (1-t)\xi_1+ t \xi_2) < N_1(x,(1-t)\xi_1)+ N_1(x,t \xi_2)= (1-t) N_1(x,\xi_1)+ t N_1(x,\xi_2).$$
 This proves that $ \xi \to N_1(x, \xi)$ is ray-strictly convex.
 Now consider the case $r >1$. Since for any $x \in \Omega,\ \xi \to N_r^{1/r}(x,\xi)=N_1(x,\xi)$  is ray-strictly convex and thanks to the strict convexity of $t\to t^r$ on $\R^+$, we deduce that $\xi \to N_r(x, \xi)=N_1^r(x,\xi)$ is strictly convex when $r>1$.
%Precisely,
 %$$ N_r(x, (1-t)\xi_1+ t \xi_2) < (1-t) N_r(x,\xi_1)+ t N_r(x, \xi_2).$$
%If, for some $\lambda>0$, $\xi_1 =\lambda \xi_2$
 %\begin{align*}
  %   N_r(x, (1-t)\xi_1+ t \xi_2) &\leq \left((1-t) N_r^{1/r}(x,\xi_1)+ t\lambda N_r^{1/r}(x,\xi_1)\right)^{r}\\
   %  &= ((1-t)+t \lambda)^r N_r(x,\xi_1)\\
    % &< ((1-t)+t \lambda^r) N_r(x,\xi_1)= (1-t)N_r(x,\xi_1)+ t N_r(x,\xi_2).
 %\end{align*}
%This concludes the proof.
\end{proof}
From Proposition \ref{1} and from the $r$-homogeneity of $N_r$, we easily deduce the following convexity property of the energy functional:
\begin{pro}\label{1bis}
Assume the hypothesis in Proposition \ref{1}. Then, for $1 \leq r < p_-$
\begin{equation*}
\dot{V}_+^r \cap L^{\infty}(\Omega) \ni v\to   \int_\Omega A(x,\nabla( v^{1/r}))\,dx
\end{equation*}
is ray-strictly convex (if $r>1$, it is even strictly convex).
\end{pro}
\begin{proof}
We know that  $ \xi \to N_r(x,\xi) = A^{r/p(x)}(x,\xi)$ is  $r$-positively homogeneous and strictly convex if $r>1$ and for $r=1$ this function is ray-strictly convex. For $v_1, v_2 \in \dot{V}_+^r$ and $\theta \in (0,1)$ define $v=(1-\theta) v_1+ \theta v_2$   
 and we get 
 $$N_r\bigg(x,\dfrac{\nabla v}{v}\bigg) \leq (1-\theta) \dfrac{v_1}{v} N_r\bigg(x,\dfrac{\nabla v_1}{v_1}\bigg) + \theta \dfrac{v_2}{v} N_r\bigg(x,\dfrac{\nabla v_2}{v_2}\bigg).$$
  By homogeneity,
  $$N_r(x, \nabla (v^{1/r}) ) \leq (1-\theta) N_r(x,\nabla (v_1^{1/r})) + \theta N_r(x,\nabla (v_2^{1/r})) $$ and equality holds if and only if $v_1=\lambda v_2$ for some $\lambda>0$. Using the convexity of $t \to t^{p(x)/r}$  for $1 \leq r< p_-$
   we obtain 
   $$\int_{\Omega} A(x, \nabla v^{1/r}) ~dx \leq (1-\theta) \int_{\Omega} A(x, \nabla v_1^{1/r})  ~dx + \theta \int_{\Omega} A(x, \nabla v_2^{1/r}) ~dx. $$ 
   Moreover, if $p(x)\neq r$ equality holds if and only if $v_1=v_2$.
\end{proof}
\noindent From Proposition \ref{1}, we deduce the proof of Picone identity.\\
\textbf{Proof of Theorem \ref{picone}:}
Firstly, we deal with the case $r >1.$
Then from Proposition \ref{1}, for any $x\in \Omega$ the function $\xi \to N_r(x,\xi)=A(x,\xi)^{r/p(x)}$ is strictly convex.
 Let $\xi , \xi_0 \in \mathbb{R}^N \backslash \{0\}$ such that $\xi \neq \xi_0 $ then 
 $$N_r(x,\xi)-N_r(x,\xi_0) > \langle \partial_\xi N_r(x,\xi_0),\xi-\xi_0\rangle.$$
 Setting $\tilde a(x,\xi)=\dfrac{1}{r} \partial_\xi N_r(x,\xi)$, we obtain:
 $$N_r(x,\xi)-\langle \tilde a(x,\xi_0),\xi_0 \rangle >  r \langle \tilde a(x,\xi_0),\xi-\xi_0 \rangle.$$
 %% Euler Homogeneous Function Theorem
 Let $v, v_0 >0$ and replacing $\xi, \xi_0$ by $\xi/v$ and $\xi_0/ v_0$ respectively in the above expression, we get
$$N_r\bigg(x,\dfrac{\xi}{v}\bigg) > r \bigg< \tilde a\bigg(x, \dfrac{\xi_0}{v_0}\bigg), \dfrac{\xi}{v}-\dfrac{r-1}{r} \dfrac{\xi_0}{v_0}\bigg>.$$
Taking $\xi =\nabla v $ and $\xi_0= \nabla v_0$ and using $(r-1)$-homogeneity of $\tilde a(x,.)$,
 $$N\bigg(x,\dfrac{\nabla v}{r v^{(r-1)/r}}\bigg) > \dfrac{1}{v_0^{(r-1)/r}} \bigg\langle\tilde a\bigg(x,\dfrac{\nabla v_0}{r v_0^{(r-1)/r}}\bigg), \nabla v- \dfrac{r-1}{r} \dfrac{\nabla v_0}{v_0}v\bigg\rangle$$ 
 where the inequality is strict unless $\dfrac{\nabla v}{v}= \dfrac{\nabla v_0}{v_0}.$\\
Since $v^{1/r},\, v_0^{1/r}\in W^{1,p(x)}(\Omega)\cap L^\infty(\Omega)$, we can write 
$$\nabla(v^{1/r}) = \dfrac{\nabla v}{ r v^{(r-1)/r}}\ \ \ \  \text{and}\ \ \ \  \nabla \bigg( \dfrac{v}{v_0^{(r-1)/r}}\bigg) =\dfrac{1}{v_0^{(r-1)/r}} \bigg( \nabla v - \dfrac{r-1}{r} \dfrac{\nabla v_0}{v_0} v \bigg)$$
and we obtain
\begin{equation}\label{newq}
    N(x,\nabla v^{1/r}) > \bigg\langle \tilde a(x,\nabla v_0^{1/r}), \nabla\bigg(\dfrac{v}{v_0^{(r-1)/r}}\bigg)\bigg\rangle.
\end{equation}
 We have  
\begin{align*}
     \tilde a(x, \nabla v_0^{1/r}) &=\dfrac{1}{r} \partial_\xi N(x,\nabla v_0^{1/r}) =\dfrac{1}{r} \partial_\xi A^{r/p(x)}(x,\nabla v_0^{1/r})\\
     &=\dfrac{1}{p(x)} \partial_\xi A(x,\nabla v_0^{1/r}) A^{\frac{r-p(x)}{p(x)}}(x,\nabla v_0^{1/r})
\end{align*}
 and by replacing in \eqref{newq} we obtain
 $$A^{\frac{r}{p(x)}}(x,\nabla v^{1/r})\  A^{\frac{p(x)-r}{p(x)}}(x,\nabla v_0^{1/r}) 
 > \dfrac{1}{p(x)} \bigg\langle \partial_\xi A(x,\nabla v_0^{1/r}),\nabla \bigg(\dfrac{v}{v_0^{(r-1)/r}}\bigg)\bigg\rangle.$$
Now we deal with the case $r=1$. Let $\xi , \xi_0 \in \mathbb{R}^N \backslash \{0\}$ such that for any $\lambda>0$, $\xi \neq \lambda \xi_0$.  Then, from Proposition \ref{1}, we have that 
 $$N(x,\xi)-N(x,\xi_0) \geq \langle \partial_\xi N(x,\xi_0),\xi-\xi_0\rangle. $$
 Taking $\xi= \nabla v$ and $\xi_0=\nabla v_0$, we deduce
 $$N(x,\nabla v)-N(x,\nabla v_0) \geq \langle \partial_\xi N(x,\nabla v_0),\nabla(v-v_0)\rangle $$
 and 
 $$A^{\frac{1}{p(x)}}(x,\nabla v) A^{\frac{p(x)-1}{p(x)}}(x,\nabla v_0) 
 \geq \frac{1}{p(x)} \langle \partial_\xi A(x,\nabla v_0),\nabla v  \rangle$$
 for any $x \in \Omega$ and the inequality is strict unless $v= \lambda v_0$ for some $\lambda>0$.
 \qed

\noindent The Picone identity also holds for anisotropic operators of the following type: $$\displaystyle\sum_{i=1}^N \nabla_i(b_i(x,\nabla_i u)) = \sum_{i=1}^N \frac{\partial}{\partial x_i}\left(b_i\left(x, \frac{\partial u}{\partial x_i}\right)\right).$$
Precisely we have:
\begin{cor}\label{piciso}
Let $B: \Omega \times \mathbb{R} \to \mathbb{R}^N$ is a continuous and differentiable function such that $B(x,s)= (B_i(x,s))_{i=1,2,\dots N} $ satisfying for any $i$, for any $x\in \Omega$, the map $s \to B_i(x,s)$ is $p_i(x)$-homogeneous and strictly convex with $1<p_i^-\leq p_i(\cdot)\leq p_i^+<\infty$. For any $i$, we define $b_i(x, s)= \frac{1}{p_i(x)} \partial_s B_i (x,s)$. Then, for $v, v_0 \in \dot{V}_+^r \cap L^{\infty}(\Omega)$, we have
$$\displaystyle\sum_{i=1}^N b_i(x, \partial_{x_i} (v_0^{1/r}))\partial_{x_i} \left(\frac{v}{v_0^{\frac{r-1}{r}}}\right) \leq \sum_{i=1}^N B_i^{\frac{r}{p_i(x)}}\left(x, \partial_{x_i} (v^{1/r})\right) B_i^{\frac{p_i(x)-r}{p_i(x)}}\left(x, \partial_{x_i} (v_0^{1/r})\right) .$$
\end{cor}
  \begin{proof}
    By taking $A(x,s)= B_i(x, s)$ in Theorem \ref{picone}, we obtain $\forall i \in \{1,2,\dots,N\}$
   $$ \dfrac{1}{p_i(x)}  \partial_s B_i(x, \partial_{x_i} (v_0^{1/r})). \partial_{x_i} \left(\dfrac{v}{v_0^{\frac{r-1}{r}}}\right) \leq B_i^{\frac{r}{p_i(x)}}(x, \partial_{x_i} (v^{1/r})). B_i^{\frac{p_i(x)-r}{p_i(x)}}(x, \partial_{x_i} (v_0^{1/r})) $$
   for all $v, v_0 \in \dot{V}_+^r \cap L^{\infty}(\Omega)$ and $i=1,2, \dots, N$.\\
   Then by summing the expression over $i=1,2,\dots, N$, we obtain
   $$\sum_{i=1}^N b_i (x, \partial_{x_i}(v_0^{1/r})).\partial_{x_i} \left(\dfrac{v}{v_0^{\frac{r-1}{r}}}\right) \leq \sum_{i=1}^N B_i^{\frac{r}{p_i(x)}}(x,\partial_{x_i} (v^{1/r})). B_i^{\frac{p_i(x)-r}{p_i(x)}}(x, \partial_{x_i} (v_0^{1/r})) .$$
\end{proof}
%%%%%%%%%%%%%%%%%%%%%%   Proof of Diaz-Saa Inequality %%%%%%%%%%%%%%%%%%%%%%%%%%%%%%%%%%%%%%%%%%%%

\subsection{An extension of the Diaz-Saa inequality }
We prove the first application of Picone identity.\\
\textbf{Proof of Theorem \ref{Diaz-Saa}:}
The Picone identity implies
 $$ A^{r/p(x)}(x,\nabla w_1) A^{(p(x)-r)/p(x)}(x,\nabla w_2)  \geq a(x,\nabla w_2) .\nabla \bigg( \dfrac{w_1^r}{w_2^{r-1}}\bigg).$$
Using the  Young inequality for $r \in [1,p_-]$, we get
 \begin{equation*}
     \dfrac{r}{p(x)}(A(x,\nabla w_1)- A(x,\nabla w_2))+ A(x,\nabla w_2) \geq  a(x,\nabla w_2). \nabla \bigg( \dfrac{w_1^r}{w_2^{r-1}} \bigg). 
 \end{equation*}
 Noting that for any $\xi \in  \mathbb{R}^N, A(x, \xi)= a(x, \xi). \xi,$ we deduce
 \begin{equation}\label{Young1}
      a(x,\nabla w_2) .\nabla \bigg( w_2 - \dfrac{w_1^r}{w_2^{r-1}}\bigg) ~dx \geq \dfrac{r}{p(x)} (A(x,\nabla w_2) -  A(x,\nabla w_1)).
 \end{equation}
  Commuting $w_1$ and $w_2$, we have 
  \begin{equation}\label{Young2}
       a(x,\nabla w_1). \nabla \bigg( w_1 - \dfrac{w_2^r}{w_1^{r-1}}\bigg) \geq  \dfrac{r}{p(x)}  (A(x,\nabla w_1)- A(x,\nabla w_2)).  
  \end{equation}
Summing \eqref{Young1} and \eqref{Young2} and integrating over $\Omega$ yield 
$$\int_{\Omega}  a(x,\nabla w_1) .\nabla \bigg(\dfrac{w_1^r- w_2^r}{w_1^{r-1}}\bigg) ~dx + \int_{\Omega} a(x,\nabla w_2). \nabla \bigg(\dfrac{w_2^r- w_1^r}{w_2^{r-1}}\bigg) \geq 0.$$ The rest of the proof is the consequence of Proposition \ref{1bis}.
\hfill\qed \\
  
Diaz-Saa inequality also holds for anisotropic operators. Here we require that $\xi \to B_i(x,\xi)$ is $p_i(x)$-homogeneous and strictly convex and $b_i(x, \xi)= \dfrac{1}{p_i(x)} \partial_i B_i (x,\xi)$ where  $r\in \mathbb{R}, 1 \leq r \leq \min_{i =1,2,\dots, N} \{{(p_i)}_-\}.$
  \begin{cor}
Under the assumptions of Corollary \ref{piciso} and in addition that there exist $\Lambda >0$ such that for each $i$, $\left|\dfrac{\partial b_i}{\partial s}(x,s)\right| \leq \Lambda |s|^{p(x)-2}.$
Then we have in the sense of distributions, for
 $r\in [1, \min_{i} \{{(p_i)}_-\}]$ and $v, v_0 \in \dot{V}_+^r \cap L^{\infty}(\Omega)$:
  $$\sum_{i=1}^N \int_\Omega \bigg(-\dfrac{\partial_{x_i}(b_i(x,\partial_{x_i} v))}{v^{r-1}(x)} + \dfrac{\partial_{x_i}(b_i(x,\partial_{x_i}v_0))}{v_0^{r-1}(x)}\bigg) (v^r-v_0^r) ~dx \geq 0$$
\end{cor}
\begin{proof} 
  We apply Theorem \ref{Diaz-Saa}. For $A=B_i : \Omega \times \mathbb{R} \to \mathbb{R}$ and by replacing $\nabla$ by $\partial_{x_i}.$
  \end{proof}
 %%\begin{remark}
%% Contrary to the Picone identity, the Diaz-Saa Inequality is not pointwise, then this is weaker.
 %%\end{remark}
%%%%%%%%%%%%%%%%%%%%%%%%%%%%%%%%%%%%%%%%%%%%%%%   Application 2 %%%%%%%%%%%%%%%%%%%%%%%%%%%%%%%%%%%%%%%%%%%%%%
\section{Application of Picone identity to quasilinear elliptic equations}
The aim of this section is to establish Theorem \ref{Application 2}.
\subsection{Preliminary results}
The first lemma is the Picone identity in the context of the  $p(x)$-Laplacian operator.
 \begin{Lem}\label{implem}
 Let $r \in [1,p_-]$ and $u,v \in W_0^{1,p(x)}(\Omega) \cap L^\infty(\Omega)$ two positive functions. Then for any $x\in \Omega$
\begin{equation*}
|\nabla u|^{p(x)} + |\nabla v|^{p(x)}  \geq |\nabla v|^{p(x)-2} \nabla v. \nabla\bigg(\dfrac{u^r}{v^{r-1}}\bigg)+ |\nabla u|^{p(x)-2} \nabla u. \nabla\bigg(\dfrac{v^r}{u^{r-1}}\bigg). 
\end{equation*}
 \end{Lem}
 %\begin{proof} Using Young inequality in \eqref{piconep(x)}, we obtain
 %\begin{equation*}
%\begin{split}
 %  |\nabla u|^{p(x)} + |\nabla v|^{p(x)} &=|\nabla u|^{p(x)} \frac{p(x)-r}{p(x)}+|\nabla v|^{p(x)}\frac{r}{p(x)} \\ 
    % &\ \ \ \ \ \ \ \ \ \ \ +|\nabla u|^{p(x)} \frac{r}{p(x)}  +|\nabla v|^{p(x)} \frac{p(x)-r}{p(x)} \\
    % & \geq |\nabla u|^{p(x)-r} |\nabla v|^r + |\nabla v|^{p(x)-r} |\nabla u|^r \\
   %  & \geq |\nabla v|^{p(x)-2} \nabla v. \nabla\left(\frac{u^r}{v^{r-1}}\right) + |\nabla u|^{p(x)-2} \nabla u. \nabla\left(\frac{v^r}{u^{r-1}}\right).
	%	\end{split}
	%	 \end{equation*} 
%\end{proof}
\noindent Following the proof of Theorem 1.1 in \cite{*8}, we first prove the following comparison principle:
\begin{Lem}\label{CP}
Let $\lambda \geq 0$ and $u, v \in W_0^{1,p(x)}(\Omega)\cap L^{\alpha(x)}(\Omega)$ two nonnegative functions for some function $\alpha \in \mathcal{P}(\Omega)$ satisfying $1<\alpha_-\leq\alpha_+<\infty$. Assume for any $\phi \in W_0^{1,p(x)}(\Omega)$, $\phi\geq 0$:
$$\int_{\Omega} |\nabla u|^{p(x)-2} \nabla u. \nabla \phi + u^{\alpha(x)-1} \chi_{u \geq \lambda} \phi ~dx \geq\int_{\Omega} |\nabla v|^{p(x)-2} \nabla v. \nabla \phi + v^{\alpha(x)-1} \chi_{v \geq \lambda} \phi ~dx$$
where
\begin{equation*}
     \chi_{v\geq \lambda}(x) =
     \left\{
     \begin{alignedat}{2}
         & 1
          && \quad\mbox{ if }\, \lambda \leq v < \infty  \,;
          \\ 
          & 0
          && \quad\mbox{ if }\, 0 \leq v < \lambda \,,
          \end{alignedat}
    \right.
\end{equation*}
and $u \geq v$ a.e. in $\partial\Omega.$ Then $u \geq v$ a.e. in $\Omega.$
\end{Lem}
\begin{proof}
Let $\phi =(v-u)^+\in W^{1,p(x)}_0(\Omega)$ and $\Omega_1 = \{x \in \Omega : u(x) < v(x) \}.$ Then 
\begin{align*}
    0 \leq - \int_{\Omega_1} (|\nabla u|^{p(x)-2} \nabla u &- |\nabla v|^{p(x)-2} \nabla v) .\nabla (u -v) ~dx\\
    & -\int_{\Omega_1}(u^{\alpha(x)-1} \chi_{u \geq \lambda} - v^{\alpha(x)-1} \chi_{v \geq \lambda}) (u-v )~dx \leq 0   
\end{align*}
from which we obtain $u \geq v$ a.e. in  $\Omega.$
\end{proof}
\noindent Using lemma \ref{CP}, we show the following strong maximum principle:
\begin{Lem}\label{SMPl}
Let $h,\ l \in L^{\infty}(\Omega)$ be nonnegative functions, $h>0$ and $k:\Omega\times \R^+\to \R^+$. Let $\alpha, \beta \in \mathcal{P}(\Omega)$ be two functions such that $1<\beta_-\leq\beta_+<\alpha_-\leq\alpha_+<\infty$. Let $u \in C^1(\overline\Omega)$ be nonnegative and a nontrivial solution to
\begin{equation}\label{61}
     \left\{
         \begin{alignedat}{2} 
             {} -\Delta_{p(x)} u  +l(x) u^{\alpha(x)-1} 
             & {}= h(x) u^{\beta(x)-1} + k(x,u)
             && \quad\mbox{ in }\, \Omega \,;
             \\
             u & {}= 0
             && \quad\mbox{ on }\, \partial\Omega\,.
          \end{alignedat}
     \right.
\end{equation}
Assume in addition either\\
\textbf{(c1)} $\dfrac{l}{h} \in L^\infty(\Omega)$ \\ or \\
\textbf{(c2)} $k:\Omega\times \R^+\to \R^+$ satisfying $\displaystyle\liminf_{t \to 0^+} \frac{k(x,t)} {t^{\alpha(x)-1}} > \|l\|_{L^{\infty}}$ uniformly in $x$.\\ 
 Then $u$ is positive in $\Omega$.
\end{Lem}
\begin{proof} We follow the idea of the proof of Theorem 1.1 in \cite{*8}. For the reader's convenience we have included the detailed proof. We rewrite our equation \eqref{61}
under condition {\it \textbf{(c1)}} as follows:
$$-\Delta_{p(x)} u+ l(x)u^{\alpha(x)-1} \chi_{u \geq \lambda} \geq h(x) u^{\beta(x)-1} (1- \chi_{u \geq \lambda})\left(1-\frac{l(x)}{h(x)} u^{\alpha(x)-\beta(x)}\right),$$
since $\frac{l}{h} \in L^{\infty}(\Omega)$, we choose $\lambda \in (0,1)$ small enough such that for any $u(x) \leq \lambda$, we have $1- \frac{l(x)}{h(x)} u^{\alpha(x)-\beta(x)} \geq 1- \|\frac{l}{h}\|_{L^{\infty}(\Omega)} \lambda^{\alpha_- - \beta_+} \geq 0$.\\
Assuming condition {\it \textbf{(c2)}}, we have
$$-\Delta_{p(x)} u+ l(x)u^{\alpha(x)-1} \chi_{u \geq \lambda} \geq k(x,u) - (1- \chi_{u \geq \lambda})l(x) u^{\alpha(x)-1}.
$$
We choose $\lambda$ small enough such that for any $u(x)\leq \lambda$, we have $k(x,u)- l(x) u^{\alpha(x)-1} \geq 0.$ Hence under both conditions, we get for any $x\in \Omega$,
$$-\Delta_{p(x)} u+ l(x)u^{\alpha(x)-1} \chi_{u \geq \lambda} \geq 0.$$
Suppose that there exists $x_1$ such that $u(x_1)=0$ then using the fact that $u$ is nontrivial, we can find $ x_2 \in \Omega$ and a ball $B(x_2, 2C)$ in  $\Omega$ such that $x_1 \in \partial B(x_2, 2C)$ and $u > 0$ in $B(x_2, 2C).$ \\
Let $a = \inf\{ u(x) : |x-x_2|= C\}$ then $a >0 $ and choosing $x_2$ close enough to $x_1$ such that  $0< a< \lambda$ and $\nabla u(x_1)=0$ since $u(x_1)=0.$ \\
Denote the annulus $P = \{x \in \Omega : C < |x_1-x_2| < 2C \}$. We define $p_1=p(x_1),\ M= \sup\{|\nabla p(x)|: x \in P\}, \ \ b = 8M+2, \ \ l_1= -b \ \ln\left(\frac{a}{C}\right)+ \frac{2(N-1)}{C}$ and
 $$j(t)= \dfrac{a}{e^{\frac{l_1 C}{p_1-1}} -1} \bigg(e^{\frac{l_1 t}{p_1-1}} -1\bigg) \ \ \ \ \forall \  t \in [0,C].$$
We have
 $$ \frac{a}{C} e^{\frac{-l_1 C}{p_1-1}} < j'(0) \leq j'(t) \leq j'(C) <  \frac{a}{C} e^{\frac{l_1 C}{p_1-1}} $$
and then  
\begin{align}\label{65}
    \bigg(\dfrac{a}{C}\bigg)^3 \leq j'(t) \leq 1 \ \ \forall \ t \in [0,C].
 \end{align}
\noindent We choose $C < 1$ and using $\nabla u(x_1)=0$, $\frac{a}{C} < 1$ small enough such that for any $x \in P $
 \begin{align}\label{66}
     \dfrac{p(x)-1}{p_1-1} \geq \dfrac{1}{2}.
		\end{align}
Without loss of generality we can take $x_2=0$ and we set $r= |x-x_2|= |x|$, $t= 2C-r$. For $t \in [0,C]$ and $r\in [C, 2C]$, denote $w(r)= j(2C-r)= j(t)$, then  $$w'(r)= -j'(t), \ \ w''(t)= j''(t).$$
From \eqref{65} and \eqref{66}, we obtain
 \begin{align*}
     \div (|\nabla w|^{p(x)-2} \nabla w)&= (p(x)-1))(j'(t))^{p(x)-2} j''(t) -  \dfrac{N-1}{r} (j'(t))^{p(x)-1} \\
     &-(j'(t))^{p(x)-1} \ln (j'(t)) \sum_{i=1}^n \dfrac{\partial p}{\partial x_i} . \dfrac{x_i}{r}\\
     & \geq (j'(t))^{p(x)-1}\bigg(\dfrac{1}{2} l_1 + M \ln(j'(t))- \dfrac{N-1}{r}\bigg)\\
     & \geq  - \ln\bigg(\dfrac{a}{C}\bigg) (j'(t))^{p(x)-1} \geq 0.
\end{align*}
Since $ j(t) < a< \lambda $, we deduce
$$-\div (|\nabla w|^{p(x)-2} \nabla w) + w^{\alpha(x)-1} \chi_{w \geq \lambda} \leq 0.$$
On $\partial P$, $w(C)= j(C)= a \leq u(x) \ \ \text{and} \ \ w(2C)=j(0)= 0 \leq u(x).$
Then by Lemma \ref{CP}, we obtain $w \leq u$ on $P$. Finally, 
\begin{equation*}
\begin{split}
\lim_{s \to 0^+} \dfrac{u(x_1+s(x_2-x_1))-u(x_1)}{s} &\geq \lim_{s \to 0^+} \dfrac{w(x_1+s(x_2-x_1))-w(x_1)}{s}\\
&= j'(0) > 0
\end{split}
\end{equation*}
which contradicts $\nabla u(x_1)=0.$ Therefore, $u>0$ in $\Omega$.
\end{proof}
\begin{remark}
Conditions \textbf{(c1)} and \textbf{(c2)} can be replaced by the condition that there exists $t_0 $ such that $h(x)t^{\beta(x)-1} +k(x,t)-l(x) t^{\alpha(x)-1} \geq 0$ for all $0<t<t_0$ and $x\in\Omega$.
\end{remark}
\begin{Lem}\label{HMP}
Under the same conditions of $h,l,k$ as in Lemma \ref{SMPl}, let $u \in C^1(\overline{\Omega})$ be the nonnegative and nontrivial solution of \eqref{61}, $x_1 \in \partial \Omega $, $u(x_1)=0$ and  $\Omega$ satisfies the interior ball condition at $x_1$, then $\dfrac{\partial u}{\partial \vec n} (x_1)<0$ where $\vec n$ is the outward unit normal vector at $x_1.$
\end{Lem}
\begin{proof}
Choose $C>0$ small enough such that $B(x_2, 2C) \subset \Omega,\ x_1 \in \partial B(x_2, 2C)$. Then $x_2= x_1+ 2C\vec n$, where $\vec n$ is the outward normal at $x_1$. Denote $P=\{x \in \Omega : C< |x-x_2| <2C\} $ and by choosing $a$ such that $0<a< \lambda$, then by Lemma \ref{SMPl}, there exist a subsolution $w \in C^1(\overline{P}) \cap C^2(P)$ of \eqref{61} in $P$ and $w$ satisfies $w \leq u$ in $P$ with $w(x_1)=0, \frac{\partial w}{\partial  \vec n}(x_1) <0.$ Hence, we get $\frac{\partial u}{\partial  \vec n}(x_1) \leq \frac{\partial w}{\partial  \vec n}(x_1) <0.$ 
\end{proof}
\subsection{Proof of Theorem \ref{Application 2}}

\textbf{Proof of Theorem \ref{Application 2}:} We perform the proof along five steps. First we introduce notations. Define $F,\,G : {\Omega} \times \mathbb{R} \to \mathbb{R}^{+}$ as follows:
  \begin{equation*}
     F(x,t) =
     \left\{
     \begin{alignedat}{2}
         & \dfrac{h(x)}{q(x)} t^{q(x)}
          && \quad\mbox{ if }\, 0 \leq t < \infty  \,;
          \\ 
          & 0
          && \quad\mbox{ if }\, -\infty < t < 0 \,,
          \end{alignedat}
    \right.
\end{equation*}
and 
\begin{equation*}
     G(x,t) =
     \left\{
     \begin{alignedat}{12}
         & \dfrac{l(x)}{s(x)} t^{s(x)} 
          && \quad\mbox{ if }\, 0 \leq t < \infty  \,;
          \\ 
          & 0
          && \quad\mbox{ if }\, -\infty < t < 0 \,.
          \end{alignedat}
    \right.
\end{equation*}
We also extend the domain of $f$ and $g$ to all $ {\Omega} \times \mathbb{R} $ by setting $$f(x,t)= \dfrac{\partial F}{ \partial t}(x,t)=0 \ \text{and}\ g(x,t)= \dfrac{\partial G}{ \partial t}(x,t)=0\  \text{for} \  (x,t) \in {\Omega} \times (- \infty, 0).$$  %For fixed $x \in {\Omega}, F(x,\ . ), G(x,\ . ):  \mathbb{R} \to \mathbb{R}_{+} $ is a monotone increasing function.\\
Define the energy functional 
$ \mathcal{E} : W_0^{1,p(x)} (\Omega) \cap L^{s(x)}(\Omega) \to \mathbb{R}$ by
 \begin{equation}\label{13}
%\begin{split}
     \mathcal{E} (u) = \int_{\Omega} \dfrac{|\nabla u|^{p(x)}}{p(x)} ~dx + \int_{\Omega} G(x,u(x)) ~dx - \int_{\Omega} F(x,u(x)) ~dx.
 %    &=\int_{\Omega} \dfrac{|\nabla u|^{p(x)}}{p(x)} ~dx + \int_{\Omega^+} \dfrac{l(x)}{s(x)} u^{s(x)} ~dx - \int_{\Omega^+} \dfrac{h(x)}{q(x)} u^{q(x)} ~dx. 
%\end{split} 
\end{equation}
 %where $\Omega = \{u >0\}= \{x: u(x) \geq 0\}$\\
\noindent{\bf Step 1 :} Existence of a global minimizer\\
Since $W_0^{1,p(x)}(\Omega)\hookrightarrow L^{q(x)}(\Omega)$ (see Theorem $3.3.1$ and Theorem $8.2.4$ in \cite{*6}), the functional $\mathcal{E}$ is well-defined for every function $u \in W_0^{1,p(x)}(\Omega) \cap L^{s(x)}(\Omega).$\\
For $\|u\|_{W_0^{1,p(x)}}$ large enough: by \eqref{a} or \eqref{b}
 \begin{align*}
     \mathcal{E}(u) \geq \int_{\Omega} \dfrac{|\nabla u|^{p(x)}}{p(x)} - \int_{\Omega} \dfrac{h(x)}{q(x)} u^{q(x)} &\geq \dfrac{1}{p_-} \|\nabla u\|^{p_-}_{L^{p(x)}} - C\rho_q(u)\\
     & \geq \dfrac{1}{p_-} \|u\|^{p_-}_{W_0^{1,p(x)}} - C \|u\|_{W_0^{1,p(x)}}^{\tilde q}
 \end{align*}
where $\displaystyle\tilde q=\left\{\begin{array}{cc}
q_-&\mbox{ if }\|u\|_{L^{p(x)}}\leq 1\\
q_+&\mbox{ if }\|u\|_{L^{p(x)}} > 1
\end{array}\right.
$. Since $p_-> q_+$, this implies
$$ \mathcal{E}(u) \to \infty \ \ \text{as} \ \ \|u\|_{W_0^{1,p(x)}} \to +\infty.  $$
We argue similarly when  $\|u\|_{L^{s(x)}}\to \infty$ and we deduce $\mathcal E$ is coercive. The continuity of $\mathcal{E}$ on $W_0^{1,p(x)}(\Omega) \cap L^{s(x)}(\Omega)$ is given by Theorem 3.2.8 and 3.2.9 of \cite{*6}. Hence we get the existence of  at least  one global minimizer, say $u_0$, to \eqref{13}.\\
 
\noindent {\bf Step 2:}  Claim: $u_0\geq 0$ and $u_0\not\equiv 0$\\
Since $u_0$ is a global minimizer of $\mathcal{E}$ then $\mathcal{E}(u_0^+) \geq \mathcal{E}(u_0)$ where $u_0^+= \max\{u_0,0\} \in W_0^{1,p(x)}(\Omega)$. Set $\Omega^-=\{ x \in \Omega: u_0(x) < 0\}.$  We have
 \begin{align*}
     \mathcal{E} (u_0) &=\int_{\Omega} \dfrac{|\nabla u_0|^{p(x)}}{p(x)} ~dx + \int_{\Omega} G(x,u_0(x)) ~dx - \int_{\Omega} F(x,u_0(x)) ~dx\\
%     &=\int_{\Omega^+} \dfrac{|\nabla u_0|^{p(x)}}{p(x)} ~dx + \int_{\Omega\backslash \Omega^+} \dfrac{|\nabla u_0|^{p(x)}}{p(x)} ~dx + \int_{\Omega^+} \dfrac{l(x)}{s(x)} u_0^{s(x)} ~dx - \int_{\Omega^+} \dfrac{h(x)}{q(x)} u_0^{q(x)} ~dx.\\
     &= \mathcal{E}(u_0^+) + \int_{\Omega^-} \dfrac{|\nabla u_0|^{p(x)}}{p(x)} ~dx 
    % &\geq \mathcal{E}(u_0^+) %\geq \mathcal{E}(u_0)
 \end{align*}
which implies $\displaystyle \int_{\Omega^-} \frac{|\nabla u_0|^{p(x)}}{p(x)} = 0$ {\it i.e.} $\nabla u_0 (x) =0$ a.e. in $\Omega^-$ then by \eqref{a} \quad and \eqref{b} we have  $u_0=0$ a.e in $\Omega^-.$ This implies that $u_0 \geq 0$.\\
 In order to show that $u_0 \not\equiv 0 $ in $\Omega$, we construct a function $v$ in $W^{1,p(x)}_0(\Omega)\cap L^\infty(\Omega)$ such that $\mathcal{E}(v)<0 =\mathcal{E}(0).$ Precisely, consider $v=t \phi$ where $\phi \in C_c^1(\Omega)$, $\phi\geq 0$, $\phi \not\equiv 0 $ in $\Omega$ and for $0< t \leq 1$ small enough, we have
$$\mathcal{E}(v) \leq t^{q_+}(c_1t^{p_--q_+} +c_2t^{s_--q_+}-c_3) $$
where for any $i\in\{1,2,3\}$, $c_i$ are suitable constants independent of $t$. Hence, choosing $t$ small enough the right-hand side is negative and we conclude that $\mathcal{E}(t \phi) < 0 = \mathcal{E}(0)$ which implies $u_0 \not\equiv 0$. \\
\ \\
{\bf Step 3:} $u_0$ satisfies the equation in \eqref{10}\\
Since $u_0$ is a global minimizer and $\mathcal{E}$ is $C^1$ on $W^{1,p(x)}_0(\Omega)\cap L^{s(x)}(\Omega)$, then 
%for all $ \lambda \in \mathbb{R},\ \forall \phi \in W_0^{1,p(x)}(\Omega)\cap L^{s(x)}(\Omega)$, $\mathcal{E}(u_0+\lambda \phi) - \mathcal{E}(u_0) \geq 0$. So by using Taylor expension we have
%$$ \lambda \bigg( \int_{\Omega} |\nabla u_0|^{p(x)-2} \nabla u_0. \nabla \phi ~dx -\int_{\Omega} \dfrac{\partial F}{\partial t}(x, u_0) \phi ~dx +\int_{\Omega} \dfrac{\partial G}{\partial t}(x, u_0) \phi ~dx  + o(1)\bigg) \geq 0.$$
%Then divide by $\lambda\neq 0$ and taking $\lambda \to 0^+$ and $\lambda \to 0^-$ we conclude that $u_0$ is a weak solution of \eqref{10} \textit{i.e.}
for any $\phi\in W^{1,p(x)}_0(\Omega)\cap L^{s(x)}(\Omega)$, we have
$$\langle \mathcal{E}'(u_0),\phi\rangle=\int_{\Omega} |\nabla u_0|^{p(x)-2} \nabla u_0. \nabla \phi ~dx -\int_{\Omega} f(x,u_0) \phi ~dx +\int_{\Omega} g(x,u_0) \phi ~dx=0.$$
\noindent {\bf Step 4:} Regularity and positivity of weak solutions\\
First we prove that all nonnegative weak solutions of \eqref{10} belongs to $L^{\infty}(\Omega)$ which yields $C^{1,\alpha}(\overline\Omega)$ regularity.\\
 Let $K(x,t)= h(x) t^{q(x)-1} - l(x) t^{s(x)-1}$ and $\Lambda  \eqdef  \max\bigg\{\left\vert\left\vert\dfrac{h}{l}\right\vert\right\vert_{L^{\infty}},1 \bigg\}^{1/(s_--q_+)}$   \\
 Then it is not difficult to show that for any $t \geq \Lambda$, $K(x,t) \leq 0.$ 
%Indeed 
 %$$ t \geq \Lambda \geq  \left\vert\left\vert\dfrac{h}{l}\right\vert\right\vert_{L^{\infty}} ^{1/(s(x)-q(x))}$$
%thus
 %$$t^{s(x)-q(x)} \geq \dfrac{h(x)}{q(x)}\ \mbox{ which implies}\quad K(x,t) \leq 0.$$
Let $u$ be a nonnegative function satisfying weakly the equation in \eqref{10}. Then for any $\phi\in W^{1,p(x)}_0(\Omega)\cap L^{s(x)}(\Omega)$,
 $$ \int_{\Omega} |\nabla u|^{p(x)-2} \nabla u. \nabla \phi ~dx= \int_{\Omega} (h(x) u^{q(x)-1} -l(x) u^{s(x)-1}) \phi(x) ~dx.$$
 Taking the testing function $\phi(x)=(u-\Lambda)^+$, we get
$$\int_{\Omega} |\nabla (u-\Lambda)^+|^{p(x)} \leq 0. $$
By using \eqref{b}, we deduce $\|(u-\Lambda)^+\|_{W^{1,p(x)}_0}=0 $ which implies $u(x) \leq \Lambda$. \\
From Theorem 1.2 in \cite{*2}, we get $u \in C^{1, \alpha}(\overline\Omega)$ for some $\alpha\in (0,1)$.
%Applying Theorem \ref{C^1} \ with $A(x, u, \eta)= |\eta|^{p(x)-2} \eta $  and $B(x,u, \eta) = (h(x)u(x)^{q(x)-1} - l(x)u(x)^{s(x)-1})$ for all $(x, u,\eta) \in \Omega \times \mathbb{R} \times \mathbb{R}^N $ , we conclude $u \in C^{1, \alpha}(\overline\Omega)$. \\ 
\\ Furthermore assuming $x\to \frac{l(x)}{h(x)}$ belongs to $L^\infty(\Omega)$, Lemma \ref{SMPl} yields $u>0$ in $\Omega$.\\
\noindent{\bf Step 5:} Uniqueness of the positive solution of \eqref{10}\\
Let $u,\ v $ be two positive solutions of \eqref{10}. Thus for any $\phi,\ \tilde \phi \in W^{1,p(x)}_0(\Omega)\cap L^{s(x)}(\Omega)$,
 $$\int_{\Omega} |\nabla u|^{p(x)-2} \nabla u. \nabla \phi ~dx = \int_{\Omega} (h(x) u^{q(x)-1} -l(x) u^{s(x)-1}) \phi(x) ~dx $$
 and
 $$\int_{\Omega} |\nabla v|^{p(x)-2} \nabla v. \nabla \tilde{\phi} ~dx = \int_{\Omega} (h(x) v^{q(x)-1} -l(x) v^{s(x)-1}) \Tilde{\phi}(x) ~dx. $$
By the previous steps, $u$ and $v$ belong to $C^1(\overline{\Omega})$ and Lemma \ref{HMP} implies $ u,v \in C_d^0(\overline\Omega)^+ $. Hence taking the testing functions as $\phi= \dfrac{(u^{p_-}-v^{p_-})^+}{u^{p_--1}}$ and $\Tilde{\phi} = \dfrac{(v^{p_-}-u^{p_-})^-}{v^{p_--1}} \in W_0^{1,p(x)}(\Omega)$ (with the following notation $t^-\eqdef\max\{0,-t\}$) and from Lemma \ref{implem} we obtain
 \begin{align*}
     0 \leq \int_{\{u >v\}} (|\nabla u|^{p(x)-2} \nabla u &- |\nabla v|^{p(x)-2} \nabla v).\nabla (u-v) dx\\
     &= \int_{\{u >v\}} h(x) (u^{q(x)-p_-}-v^{q(x)-p_-}) (u^{p_-}-v^{p_-}) ~dx \\
     &\ \ \ \ \ \ + \int_{\{u>v\}} l(x) (v^{s(x)-p_-}-u^{s(x)-p_-}) (u^{p_-}-v^{p_-}) ~dx .
 \end{align*}
Since $q_+\leq p_-\leq s_-$,  the both terms in right-hand side are nonpositive. This implies $v(x) \geq u(x)$ a.e in $\Omega$.\\
Finally reversing the role of $u$ and $v$, we get $u=v$. \qed

\begin{remark}
Theorem \ref{Application 2} still holds when the condition $\dfrac{l}{h} \in L^{\infty}(\Omega)$ is replaced by $p_+ < s_-$ and using strong maximum principle in \cite{*8}.
\end{remark}

%%%%%%%%%%%%%%%%%%%%%%%%%%%%%%%%%%%%%%%%%%%%%%   Application 3                      %%%%%%%%%%%%%%%%%%%%%%%%%%%%%%
\section{Application to Fast diffusion equations }
In this section, we establish Theorems \ref{wea} and \ref{contr}. To this aim, we use a time semi-discretization method associated to \eqref{WeakS1}. With the help of accurate energy estimates about the related quasilinear elliptic equation and passing to the limit as the discretization parameter goes to $0$, we prove the existence and the properties of weak solutions to \eqref{P}. In the subsection below, we study the associated elliptic problem. 
\subsection{Study of the quasilinear elliptic problem associated to F.D.E.}\label{section1}
Consider the following problem
\begin{equation}\label{E_1}
     \left\{
         \begin{alignedat}{2} 
             {} v^{2q-1} -\lambda \Delta_{p(x)} v  
             & {}= h_0(x) v^{q-1} + \lambda f(x,v)
             && \quad\mbox{ in }\,  \Omega \,;
             \\
             v & {}> 0
             && \quad\mbox{ in }\,  \Omega \,;
             \\
             v & {}= 0
             && \quad\mbox{ on }\,  \partial\Omega \,.
          \end{alignedat}
     \right.
\end{equation}
Assume $h_0\in L^\infty(\Omega)^+$ and $f$ satisfies $(f_1)$-$(f_3)$. Then from $(f_3)$, we have $(f_0): \lim_{s  \to +\infty} \dfrac{f(x,s)}{s^{p_- -1}} =0$ uniformly in $x \in \Omega.$ Therefore,
for any $\epsilon >0,$ there exists a positive constant $C_{\epsilon}$ such that for any  $(x,s) \in \Omega \times \mathbb{R}^+$:
\begin{equation}\label{infty} 
0 \leq f(x,s) \leq C_{\epsilon} + \epsilon s^{p_- -1}. 
\end{equation}
%\begin{define}
%Let $X$ be a Banach space. An operator $S :\mathcal{D}(S) \subset X \to X $ is called accretive if for every $u, \tilde{u} \in \mathcal{D}(S)$ and $\lambda > 0$
%$$\|u-\tilde{u}\|_X \leq \|u-\tilde{u}+ \lambda (Su-S\tilde{u})\|_X $$  and S is called T-accretive if
%$$\|[u-\tilde{u}]^+\|_X \leq \|[u-\tilde{u}+ \lambda (Su-S\tilde{u})]^+\|_X.$$
%\end{define}
%%\begin{remark}
%%For monotone operators P on Hilbert Space $u,v \in X=H$ we have 
%%$$\langle Pu-Pv, u-v\rangle \geq 0$$ which is equivalent to saying that $(I+\lambda P)^{-1}$ is a contraction on H , i.e. for every $\lambda >0$
%%$$\|u-\tilde{u}\|_H \leq \|u-\tilde{u}+ \lambda (Pu-P\tilde{u})\|_H $$
%%\end{remark}
%%From the Diaz-Saa inequality, the operator $ P: X \to X^{'} $ defined by $ Pv= \Delta_{p(x)} (v^{1/q}) . v^{(1-q)/q}  $ where $ X=\dot{V}_+^q $ is the subdifferential of the strictly ray-convex functional $\hat{\mathcal{E}}
%: v \to \int_{\Omega} \dfrac{q}{p(x)} |\nabla v^{1/q}|^{p(x)} ~dx $ and $P$ is monotone operator.
We have the following preliminary result about \eqref{E_1}:
\begin{thm}\label{exis}
Let $\lambda>0$, $ q\in(1, p_-]$, $f:\Omega\times \R^+ \to \R^+$ satisfying $(f_0)$ and $(f_1)$ and $h_0 \in L^{\infty}(\Omega)^+$. Then there exists a weak solution $v \in C^1(\overline{\Omega})$ to \eqref{E_1}, {\it i.e.}  for any  $\phi\in {\mathbf{W}} \eqdef  W^{1,p(x)}_0(\Omega)\cap L^{2q}(\Omega)$
\begin{equation}\label{64}
    \int_{\Omega} v^{2q-1} \phi ~dx + \lambda \int_{\Omega} |\nabla v|^{p(x)-2} \nabla v. \nabla \phi ~dx = \int_{\Omega} h_0 v^{q-1} \phi ~dx + \lambda \int_{\Omega} f(x,v) \phi ~dx.
\end{equation}
In addition, if $(f_2)$ and $(f_3)$ hold then $v \in C_d^0(\overline \Omega)^+$. Moreover if $v_1,\ v_2 \in C_d^0(\overline \Omega)^+$ are two weak solutions to \eqref{E_1} corresponding to $h_0=h_1,\ h_2\in L^\infty(\Omega)^+$ respectively, then we have
\begin{equation}\label{Acr}
\|(v_1^q-v_2^q)^+\|_{L^2} \leq \|(h_1-h_2)^+\|_{L^{2}}.
\end{equation}\end{thm}
\begin{remark}
\eqref{Acr} implies the uniqueness of the weak solution to \eqref{E_1} in $C_d^0(\overline \Omega)^+$.
\end{remark}
\begin{proof}
%To establish Theorem \ref{exis}, we use similar arguments as the proof of Theorem \ref{Application 2}. 
We perform the proof into several steps.\\
{\bf Step 1:} Existence of a weak solution\\
Consider the energy functional $\mathcal{J}$ defined on ${\mathbf{W}} $ equipped with $\|.\|_{\mathbf{W}}=\|.\|_{W_0^{1,p(x)}} + \|.\|_{L^{2q}}$
 \begin{equation}\label{defJ}
\mathcal{J}(v) = \dfrac{1}{2q}\int_{\Omega}v^{2q} ~dx + \lambda \int_{\Omega} \dfrac{|\nabla v |^{p(x)}}{p(x)} ~dx - \dfrac{1}{q} \int_{\Omega} h_0 D(v) ~dx - \lambda \int_{\Omega} F(x,v) ~dx 
\end{equation} where
\begin{equation*}
     D(t) =
     \left\{
     \begin{alignedat}{2}
         & t^q
          && \quad\mbox{ if }\, 0 \leq t < \infty  \,;
          \\ 
          & 0
          && \quad\mbox{ if }\, -\infty < t < 0 \,,
          \end{alignedat}
    \right.
\mbox{and}\ 
     F(x,t) =
     \left\{
     \begin{alignedat}{2}
         & \int_0^t f(x,s) ds
          && \quad\mbox{ if }\, 0 \leq t< \infty  \,;
          \\ 
          & 0
          && \quad\mbox{ if }\, -\infty < t < 0 \,.
          \end{alignedat}
    \right.
\end{equation*}
We also extend the domain of $f$ to all of ${\Omega}\times \mathbb{R}$ by setting $f(x,t)= \dfrac{\partial F}{\partial t} (x,t) =0$ for $(x,t) \in  {\Omega}\times (-\infty, 0).$ From \eqref{infty}, H\"older inequality \eqref{hi} and since $W_0^{1,p(x)} \hookrightarrow L^{p_-}(\Omega)$, we obtain
\begin{align*}
      \mathcal{J}(v)&\geq \dfrac{1}{2q} \|v\|^{2q}_{L^{2q}} + \lambda \|v\|^{p_-}_{W_0^{1,p(x)}} - \dfrac{1}{q} \|h_0\|_{L^2} \|v\|_{L^{2q}}^{q} - \lambda C_{\epsilon} \int_{\Omega} v dx - \lambda \dfrac{\epsilon}{p_-} \int_{\Omega} v^{p_-} ~dx\\
      &\geq \dfrac{1}{q} \|v\|^{q}_{L^{2q}} \bigg(\dfrac{1}{2} \|v\|^{q}_{L^{2q}} - \|h_0\|_{L^2}\bigg) + \lambda \|v\|_{W_0^{1,p(x)}}((1- \epsilon) \|v\|^{p_- -1}_{W_0^{1,p(x)}} - \tilde{C}) .
  \end{align*}
  Then by choosing $\epsilon$ small enough we conclude the coercivity of $\mathcal J$ on ${\bf W}$ and  $\mathcal{J}$ is also continuous on $\mathbf{W}$ therefore we deduce the existence of a global minimizer $v_0$ to $\mathcal{J}$.\\
Furthermore we note
$$    \mathcal{J}(v_0) \geq \mathcal{J}(v_0^+) + \dfrac{1}{2q}\int_{\Omega}(v_0^-)^{2q} ~dx +  \lambda \int_{\Omega} \dfrac{|\nabla v_0^-|^{p(x)}}{p(x)} ~dx$$
  which implies $v_0\geq 0$. \\
Now we claim that $v_0 \not\equiv 0$ in $\Omega.$ Since $\mathcal{J}(0)=0,$ it is sufficient to prove the existence of $\tilde v\in \mathbf{W}$ such that $\mathcal{J}(\tilde v) < 0.$ For that take $\tilde v=t\phi$ where $\phi \in C_c^1(\Omega)$ is nonnegative function such that $ \phi \not\equiv 0$ and $t>0$ small enough.

 Since $v_0$ is a global minimizer for the differentiable functional $\mathcal{J}$, we have that $v_0$ satisfies \eqref{64}
{\it i.e.} $v_0$ is a weak solution to \eqref{E_1}. From Corollary \ref{reg3} we infer that $v_0 \in L^{\infty}(\Omega)$. Then by using Theorem \ref{C^1}, we obtain, $v_0 \in C^{1, \alpha}(\overline{\Omega})$ for some $\alpha \in (0,1)$.\\
From $(f_2)$ and Lemma \ref{SMPl} (with condition {\it\textbf{(c2)})}, we obtain $v_0 > 0$ and by Lemma \ref{HMP} we get $\dfrac{\partial v_0}{\partial  \vec n} < 0$ on $\partial \Omega$. Therefore, $v_0$ belongs to $C^0_d(\overline \Omega)^+$.\\
\textbf{Step 2:} Contraction property \eqref{Acr}\\
Let $v_1$ and $v_2$ two positive weak solutions of \eqref{E_1} such that $v_1, v_2 \in  C^0_d(\overline \Omega)^+$. For any $\phi,\ \Psi \in \mathbf{W}$:
$$ \int_{\Omega} v_1^{2q-1} \phi ~dx + \lambda \int_{\Omega} |\nabla v_1|^{p(x)-2} \nabla v_1. \nabla \phi ~dx = \int_{\Omega} h_1 v_1^{q-1} \phi ~dx + \lambda  \int_{\Omega} f(x,v_1) \phi ~dx$$
and
$$ \int_{\Omega} v_2^{2q-1} \Psi ~dx + \lambda \int_{\Omega} |\nabla v_2|^{p(x)-2} \nabla v_2. \nabla \Psi ~dx = \int_{\Omega} h_2 v_2^{q-1} \Psi ~dx + \lambda \int_{\Omega} f(x,v_2) \Psi ~dx.$$
Since $v_1, v_2\in C^0_d(\overline \Omega)^+$, $\phi= \left( v_1 - \frac{v_2^q}{v_1^{q-1}}\right)^+$ and $\Psi= \left( v_2 - \frac{v_1^q}{v_2^{q-1}}\right)^- $ are well-defined and belong to $ \mathbf{W}$. Subtracting the  two above expressions and using $(f_3)$ together with Lemma \ref{implem} we obtain
\begin{align*}
    \int_{\Omega} ((v_1^q- v_2^q)^+)^2 ~dx \leq \int_{\Omega} (h_1-h_2) (v_1^q - v_2^q)^+ ~dx.
\end{align*}
Finally, applying the  H\"older inequality we get \eqref{Acr}.
%\begin{equation}\label{accre}
   % \|(v_1^q-v_2^q)^+\|_{L^2} \leq \|(h_1-h_2)^+\|_{L^{2}}.
%\end{equation}
\end{proof}
\noindent From Theorem \ref{exis}, we deduce the accretivity of ${\mathcal R}$:
\begin{cor}\label{L^infty}
Let $\lambda>0$, $q \in (1,p_-], $ $f:\Omega\times \R^+ \to \R^+$ satisfying $(f_1)$-$(f_3)$ and $h_0 \in L^{\infty}(\Omega)^+$. Consider the following problem
\begin{equation}\label{E_12}
     \left\{
         \begin{alignedat}{2} 
             {} u +\lambda \mathcal R u
             & {}= h_0(x)  
             && \quad\mbox{ in }\,  \Omega;
             \\
u&>0 && \quad\mbox{ in }\,  \Omega;\\
             u & {}= 0
             && \quad\mbox{ on }\,  \partial\Omega .
          \end{alignedat}
     \right.
\end{equation}
Then there exists a unique distributional solution $u \in \mathcal{D}(\mathcal{R}) \cap C^1{(\overline\Omega)}$ of \eqref{E_12}\ {\it i.e.} $\forall \phi \in C_c^1(\Omega)$ 
\begin{align*}
\int_{\Omega} u_0 \phi ~dx+ \lambda \int_{\Omega} |\nabla u_0^{1/q}|^{p(x)-2} \nabla u_0^{1/q}&.\nabla\bigg(\dfrac{\phi}{u_0^{(q-1)/q}}\bigg)~dx\\
& = \int_{\Omega} h_0 \phi ~dx + \lambda \int_{\Omega} \dfrac{f(x, u_0^{1/q})}{u_0^{(q-1)/q}} \phi~dx.
\end{align*}
Moreover, if $u_1$ and $u_2$ are two distributional solutions of \eqref{E_12} in $\mathcal{D}(\mathcal{R}) \cap C^1{(\overline\Omega)}$  associated to $h_1$ and $h_2$ respectively, then the operator $\mathcal{R}$ satisfies
\begin{equation}\label{accreti}
    \|(u_1-u_2)^+\|_{L^2} \leq \|(u_1-u_2+\lambda (\mathcal{R}u_1- \mathcal{R}u_2))^+\|_{L^{2}}.
\end{equation}
\end{cor}
\begin{proof}
Define the energy functional $\mathcal{E}$ on $\dot V^{q}_+\cap L^2(\Omega)$ as $\mathcal{E}(u) =\mathcal{J}(u^{1/q})$ where $\mathcal{J}$ is defined in \eqref{defJ}.\\
Let $\phi \in C_c^{1}(\Omega)$ and $v_0$ is the global minimizer of \eqref{defJ} which is also the weak solution of \eqref{E_1} and $u_0 = v_0^q$ then there exists $ t_0=t_0(\phi) > 0$ such that for $t \in (-t_0, t_0)$, $u_0 +t\phi >0$.
Hence we have\begin{align*}
    0 \leq \mathcal{E}(u_0 +t \phi) - \mathcal{E}(u_0) &= \dfrac{1}{2q} \bigg(\int_{\Omega}(t \phi)^2~dx + \int_{\Omega} 2 t u_0 \phi~dx \bigg) - \frac{1}{q} \int_{\Omega} h t \phi~dx \\
    &+ \lambda  \bigg(\int_{\Omega} \dfrac{|\nabla (u_0+t \phi)^{1/q}|^{p(x)}}{p(x)}~dx- \int_{\Omega} \dfrac{|\nabla u_0^{1/q}|^{p(x)}}{p(x)}~dx \bigg)\\
    &- \lambda \bigg( \int_{\Omega} F(x,(u_0+t \phi)^{1/q}) ~dx - \int_{\Omega} F(x,u_0^{1/q})~dx \bigg).
\end{align*}
Then divide by $t$ and passing to the limits $t \to 0$ we obtain %$$ \int_{\Omega} u_0 \phi + \lambda \int_{\Omega} |\nabla u_0^{1/q}|^{p(x)-2} \nabla u_0^{1/q} \nabla\bigg(\dfrac{\phi}{u_0^{(q-1)/q}}\bigg) ~dx = \int_{\Omega} h \phi + \lambda \int_{\Omega} \dfrac{f(x, u_0^{1/q})}{u_0^{(q-1)/q}} \phi ~dx $$
$u_0 = v_0^q$ is the distributional solution of ($\ref{E_12}$).  Finally \eqref{accreti} and uniqueness follow from \eqref{Acr}. 
%, Let $(v_n)$ be the weak solution of \eqref{E_1} obtained from Corollary \ref{L^2} corresponding to $(h_n) \in L^{\infty}(\Omega)$ such that $(v_n)$ is uniformly bounded and increasing sequence in $L^{2q}(\Omega), (h_n)_i \to h_i $ in $L^2(\Omega)$ for $i= 1,2$  and 
%\begin{align}
 %   \|[(v_n^q)_1-(v_n^q)_2]^+\|_{L^2{(\Omega)}} \leq \|[(h_n)_1-(h_n)_2]^+\|_{L^{2}(\Omega)}
%\end{align}
%Since $\|(v_n^q)_i\|_{L^2(\Omega)}=\|(v_n)_i\|^q_{L^{2q}(\Omega)}$ therefore upto a subsequence $(v_n^q)_i \to v_i^q$ a.e. then passing to the limit in above inequality we obtain for $h_1, h_2 \in (L^2(\Omega))^+$,
%$$\|[v_1^q-v_2^q]^+\| =\|[u_1-u_2]^+\|_{L^2{(\Omega)}} \leq \|[h_1-h_2]^+\|_{L^{2}(\Omega)}$$
%i.e. $$\|[u_1-u_2]^+\|_{L^2{(\Omega)}} \leq \|[u_1-u_2 + \lambda (Ru_1-Ru_2)]^+\|_{L^{2}(\Omega)}.$$
\end{proof}
\noindent We now generalize some above results for a larger class of potentials $h_0$:
\subsection{Further results for \eqref{E_1} and uniqueness }
\begin{thm}\label{L^2}
Let $\lambda>0$,  $f:\Omega\times \R^+ \to \R^+$ satisfying $(f_1)$-$(f_3)$ and $h_0 \in L^{2}(\Omega)^+$ and $q \in (1,p_-]$. Then there exists a positive weak solution $v\in \mathbf{W}$ of \eqref{E_1} in the sense of \eqref{64}.  Moreover assuming that $h_0$ belongs to $L^\nu(\Omega)$ for some $\nu> \max\left\{1, \frac{N}{p_-}\right\}$, $v\in L^\infty(\Omega)$.
\end{thm}
\begin{proof}
Let $h_n\in C^1_c(\Omega)$ such that $h_n \geq 0$ and $h_n \to h $ in $L^2(\Omega)$. Define $(v_n)\subset C^{1, \alpha}(\overline{\Omega})\cap C^0_d(\overline \Omega)^+$ as for a fixed $n$, $v_n$ is the unique positive weak solution of \eqref{E_1} with $h_0=h_n$ {\it i.e.} $v_n$ satisfies: for $\phi \in \mathbf{W}$
\begin{equation}\label{ws}
\int_{\Omega} v_n^{2q-1} \phi ~dx + \lambda \int_{\Omega} |\nabla v_n|^{p(x)-2} \nabla v_n .\nabla \phi ~dx= \int_{\Omega} h_n v_n^{q-1} \phi ~dx + \lambda \int_{\Omega} f(x,v_n) \phi ~dx .
\end{equation}
Since $(a-b)^{2q}\leq (a^q-b^q)^2$ for any $q\geq 1$, \eqref{Acr} implies for any $n,\ p\in \N^*$
\begin{equation*}
\|(v_n-v_p)^+\|_{L^{2q}} \leq \|(h_n-h_p)^+\|^q_{L^{2}}
\end{equation*}
thus we deduce that $(v_n)$ converges to $v\in L^{2q}(\Omega)$.\\
We infer that the limit $v$ does not depend on the choice of the sequence $(h_n)$. Indeed, consider $\tilde h_n\neq h_n$ such that $\tilde h_n\to h_0$ in $L^2(\Omega)$ and $\tilde v_n$ the positive solution of \eqref{E_1} corresponding to $\tilde h_n$ which converges to $\tilde v$.\\
Then, for any $n\in \N$, \eqref{Acr} implies 
\begin{equation*}
\|(v_n^q-\tilde v_n^q)^+\|_{L^2} \leq \|(h_n-\tilde h_n)^+\|_{L^{2}}
\end{equation*}
and passing to the limit we get $\tilde v \geq v$ and then by reversing the role of $v$ and $\tilde{v}$ we obtain $v= \tilde{v}$.\\
So define, for any $n\in \mathbb{N}^*$, $h_n=\min\{h,n\}$. Thus $(v_n)$ is nondecreasing and for any $n\in \mathbb{N}^*$, $v_n\leq v$ {\it a.e.} in $\Omega$ which implies $v\geq v_1>0$ in $\Omega$.\\
We choose $\phi= v_n$ in \eqref{ws}. Applying the H\"older inequality and \eqref{infty}, we obtain 
\begin{equation*}
\begin{split}
\lambda \int_{\Omega} |\nabla v_n|^{p(x)}~dx&\leq  \|h_n\|_{L^2} \|v_n\|^q_{L^{2q}} + \lambda C_\epsilon\|v_n\|_{L^1} + \lambda \epsilon\|v_n\|_{L^{p_-}}^{p_-}\\
&\leq C + \lambda \epsilon\|v_n\|_{L^{p_-}}^{p_-}.
\end{split}
\end{equation*}
Assume $\|\nabla v_n\|_{L^{p(x)}} \geq 1$. Since $W_0^{1,p(x)}(\Omega) \hookrightarrow L^{p_-}(\Omega)$ and by \eqref{a} we deduce for some positive constant $C>0$:
$$ \lambda \int_{\Omega} |\nabla v_n|^{p(x)}~dx \leq C+ \lambda \epsilon C\int_{\Omega} |\nabla v_n|^{p(x)}~dx. $$
Choosing $\epsilon$ small enough and gathering with the case $\|\nabla v_n\|_{L^{p(x)}} \leq 1$, we conclude $(v_n)$ is uniformly bounded in ${W_0^{1,p(x)}(\Omega)}$ and $L^{p_-}(\Omega)$. Hence $v_n$ converges weakly to $v$ in $W^{1,p(x)}_0(\Omega)$ and by monotonicity of $(v_n)$ strongly in $L^{p_-}(\Omega)$ and in $L^{2q}(\Omega)$.
Taking now $\phi=v_n-v$ in \eqref{ws}, from \eqref{infty} with $\epsilon=1$ and by H\"older inequality
\begin{equation*}
\left|\int_\Omega f(x,v_n)(v_n-v)~dx\right|\leq C\|v_n-v\|_{L^{2q}}+\|v_n\|^{p_--1}_{L^{p_-}}\|v_n-v\|_{L^{p_-}} \to 0
\end{equation*}
and  
\begin{equation*}%\label{conv1}
\int_{\Omega} h_n v_n^{q-1} (v_n-v)~dx\to 0 \ \ \text{and}\ \ \int_{\Omega} v_n^{2q-1} (v_n-v)~dx \to 0. 
\end{equation*}
%Let $1<r<2$ 
%\begin{align*}
%    \int_{\Omega} (g_n v_n^{q-1})^r &\leq \|g_n^r\|_{L^{2/r}(\Omega)} \|v_n^{r(q-1)}\|_{L^{2/(2-r)}(\Omega)}\\
%    & \leq \|g_n\|^r_{L^2(\Omega)} \|v_n\|^{(2-r)/2}_{L^{2q}(\Omega)}.
%\end{align*}
%which implies $(g_n v_n^{q-1})$ is bounded in $L^r(\Omega) \ \forall \ 1<r<2.$ Here we can choose $r= \frac{2q}{2q-1}.$ Then we obtain
%$$\int_{\Omega} g_n v_n^{q-1} \phi \to \int_{\Omega} K \phi  \ \ \forall \phi \in L^{2q}(\Omega)$$ 
Finally \eqref{ws} becomes
$$\int_{\Omega} |\nabla v_n|^{p(x)-2} \nabla v_n .\nabla (v_n - v)~dx \to 0.$$
Then, since $v_n\rightharpoonup v$ in $W^{1,p(x)}_0(\Omega)$
$$\int_{\Omega} ( |\nabla v_n|^{p(x)-2} \nabla v_n -|\nabla v|^{p(x)-2} \nabla v).\nabla (v_n - v)~dx \to 0.$$
Lemma A.2 and Remark A.3 of \cite{GTW} give the strong convergence of $v_n$ to $v$ in $W_0^{1,p(x)}(\Omega)$.\\
Since $(v_n^{2q-1})$ and $(h_n v_n^{q-1})$ are uniformly bounded in $L^{2q/(2q-1)}(\Omega)$ and by \eqref{infty}, $f(x,v_n)$ is uniformly bounded in $L^{2q/q-1}(\Omega)$ and $f(x,v_n) \to f(x,v)$ a.e.\ in $\Omega.$ Then by Lebesgue dominated convergence theorem we have (up to a subsequence), for $\phi \in \mathbf{W}$
$$\int_{\Omega} v_n^{2q-1} \phi ~dx \to \int_{\Omega} v^{2q-1} \phi ~dx, \int_{\Omega} h_n v_n^{q-1} \phi ~dx \to \int_{\Omega} h v^{q-1} \phi ~dx$$
and
$$\int_{\Omega} f(x,v_n) \phi ~dx\to \int_{\Omega} f(x,v) \phi ~dx.$$
Finally we pass to the limit in \eqref{ws} and we obtain $v$ is a weak solution of \eqref{E_1}. 
To conclude corollary \ref{reg3} implies $v\in L^\infty(\Omega)$.
\end{proof}
\begin{remark}\label{rema}
Let $v_1,\ v_2$ are the weak solutions of \eqref{E_1} obtained by Theorem \ref{L^2} corresponding to $h_1,\ h_2\in L^2(\Omega)^+$, $h_1\not\equiv h_2$ respectively. Then
\begin{equation*}
\|(v_1^q-v_2^q)^+\|_{L^2} \leq \|(h_1-h_2)^+\|_{L^{2}}.
\end{equation*}
\end{remark}
\begin{remark}\label{othsol}
As in \textbf{Step 1} of the proof of Theorem \ref{exis}, we can alternatively prove the existence of a  weak solution by global minimization method.
\end{remark}
\noindent Under the hypothesis of Theorem \ref{L^2} and with the help of Picone identity, the following theorem gives the uniqueness of the solution to \eqref{E_1}.
\begin{thm}\label{uniq3}
Let $v, \tilde{v} $ be respectively a subsolution and supersolution to \eqref{E_1} for $h \in L^{p_0}(\Omega),\ p_0 \geq 2$, $h \geq 0$ and $f$ satisfies $(f_1)$ and $(f_3)$. Then $v\leq\tilde{v}.$
\end{thm}
\begin{proof}
We have for any nonnegative $\phi, \, \Psi \in \mathbf{W}$
$$ \int_{\Omega} v^{2q-1} \phi ~dx + \lambda \int_{\Omega} |\nabla v|^{p(x)-2} \nabla v.\nabla \phi ~dx \leq \int_{\Omega} h v^{q-1} \phi ~dx +\lambda \int_{\Omega} f(x,v) \phi ~dx$$
and
$$ \int_{\Omega} \tilde{v}^{2q-1} \Psi ~dx + \lambda \int_{\Omega} |\nabla \tilde{v}|^{p(x)-2} \nabla \tilde{v}.\nabla \Psi ~dx\geq\int_{\Omega} h \tilde{v}^{q-1} \Psi ~dx + \lambda \int_{\Omega} f(x,\tilde{v}) \Psi ~dx.$$
Subtracting the above inequalities with test functions $\phi = \bigg(\frac{(v+\epsilon)^q-(\tilde{v}+\epsilon)^q}{(v+\epsilon)^{q-1}}\bigg)^+$ and \\ $\Psi = \bigg(\frac{(\tilde{v}+\epsilon)^q- (v+\epsilon)^q }{(\tilde{v}+\epsilon)^{q-1}}\bigg)^- \in  \mathbf{W}$ for $\epsilon \in (0,1)$, we obtain
\begin{equation}\label{uni}
      \begin{aligned}
         \int_{\{v > \tilde{v}\}} &\bigg( \frac{v^{2q-1}}{(v+\epsilon)^{q-1}} - \frac{\tilde{v}^{2q-1}}{(\tilde{v}+ \epsilon)^{q-1}} \bigg) ((v+\epsilon)^q- (\tilde{v}+ \epsilon)^q) ~dx\\
         &+ \lambda \int_{\{v > \tilde{v}\}} |\nabla (v+\epsilon)|^{p(x)-2} \nabla (v+\epsilon). \nabla\bigg(\frac{(v+\epsilon)^q - (\tilde{v}+ \epsilon)^q}{(v+ \epsilon)^{q-1}}\bigg)~dx\\
          &+ \lambda \int_{\{v > \tilde{v}\}} |\nabla (\tilde{v}+\epsilon)|^{p(x)-2} \nabla (\tilde{v}+\epsilon). \nabla\bigg(\frac{(\tilde{v}+\epsilon)^q - (v+ \epsilon)^q}{(\tilde{v}+ \epsilon)^{q-1}}\bigg)~dx\\
          &\leq \int_{\{v > \tilde{v}\}} h \bigg( \frac{v^{q-1}}{(v+\epsilon)^{q-1}} - \frac{\tilde{v}^{q-1}}{(\tilde{v}+ \epsilon)^{q-1}} \bigg)((v+\epsilon)^q- (\tilde{v}+ \epsilon)^q) ~dx\\
          &+ \lambda \int_{\{v > \tilde{v}\}} \bigg( \frac{f(x,v)}{(v+\epsilon)^{q-1}} - \frac{f(x,\tilde{v})}{(\tilde{v}+ \epsilon)^{q-1}} \bigg)  ((v+\epsilon)^q- (\tilde{v}+ \epsilon)^q) ~dx. 
      \end{aligned}
\end{equation}
Since $\frac{\tilde{v}}{\tilde{v}+ \epsilon} \leq \frac{v}{v+ \epsilon} < 1$ in $\{v>\tilde v\}$, then we obtain
\begin{align*}
    \bigg(\frac{v^{2q-1}}{(v+\epsilon)^{q-1}} &- \frac{\tilde{v}^{2q-1}}{(\tilde{v}+ \epsilon)^{q-1}} \bigg) ((v+\epsilon)^q- (\tilde{v}+ \epsilon)^q)\\
    &= \bigg( v^q \bigg(\frac{v}{v+\epsilon}\bigg)^{q-1} - \tilde{v}^q \bigg(\frac{\tilde{v}}{\tilde{v}+ \epsilon}\bigg)^{q-1}\bigg) ((v+\epsilon)^q- (\tilde{v}+ \epsilon)^q)\\
    & \leq v^q ((v+\epsilon)^q- (\tilde{v}+ \epsilon)^q) \leq v^q (v+ \epsilon)^q \leq v^q(v+1)^q.
\end{align*}
In the same fashion, we have
\begin{align*}
    0\leq h \bigg( \frac{v^{q-1}}{(v+\epsilon)^{q-1}} - \frac{\tilde{v}^{q-1}}{(\tilde{v}+ \epsilon)^{q-1}} \bigg) ((v+\epsilon)^q- (\tilde{v}+ \epsilon)^q)\leq h (v+\epsilon)^q \leq  h(v+1)^q.
\end{align*}
%where $C$ and $\tilde{C}$ do not depend on $\epsilon$.\\
Moreover, as $\epsilon \to 0$
$$\bigg(\frac{v^{2q-1}}{(v+\epsilon)^{q-1}} - \frac{\tilde{v}^{2q-1}}{(\tilde{v}+ \epsilon)^{q-1}} \bigg) ((v+\epsilon)^q- (\tilde{v}+ \epsilon)^q)\to(v^q- \tilde{v}^q)^2$$
and
$$h \bigg( \frac{v^{q-1}}{(v+\epsilon)^{q-1}} - \frac{\tilde{v}^{q-1}}{(\tilde{v}+ \epsilon)^{q-1}} \bigg)((v+\epsilon)^q- (\tilde{v}+ \epsilon)^q)\to 0$$
{\it a.e.} in $\Omega$. Then by Lebesgue dominated convergence theorem we have  
\begin{align}\label{0th}
    \int_{\{v > \tilde{v}\}} \bigg( \frac{v^{2q-1}}{(v+\epsilon)^{q-1}} - \frac{\tilde{v}^{2q-1}}{(\tilde{v}+ \epsilon)^{q-1}} \bigg) ((v+\epsilon)^q &- (\tilde{v}+ \epsilon)^q) ~dx\\
    &\to \int_{\{v > \tilde{v}\}} (v^q- \tilde{v}^q)^2~dx\nonumber
\end{align}
and
\begin{equation}\label{1sty}
    \int_{\{v > \tilde{v}\}} h \bigg( \frac{v^{q-1}}{(v+\epsilon)^{q-1}} - \frac{\tilde{v}^{q-1}}{(\tilde{v}+ \epsilon)^{q-1}} \bigg)((v+\epsilon)^q- (\tilde{v}+ \epsilon)^q)~dx \to 0.
\end{equation}
Then by using Fatou's Lemma and $(f_1)$, we have 
\begin{equation}\label{2ndy}
\begin{split}
    &-\displaystyle\liminf_{\epsilon \to 0} \int_{\{v > \tilde{v}\}} \frac{f(x,v)}{(v+\epsilon)^{q-1}} (\tilde{v}+ \epsilon)^q ~dx\leq - \int_{\{v > \tilde{v}\}} \frac{f(x,v)}{v^{q-1}} \tilde{v}^q ~dx, 
\\
   & - \displaystyle\liminf_{\epsilon \to 0} \int_{\{v > \tilde{v}\}} \frac{f(x,\tilde{v})}{(\tilde{v}+\epsilon)^{q-1}} (v+ \epsilon)^q ~dx \leq - \int_{\{v > \tilde{v}\}} \frac{f(x,\tilde{v})}{\tilde{v}^{q-1}} v^q ~dx
\end{split}
\end{equation}
and
\begin{equation}\label{3rdy}
\begin{split}
    &\int_{\{v > \tilde{v}\}}  f(x,v) (v+ \epsilon) ~dx \to \int_{\{v > \tilde{v}\}}  f(x,v) v ~dx,\\
   & \int_{\{v > \tilde{v}\}}  f(x,\tilde{v}) (\tilde{v}+ \epsilon) ~dx\to \int_{\{v > \tilde{v}\}}  f(x,\tilde{v}) \tilde{v}~dx.	
\end{split}
\end{equation}
By Lemma \ref{implem} we have,
\begin{equation}\label{4th}
     \begin{aligned}
          \int_{\{v > \tilde{v}\}} &|\nabla (v+\epsilon)|^{p(x)-2} \nabla (v+\epsilon). \nabla\bigg(\frac{(v+\epsilon)^q - (\tilde{v}+ \epsilon)^q}{(v+ \epsilon)^{q-1}}\bigg)~dx \\
          &+ \int_{\{v > \tilde{v}\}} |\nabla (\tilde{v}+\epsilon)|^{p(x)-2} \nabla (\tilde{v}+\epsilon). \nabla\bigg(\frac{(\tilde{v}+\epsilon)^q - (v+ \epsilon)^q}{(\tilde{v}+ \epsilon)^{q-1}}\bigg)~dx\\
          & \geq 0.
      \end{aligned}
\end{equation}
Then plugging \eqref{0th}-\eqref{4th} and taking $\displaystyle\limsup_{\epsilon \to 0} $ in \eqref{uni}, we get by $(f_3)$
$$ 0 \leq \int_{\{v > \tilde{v}\}} (v^q- \tilde{v}^q)^2~dx \leq \lambda \int_{\{v > \tilde{v}\}}  \bigg( \frac{f(x,v)}{v^{q-1}} - \frac{f(x,\tilde{v})}{\tilde{v}^{q-1}} \bigg)  (v^q- \tilde{v}^q) ~dx\leq 0.$$
It implies $ \tilde{v} \geq v$.
\end{proof}
\begin{cor}
Let $\lambda>0$,  $f:\Omega\times \R^+ \to \R^+$ satisfying $(f_1)$-$(f_3)$ and $h_0 \in L^{2}(\Omega)^+ \cap L^{\gamma}(\Omega)$ where $\gamma >\max \{ 1,\frac{N}{p_{-}} \}$. Then there exists a unique positive distributional solution $u \in \mathcal{D}(\mathcal{R})\cap L^\infty(\Omega)$ of \eqref{E_12} in the same sense as in Corollary \ref{L^infty}.\\
Moreover if $u_1$ and $u_2$ are two positive distributional solutions of \eqref{E_12} for $h_1, h_2 \in L^2(\Omega)^+$ then $\mathcal{R}$ satisfies
\begin{equation}\label{accretive}
    \|(u_1-u_2)^+\|_{L^2} \leq \|(u_1-u_2+\lambda (\mathcal{R}u_1- \mathcal{R}u_2))^+\|_{L^{2}}.
\end{equation}
\end{cor}
\begin{proof}
Define the functional  energy $\mathcal{E}$ on $\dot V^{q}_+\cap L^2(\Omega)$ as $\mathcal{E}(u) =\mathcal{J}(u^{1/q})$ where $\mathcal{J}$ is given by \eqref{defJ}.\\
By Theorem \ref{L^2}, Remark \ref{othsol} and Theorem \ref{uniq3}, $v_0$ is the unique positive solution of \eqref{E_1} and then unique global minimizer of $\mathcal{J}.$ We take $u_0=v_0^q$ and proceed as the proof of Corollary \ref{L^infty} and we obtain
 %$$ \int_{\Omega} u_0 \phi + \lambda \int_{\Omega} |\nabla u_0^{1/q}|^{p(x)-2} \nabla u_0^{1/q} \nabla\bigg(\dfrac{\phi}{u_0^{(q-1)/q}}\bigg) ~dx = \int_{\Omega} h \phi + \lambda \int_{\Omega} \dfrac{f(x, u_0^{1/q})}{u_0^{(q-1)/q}} \phi ~dx $$
$u_0 = v_0^q$ is a distributional solution of ($\ref{E_12}$). Finally Remark \ref{rema} gives \eqref{accretive}.%, Let $(v_n)$ be the weak solution of \eqref{E_1} obtained from Corollary \ref{L^2} corresponding to $(h_n) \in L^{\infty}(\Omega)$ such that $(v_n)$ is uniformly bounded and increasing sequence in $L^{2q}(\Omega), (h_n)_i \to h_i $ in $L^2(\Omega)$ for $i= 1,2$  and 
%\begin{align}
 %   \|[(v_n^q)_1-(v_n^q)_2]^+\|_{L^2{(\Omega)}} \leq \|[(h_n)_1-(h_n)_2]^+\|_{L^{2}(\Omega)}
%\end{align}
%Since $\|(v_n^q)_i\|_{L^2(\Omega)}=\|(v_n)_i\|^q_{L^{2q}(\Omega)}$ therefore upto a subsequence $(v_n^q)_i \to v_i^q$ a.e. then passing to the limit in above inequality we obtain for $h_1, h_2 \in (L^2(\Omega))^+$,
%$$\|[v_1^q-v_2^q]^+\| =\|[u_1-u_2]^+\|_{L^2{(\Omega)}} \leq \|[h_1-h_2]^+\|_{L^{2}(\Omega)}$$
%i.e. $$\|[u_1-u_2]^+\|_{L^2{(\Omega)}} \leq \|[u_1-u_2 + \lambda (Ru_1-Ru_2)]^+\|_{L^{2}(\Omega)}.$$
\end{proof}
\subsection{Existence of a weak solution to \eqref{P}}
%We obtain the following problem by the change of variables $v= u^{1/q}$ in the problem (\ref{weeak}):
%\begin{equation}\label{WeakS4}
 %    \left\{
  %       \begin{alignedat}{2} 
   %          {} (v^q)_t - \frac{\Delta_{p(x)} v}{v^{q-1}}
    %         & {}= h(t,x) + \frac{f(x,v)}{v^{q-1}}
     %        && \quad\mbox{ in }\, Q_T=(0,T) \times \Omega \,;
      %       \\
       %      v & {}= 0
        %     && \quad\mbox{ on }\, (0,T) \times \partial\Omega \,;
         %    \\
        %     v(0,.) & {}= v_0(x) > 0
    %         && \quad\mbox{ in }\, \Omega \,.
%          \end{alignedat}
 %    \right.
%\end{equation}

 %\begin{define}\label{weakso1}
%A weak solution to (\ref{WeakS4}) is any positive function $v \in L^{\infty}(0,T;W_0^{1,p(x)}(\Omega))$ such that $v_t \in L^{2q}(Q_T)$ and satisfying for any $\phi \in C_0^{\infty}(Q_T)$
%\begin{align}\label{weaksodef}
 %   \int_0^T \int_{\Omega} (v^q)_t \phi ~dx ~dt + \int_0^T \int_{\Omega} |\nabla v|^{p(x)-2} |\nabla v| \nabla\bigg(\frac{\phi}{v^{q-1}}\bigg) ~dx ~dt = \int_0^T %\int_{\Omega} h(t,x) \phi ~dx ~dt + \int_0^T \int_{\Omega} \frac{f(x,v)}{v^{q-1}} \phi ~dx ~dt
%\end{align}
 
%\end{define}
In this section, in light of Remark \ref{conversion}, we consider the problem \eqref{WeakS1}
 and establish the existence of weak solution when {$ v_0 \in C^0_d(\overline\Omega)^+ \cap W_0^{1,p(x)}(\Omega).$ 

\noindent{\bf Proof of Theorem \ref{wea}:}}
Let $n^* \in \mathbb{N}^*$ and set $\Delta_t =T/n^*.$ For $0 \leq n \leq n^*,$ we define $t_n= n \Delta_t.$\\
{\bf Step 1 :} Approximation of $h$\\
 For $n \in \{ 1,2,\dots n^* \}$, we define for $ t \in [t_{n-1}, t_n)$ and $ x \in \Omega$
 $$h_{\Delta_t} (t,x)= h^n(x)  \eqdef  \dfrac{1}{\Delta_t} \int_{t_{n-1}}^{t_n} h(s,x) ds.$$
 Then by Jensen inequality, 
 \begin{align*}
      \|h_{\Delta_t}\|^2_{L^2(Q_T)} &=\Delta_t \sum_{n=1}^{N} \|h^n\|^2_{L^2(\Omega)} = \Delta_t  \sum_{n=1}^{N} \left\|\dfrac{1}{\Delta_t} \int_{t_{n-1}}^{t_n} h(s,x ) ds \right\|^2_{L^2(\Omega)}\\
      & \leq\sum_{n=1}^{N} \int_{t_{n-1}}^{t_n} \|h(s,.)\|^2_{L^2(\Omega)} ds  \leq \|h\|^2_{L^2(Q_T)}.
 \end{align*}
Hence $h_{\Delta_t} \in L^2(Q_T)$ and $h^n \in L^2(\Omega)$ and let $\epsilon >0 $, then there exists a function $h_{\epsilon} \in C_0^1(Q_T)$ such that $\|h-h_{\epsilon}\|_{L^2(Q_T)} < \frac{\epsilon}{3}.$\\
 Hence, $$\|(h_{\epsilon})_{\Delta_t} - h_{\Delta_t}\|_{L^2(Q_T)} \to 0.$$
 Since $\|h_{\epsilon}- (h_{\epsilon})_{\Delta_t}\|_{L^2(Q_T)} \to 0$ as $\Delta_t \to 0$ then for small enough $\Delta_t$ we have
 $$\|h_{\Delta_t} -h \|_{L^2(Q_T)} \leq \|(h_{\epsilon})_{\Delta_t} - h_{\Delta_t}\|_{L^2(Q_T)} + \|h_{\epsilon}- (h_{\epsilon})_{\Delta_t}\|_{L^2(Q_T)} + \|h-h_{\epsilon}\|_{L^2(Q_T)} < \epsilon.$$
 Hence $h_{\Delta_t} \to h $ in $L^2(Q_T).$\\
{\bf Step 2: } Time discretization of \eqref{WeakS1}\\
Define the following implicit Euler scheme and for $n \geq 1 , v_n $ is the weak solution of 
\begin{equation}\label{Itersch1}
     \left\{
         \begin{alignedat}{2} 
             {}  \bigg(\dfrac{v^q_n-v^q_{n-1}}{\Delta_t}\bigg) v_n^{q-1} -\Delta_{p(x)} {v_n}  
             & {}=  h^n v_n^{q-1}+ f(x,v_n)
             && \quad\mbox{ in }\,  \Omega \,;
             \\
             v_n & {}> 0
             && \quad\mbox{ in }\,  \Omega \,;
             \\
             v_n & {}= 0
             && \quad\mbox{ on }\,  \partial\Omega \,.
        \end{alignedat}
     \right.
\end{equation}
Note that the sequence $(v_n)_{n=1,2,\dots ,n^*}$ is well-defined. Indeed for $n=1 $ the existence and the uniqueness of $v_1 \in C^{1,\alpha}(\overline{\Omega}) \cap C_d^0(\overline{\Omega})^+ $ follows from Theorems \ref{exis} and \ref{uniq3} with $h = \Delta_t h^1 +v_0^q \in L^{\infty}(\Omega)^+$. Hence by induction we obtain in the same way the existence and the uniqueness of the solution $v_n$ for any $n=2,3,\dots ,n^*$ where $v_n \in C^{1,\alpha}(\overline{\Omega}) \cap C^0_d(\overline{\Omega})^+.$\\
{\bf Step 3}: Existence of a subsolution and supersolution \\
Now we construct a subsolution and a supersolution $ \underline{w} $ and $\overline{w}$ of \eqref{Itersch1} such that for each $n \in \{0, 1, 2, \dots, n^*\}, v_n$ satisfies $ 0 <\underline{w} \leq v_n \leq \overline{w}$.\\ Rewrite \eqref{Itersch1} as
\begin{align}\label{subsolut}
v_n^{2q-1} - \Delta_t \Delta_{p(x)} v_n = \left( \Delta _t h^n + v_{n-1}^q \right) v_n^{q-1} + \Delta_t f(x, v_n).
\end{align}
Then following arguments in the proof of Theorems \ref{exis} and \ref{uniq3}, from Theorem \ref{C^1} and from Lemma \ref{HMP}, for any $\mu >0$ there exists a unique weak solution, $w_{\mu} \in C^{1,\alpha}(\overline{\Omega}) \cap C^0_d(\overline{\Omega})^+$, to
\begin{equation}\label{subsolu}
     \left\{
         \begin{alignedat}{2} 
             {} - \Delta_{p(x)} {w}  
             & {}=  \mu(h_0 w^{q-1}+ f(x,w))
             && \quad\mbox{ in }\,  \Omega \,;
             \\
             w& {}> 0
             && \quad\mbox{ in }\,  \Omega \,;             
             \\
             w& {}= 0
             && \quad\mbox{ on }\,  \partial\Omega \,.
        \end{alignedat}
     \right.
\end{equation}
%{\it i.e.} for all $\phi \in W_0^{1,p(x)}(\Omega)$ 
%\begin{equation*}
%\int_{\Omega} |\nabla w|^{p(x)-2} \nabla w. \nabla \phi~dx = \mu \int_{\Omega}(h_0 w^{q-1}+ f(x,w))\phi ~dx.
%\end{equation*}
% Then by following the same proof of Theorem \ref{exis},  there exists a unique non-negative solution $\underline{w}_\mu \in C^{1,\alpha}(\overline\Omega)$ of \eqref{subsolu} and by strong maximum principle and Hopf Lemma we obtain $\underline{w}_\mu>0$ and $\frac{\partial\underline{w}_\mu}{\partial \eta} > 0$ where $\eta$ is inward unit normal vector. Then Mean value theorem gives $\underline{w}_\mu \in C^0_d(\overline\Omega)^+.$ 
Let $\mu_1 < \mu_2$ and $w_{\mu_1}, w_{\mu_2}$ be weak solutions of \eqref{subsolu}. Then, 
\begin{equation*}
\int_{\Omega} |\nabla w_{\mu_1}|^{p(x)-2} \nabla w_{\mu_1}. \nabla \phi~dx= \mu_1 \int_{\Omega} (h_0 w_{\mu_1}^{q-1}+f(x, w_{\mu_1})) \phi ~dx,\\
\end{equation*}
\begin{equation*}
\int_{\Omega} |\nabla w_{\mu_2}|^{p(x)-2} \nabla w_{\mu_1}. \nabla \psi~dx= \mu_2 \int_{\Omega} (h_0 w_{\mu_2}^{q-1}+ f(x, w_{\mu_2})) \psi ~dx.\\
\end{equation*}
Subtracting the last two equations with $\phi= \frac{(w_{\mu_1}^q- w_{\mu_2}^q)^+}{w_{\mu_1}^{q-1}}$ and $\psi= \frac{(w_{\mu_2}^q- w_{\mu_1}^q)^-}{w_{\mu_2}^{q-1}} \in W_0^{1,p(x)}(\Omega)$ we obtain, by Lemma \ref{implem} and $(f_3)$, $w_{\mu_1} \leq w_{\mu_2}.$\\
Then by using Theorems \ref{C^1} and \ref{regi}, we can choose $\mu$ small enough such that $\|w_{\mu} \|_{C^{1, \alpha}(\overline{\Omega})} \leq C_{\mu_0}$ for all $\mu \leq \mu_0$ and $\|w_\mu\|_{L^{\infty}(\Omega)} \to 0$ as $\mu \to 0$. Therefore $\{w_{\mu}: \mu \leq \mu_0\}$ is uniformly bounded and equicontinuous in $C^1(\overline\Omega)$ and by Arzela Ascoli theorem $\|w_{\mu}\|_{C^1(\overline\Omega)} \rightarrow 0$ as $\mu \rightarrow 0$ up to a subsequence. Then by mean value theorem we can choose $\mu$ small enough such that there exists $\underline{w} \in C^{1,\alpha}(\overline{\Omega}) \cap C^0_d(\overline\Omega)^+$ such that $0< \underline{w} \eqdef w_\mu \leq v_0.$ Also $\underline{w}$ is the subsolution of \eqref{subsolut} for $n=1$ {\it i.e.} 
%Also $\|w_\mu\|_{C(\overline\Omega)} \to 0$ as $\mu \to 0.$ Indeed, using $(f_1)$ and $(f_3)$ we have $f(x,w) \leq C_1 + C_2 w^{q-1}$ therefore $\mu(h_0 w^{q-1}+ f(x,w))- w^{2q-1} \leq \mu (h_0 + C_1 + C_2 w^{q-1})- w^{2q-1} \leq 0$ for all $w \geq \Lambda_\mu$ where $\zeta_\mu =(\mu (\|h\|_{L^{\infty}(\Omega)}+C_2)(1+C_1 \mu))^{1/q} \to 0$ as $\mu \to 0.$ Take $\phi= (w_{\mu}- \zeta)^+$
%\begin{align*}
%int_{w_\mu \geq \zeta_\mu} |\nabla w_\mu|^{p(x)}~dx = \int_{w_\mu \geq \zeta_\mu} \mu(h_0 w_\mu^{q-1}+ f(x,w_\mu)) \phi - w_\mu^{2q-1} \phi ~dx <0. 
%\end{align*}
%implies $w_\mu \leq \zeta_\mu$ a.e. and $\|w_\mu\|_{C(\overline{\Omega})} \to 0$ as $\mu \to 0.$ Then by choosing $\mu$ small enough such that there exists $\underline{w} \in C^0_d(\overline\Omega)^+$ and $\underline{w} \leq v_0.$ Then we have,
\begin{equation*}
\begin{split}
\int_{\Omega} \underline{w}^{2q-1} \phi~dx + \Delta_t \int_{\Omega} |\nabla \underline{w}|^{p(x)-2} \nabla \underline{w}. \nabla \phi~dx &\leq \Delta_t \int_{\Omega} (h^1 \underline{w}^{q-1}+ f(x,\underline{w})) \phi ~dx\\
& \ \ \ + \int_{\Omega} v_0^q \underline{w}^{q-1} \phi ~dx
\end{split}
\end{equation*}
for all $\phi \in W_0^{1,p(x)}(\Omega)$ and $\phi \geq 0.$ We also recall $v_1$ satisfies   
\begin{equation*}
\begin{split}
\int_{\Omega} v_1^{2q-1} \psi~dx + \Delta_t \int_{\Omega} |\nabla v_1|^{p(x)-2} \nabla v_1. \nabla \psi~dx &=\Delta_t \int_{\Omega} (h^1 v_1^{q-1}+ f(x,v_1)) \psi ~dx\\
& \ \ \ + \int_{\Omega} v_0^q v_1^{q-1} \psi ~dx
\end{split}
\end{equation*}
%Using test functions $\phi=\frac{(\underline{w}^q-v_1^q)^+}{\underline{w}^{q-1}}$ and $\psi=\frac{(v_1^q-%%\underline{w}^q)^-}{v_1^{q-1}} \in \mathbf{W}$ in the last two inequalities and 
for all $\psi \in W_0^{1,p(x)}(\Omega)$.\\
By Theorem \ref{uniq3}, we obtain, $\underline{w} \leq v_1$ and then by induction a subsolution $\underline{w}$ such that $0 <\underline{w} \leq v_n$ for all $n=0,1,2,\dots,n^*$.\\
Now we construct a supersolution. For that, we consider the following problem: 
\begin{equation}\label{supprob}
     \left\{
         \begin{alignedat}{2} 
             {}  -\Delta_{p(x)} {w}  
             & {}= \|h\|_{L^{\infty}} w^{q-1} + f(x,w) + K
             && \quad\mbox{ in }\,  \Omega \,;
             \\
             w & {}> 0
             && \quad\mbox{ in }\,  \Omega \,; 
             \\
             w & {}= 0
             && \quad\mbox{ on }\,  \partial\Omega \,.
        \end{alignedat}
     \right.
\end{equation}
As above, there exists a unique weak solution to \eqref{supprob}, $\overline{w}_K \in C^1(\overline\Omega) \cap C^0_d(\overline\Omega)^+$. Let $w_{K}$ be the unique weak solution of 
\begin{equation}\label{lambd}
     \left\{
         \begin{alignedat}{2} 
             {} - \Delta_{p(x)} {w_K}  
             & {}=  K
             && \quad\mbox{ in }\,  \Omega \,;
             \\
             w_{K} & {}= 0
             && \quad\mbox{ on }\,  \partial\Omega \,.
        \end{alignedat}
     \right.
\end{equation}
From Theorem \ref{regi}, $w_{K} \geq C K^{1/(p_+-1+\nu)} \dist(x, \partial\Omega)$ where $\nu \in (0,1)$ and $\|w_{K}\|_{L^{\infty}(\Omega)} \to \infty$ as $K \to \infty $. Then by weak comparison principle we can choose $K$ large enough such that there exists such that $v_0 \leq w_K < \overline{w}\eqdef \overline{w}_K.$ We easily check that $\overline{w}$ is the supersolution of \eqref{subsolut} for $n=1$ {\it i.e.}
\begin{equation*}
\begin{split}
\int_{\Omega} \overline{w}^{2q-1} \phi~dx + \Delta_t \int_{\Omega} |\nabla \overline{w}|^{p(x)-2} \nabla \overline{w}). \nabla \phi~dx & \geq \Delta_t \int_{\Omega} (h^1 \overline{w}^{q-1}+ K+ f(x,\overline{w}) \phi ~dx\\
& \ \ \ + \int_{\Omega} v_0^q \overline{w}^{q-1} \phi ~dx
\end{split}
\end{equation*}
for all $\phi \in W_0^{1,p(x)}(\Omega)$ and $\phi \geq 0.$
From Theorem \ref{uniq3}, we get $\overline{w} \geq v_1$ and then by induction we have $\overline{w} \geq v_n $ for all $n \in \{1,2,\dots n^*\}.$\\
{\bf Step 4}:  Energy estimates \\
Define the function for $n=1,\dots , n^*$ and $t\in [t_{n-1}, t_n)$
\begin{align*}
    v_{\Delta_t}(t)= v_n \ \ \text{and} \ \ \tilde{v}_{\Delta_t}(t)= \dfrac{t-t_{n-1}}{\Delta_t}(v_n^q-v_{n-1}^q)+ v_{n-1}^q
\end{align*}
which satisfies
\begin{align}\label{passto}
    v_{\Delta_t}^{q-1} \dfrac{\partial \tilde{v}_{\Delta_t}}{\partial t} - \Delta_{p(x)} v_{\Delta_t}= f(x,v_{\Delta_t}) + h^n v_{\Delta_t}^{q-1}.
\end{align} 
Multiplying the equation \eqref{Itersch1} by $\dfrac{v_n^q-v_{n-1}^q}{v_n^{q-1}}$ and summing from $n=1$ to $ n' \leq n^*$, we get
\begin{align*}
    \sum_{n=1}^{n'} \int_{\Omega} &\Delta_t \bigg(\dfrac{v_n^q- v_{n-1}^q}{\Delta_t}\bigg)^2 ~dx + \sum_{n=1}^{n^{'}} \int_{\Omega} |\nabla v_n|^{p(x)-2} \nabla v_n.\nabla\bigg(\dfrac{v_n^q-v_{n-1}^q}{v_n^{q-1}}\bigg) ~dx \\
    &=  \sum_{n=1}^{n'} \int_{\Omega} h^n (v_n^q-v_{n-1}^q) ~dx + \sum_{n=1}^{n'} \int_{\Omega} \dfrac{f(x,v_n)}{v_n^{q-1}} (v_n^q-v_{n-1}^q) ~dx.
\end{align*}
Then from Young inequality we have,
\begin{align*}
    \sum_{n=1}^{n'} \int_{\Omega} & \Delta_t \bigg(\dfrac{v_n^q- v_{n-1}^q}{\Delta_t}\bigg)^2 ~dx + \sum_{n=1}^{n'} \int_{\Omega} |\nabla v_n|^{p(x)-2} \nabla v_n.\nabla\bigg(\dfrac{v_n^q-v_{n-1}^q}{v_n^{q-1}}\bigg) ~dx \\
    &\leq \sum_{n=1}^{n'} \Delta_t \|h^n\|^2_{L^2} + \dfrac{1}{4}\sum_{n=1}^{n'} \int_{\Omega} \Delta_t \bigg(\dfrac{v_n^q- v_{n-1}^q}{\Delta_t}\bigg)^2 ~dx\\
    &+\sum_{n=1}^{n'} \Delta_t \left\| \dfrac{f(x,v_n)}{v_n^{q-1}}\right\|_{L^2}^2 + \dfrac{1}{4} \sum_{n=1}^{n'} \int_{\Omega}\Delta_t \bigg(\dfrac{v_n^q- v_{n-1}^q}{\Delta_t}\bigg)^2 ~dx
\end{align*}
{\it i.e.} 
\begin{align*}
    \dfrac{1}{2} \sum_{n=1}^{n'} \int_{\Omega} \Delta_t \bigg(\dfrac{v_n^q- v_{n-1}^q}{\Delta_t}\bigg)^2 dx&+ \sum_{n=1}^{n'} \int_{\Omega} |\nabla v_n|^{p(x)-2} \nabla v_n. \nabla\bigg(\dfrac{v_n^q-v_{n-1}^q}{v_n^{q-1}}\bigg) ~dx \\
    &\leq \sum_{n=1}^{n'} \Delta_t \|h^n\|^2_{L^2} +\sum_{n=1}^{n'} \Delta_t  \left\| \dfrac{f(x,v_n)}{v_n^{q-1}}\right\|_{L^2}^2.
\end{align*}
Using $\underline{w} \leq v_n \leq \overline{w} $, from \eqref{infty} and $q \leq \frac{N}{2}+1$, we obtain 
\begin{equation}\label{require}
\begin{split}
\int_{\Omega} \left|\frac{f(x,v_n)}{v_n^{(q-1)}}\right|^2 ~dx \leq C_1 \int_{\Omega} \frac{1}{\dist^{2(q-1)}(x,\partial\Omega)} + \dist^{2(p_- -q)}(x,\partial\Omega) ~dx \leq C 
\end{split}
\end{equation}
where $C$ is independent of $n$. Then by {\bf Step 1}, we obtain 
\begin{align}\label{deribound}
    \bigg(\dfrac{\partial \tilde{v}_{\Delta_t}}{\partial t}\bigg) \ \text{is bounded in} \ L^2(Q_T)\  \text{uniformly in} \ \Delta_t.
\end{align}
Now from Lemma \ref{implem}, we have 
\begin{align*}
    |\nabla v_n|^{p(x)-2} \nabla v_n. \nabla \bigg(\dfrac{v_{n-1}^q}{v_n^{q-1}}\bigg) &\leq |\nabla v_{n-1}|^q |\nabla v_n|^{p(x)-q}\\
    & \leq \dfrac{q}{p(x)} |\nabla v_{n-1}|^{p(x)} + \dfrac{(p(x)-q)}{p(x)} |\nabla v_n|^{p(x)}.
\end{align*}
Then we obtain for any $n'\geq 1$
\begin{align*}
    \sum_{n=1}^{n'} \Delta_t\|h^n\|^2_{L^2} &+\sum_{n=1}^{n'} \Delta_t \left\| \dfrac{f(x,v_n)}{v_n^{q-1}}\right\|_{L^2}^2  \\
    &\geq \sum_{n=1}^{n'} \int_{\Omega} |\nabla v_n|^{p(x)-2} \nabla v_n. \nabla\bigg(\dfrac{v_n^q-v_{n-1}^q}{v_n^{q-1}}\bigg) ~dx\\
    &\geq \sum_{n=1}^{n'} \left[ \int_{\Omega} |\nabla v_n|^{p(x)} ~dx -  \int_{\Omega} \dfrac{q}{p(x)} |\nabla v_{n-1}|^{p(x)} ~dx\right.\\
     &\ \ \ \ \ \ \ \ \ \ \ \ \ \ \ \ \ \ \ \ \ \ \ \ \ \ \ \ \left.- \int_{\Omega}  \dfrac{(p(x)-q)}{p(x)} |\nabla v_n|^{p(x)} ~dx \right]\\
    & \geq q \int_{\Omega}  \dfrac{|\nabla v_{n'}|^{p(x)}}{p(x)} ~dx - q  \int_{\Omega}  \dfrac{|\nabla v_0|^{p(x)}}{p(x)} ~dx
\end{align*}
which implies that
\begin{align}\label{iterfuncbound1}
   (v_{\Delta_t}) \ \text{is bounded in } L^{\infty}(0,T;W_0^{1,p(x)}(\Omega)) \ \text{uniformly in }\ \Delta_t.
\end{align}
 Since 
\begin{equation*}
\begin{split}
\nabla(\tilde{v}^{1/q}_{\Delta_t}) &= \dfrac{1}{q} \zeta \nabla v_n \bigg(\zeta +(1-\zeta) \bigg(\dfrac{v_{n-1}}{v_n}\bigg)^q \bigg)^{(1-q)/q} \\
& \ \ \ \ \ \ \ \ + (1-\zeta) \nabla v_{n-1} \bigg((1-\zeta) +\zeta \bigg(\dfrac{v_{n}}{v_{n-1}}\bigg)^q \bigg)^{(1-q)/q}   
\end{split} 
\end{equation*}
where $\zeta= \dfrac{t-t_{n-1}}{\Delta_t}$, then we conclude that
\begin{align}\label{iterfuncbound2}
   (\tilde{v}^{1/q}_{\Delta_t}) \ \text{is bounded in } L^{\infty}(0,T;W_0^{1,p(x)}(\Omega)) \ \text{uniformly in }\ \Delta_t.
\end{align} 
Since $\left(\dfrac{v_n}{v_{n-1}}\right)$ is uniformly bounded in $L^{\infty}(\Omega)$,  $v_{\Delta_t} \overset{\ast}{\rightharpoonup} v$ and $\tilde{v}^{1/q}_{\Delta_t} \overset{\ast}{\rightharpoonup} \tilde{v}$ in $L^{\infty}(0,T;W_0^{1,p(x)}(\Omega)).$
% Thesis
%Now we prove that $ v =\tilde{v}.$ Let $\phi \in C_c^{\infty}(Q_T)$ then for fixed $t \in (0,T)$, there exists $\xi_t \in W_0^{1,p(x)}(\Omega)$ satisfies 
%\begin{equation}\label{distri}
%     \left\{
 %        \begin{alignedat}{2} 
 %            {} - \Delta_{p(x)} \xi_t  
 %            & {}=  \phi_t
 %            && \quad\mbox{ in }\,  \Omega \,;
 %            \\
 %            \xi_t & {}= 0
%             && \quad\mbox{ on }\,  \partial\Omega \,;
%        \end{alignedat}
%     \right.
%\end{equation}
%in weak sense where $\phi_t(x)= \phi(t,x)$ i.e.
%\begin{align*}
%    \int_{\Omega} |\nabla \xi_t|^{p(x)-2} \nabla \xi_t \nabla \tilde{\phi}_t = \int_{\Omega} \phi_t \tilde{\phi}_t \ \ \ \forall \tilde{\phi}_t \in W_0^{1,p(x)}(\Omega)
%\end{align*}
%Then by taking $ \tilde{\phi}_t  = \xi_t $, we have $\|\xi_t\|_{W_0^{1,p(x)}(\Omega)} \leq \|\xi_t\|_{L^{p(x)}(\Omega)} %\|\tilde{\phi}_t\|_{L^{p'(x)}(\Omega)}.$ Therefore $|\nabla \xi_t|^{p(x)-2} \nabla \xi_t \in L^1(0,T, W^{-1,p(x)}(\Omega))$ where $\xi_t(x)= \xi(t,x).$ and weak %$*$ convergence implies $v_{\Delta_t} \to v $ in $D'(Q_T)$ and similarly $\tilde{v}^{1/q}_{\Delta_t} \to \tilde{v} $ in $D'(Q_T).$\\
Furthermore using \eqref{deribound}, we have 
\begin{equation}\label{uniq}
 \sup_{t \in (0,T)} \|\tilde{v}_{\Delta_t}^{1/q}- v_{\Delta_t}\|_{L^{2q}(\Omega)}^{2q} \leq \sup_{t \in (0,T)} \|\tilde{v}_{\Delta_t} - v_{\Delta_t}^q\|^2_{L^2(\Omega)} \leq \Delta_t \to 0 \ \text{as}\  \Delta_t \to 0.      
\end{equation}
It follows from \eqref{uniq} that $v=\tilde{v}.$
%Thesis
%Then using above estimates and H\"older inequality, we obtain for $\phi \in C_c^{\infty}(Q_T)$
%$$\int_0^T \int_{\Omega} |\tilde{v}-v| \phi \leq \int_0^T \int_{\Omega} |\tilde{v}- \tilde{v}_{\Delta_t}^{1/q}| \phi + \int_0^T \int_{\Omega} %|\tilde{v}_{\Delta_t}^{1/q}- v_{\Delta_t}| \phi + \int_0^T \int_{\Omega} |v_{\Delta_t}-v| \phi \to 0  $$
%implies $v= \tilde{v}.$\\
By mean value theorem and \eqref{deribound}, we get that  $(\tilde{v}_{\Delta_t})_{\Delta_t}$ is equicontinuous in $C(0,T; L^r(\Omega))$ for $1 < r \leq 2$. Thus using $\underline{w}^q \leq \tilde{v}_{\Delta_t} \leq \overline{w}^q$ together with the interpolation inequality $\|.\|_r \leq \|.\|_{\infty}^{\alpha} \|.\|_2^{1-\alpha}$, with $\dfrac{1}{r}= \dfrac{\alpha}{\infty} + \dfrac{1-\alpha}{2},$ we obtain that $(\tilde{v}_{\Delta_t})_{\Delta_t}$ and $(\tilde{v}^{1/q}_{\Delta_t})_{\Delta_t}$ is equicontinuous in $C(0,T; L^r(\Omega))$ for any $1< r < +\infty$. Again using interpolation inequality and Sobolev embedding, we get as $\Delta_t \to 0^+$ and up to a subsequence that for all $r>1$
\begin{align}\label{equi2}
    \tilde{v}_{\Delta_t} \to v^q \ \text{in} \ C(0,T;L^r(\Omega)), 
\end{align}
and 
\begin{align}\label{equi3}
    v_{\Delta_t} \to v \ \text{in} \ L^{\infty}(0,T;L^r(\Omega)).
\end{align}
From \eqref{deribound} and \eqref{equi2}, we obtain
\begin{align}\label{deribound2}
    \dfrac{\partial \tilde{v}_{\Delta_t}}{\partial t} \to \dfrac{\partial v^q}{\partial t} \ \ \text{in} \ L^2(Q_T).
\end{align}
{\bf Step 5} : $v$ satisfies \eqref{weaksodef}\\
Multiplying  \eqref{passto} by $(v_{\Delta_t}-v)$ and integrating by parts, we get
\begin{align*}
    \int_0^T \int_{\Omega} v_{\Delta_t}^{q-1} &\dfrac{\partial \tilde{v}_{\Delta_t}}{\partial t} (v_{\Delta_t}-v) ~dxdt + \int_0^T \int_{\Omega} |\nabla v_{\Delta_t}|^{p(x)-2} \nabla v_{\Delta_t}. \nabla(v_{\Delta_t}-v)~dxdt\\
    &= \int_0^T \int_{\Omega}  f(x,v_{\Delta_t}) (v_{\Delta_t}-v)~dxdt + \int_0^T \int_{\Omega} h^n v_{\Delta_t}^{q-1} (v_{\Delta_t}-v)~dxdt.
\end{align*}
From \eqref{equi3} and \eqref{deribound2} , we have 
$$\left|\int_0^T \int_{\Omega} v_{\Delta_t}^{q-1} \dfrac{\partial \tilde{v}_{\Delta_t}}{\partial t} (v_{\Delta_t}-v)~dxdt\right|+\left|\int_0^T \int_{\Omega} h^n v_{\Delta_t}^{q-1} (v_{\Delta_t}-v)~dxdt\right| = o_{\Delta_t}(1)$$
and from \eqref{iterfuncbound1}, \eqref{iterfuncbound2}, \eqref{equi3} and Lebesgue Dominated convergence theorem, 
 $$\int_0^T \int_{\Omega}  f(x,v_{\Delta_t}) (v_{\Delta_t}-v)~dx = o_{\Delta_t}(1).$$
 Then we obtain 
 \begin{align*}
     \int_0^T \int_{\Omega} |\nabla v_{\Delta_t}|^{p(x)-2} \nabla v_{\Delta_t}.\nabla(v_{\Delta_t}-v) ~dx \to 0 \ \text{as} \  \Delta_t \to 0^+.
 \end{align*}
 Then from [\textbf{Step 4}, Proof of Theorem 1.1, \cite{*10}] and from classical compactness argument we get 
 \begin{align}\label{maincon}
      |\nabla v_{\Delta_{t}}|^{p(x)-2} \nabla v_{\Delta_{t}} \to |\nabla v|^{p(x)-2} \nabla v \ \ \text{in} \ (L^{p(x)/(p(x)-1)}(Q_T))^N.
 \end{align}
 From \eqref{uniq} and \eqref{equi2} we have, 
\begin{equation}\label{prop}
\begin{split}
         \|v_{\Delta_t}^{q-1} - v^{q-1}\|_{L^2(Q_T)} &\leq  \|v_{\Delta_t}^{q-1} - v^{q-1}\|_{L^{\infty}(0,T;L^2)}\\
         &\leq \|v_{\Delta_t}^{q-1} - v^{q-1}\|_{L^{\infty}(0,T;L^{\frac{2q}{q-1}})}\\
				&\leq \|v_{\Delta_t}^{q} - v^{q}\|_{L^{\infty}(0,T;L^2)}\\
         &\leq \|v_{\Delta_t}^{q} - \tilde{v}_{\Delta_t}\|_{L^{\infty}(0,T;L^2)} + \|\tilde{v}_{\Delta_t} - v^{q}\|_{L^{\infty}(0,T;L^2)} \to 0
      \end{split}
 \end{equation}
 as  $\Delta_t \to 0$. By H\"older inequality we have for $\phi \in C_c^{\infty}(Q_T)$
 \begin{align*}
     &\int_0^{T}\int_{\Omega} \bigg(v_{\Delta_t}^{q-1} \frac{\partial \tilde{v}_{\Delta_t}}{\partial t} - \frac{\partial v^q}{\partial t} v^{q-1}\bigg) \phi~dx\\
     &=\int_0^{T}\int_{\Omega}  v_{\Delta_t}^{q-1}  \bigg(\frac{\partial \tilde{v}_{\Delta_t}}{\partial t} - \frac{\partial v^q}{\partial t} \bigg) \phi~dx +  \int_0^{T}\int_{\Omega}  \frac{\partial v^q}{\partial t} (v_{\Delta_t}^{q-1}-  v^{q-1})\phi ~dx\\
     & \leq \|v_{\Delta_t}^{q-1} \phi\|_{L^2(Q_T)} \left\|\left(\frac{\partial \tilde{v}_{\Delta_t}}{\partial t} - \frac{\partial v^q}{\partial t} \right)\right\|_{L^2(Q_T)} + \|v_{\Delta_t}^{q-1} - v^{q-1}\|_{L^2(Q_T)} \left\| \phi \frac{\partial \tilde{v}_{\Delta_t}}{\partial t}\right\|_{L^2(Q_T)} 
 \end{align*}
 and 
 \begin{align*}
     \int_0^T \int_{\Omega} &(h^n v_{\Delta_t}^{q-1} - h v^{q-1}) \phi~dx \\
     &= \int_0^T \int_{\Omega} h^n (v_{\Delta_t}^{q-1} - v^{q-1}) \phi ~dx + \int_0^T \int_{\Omega} (h^n- h) v^{q-1} \phi~dx\\
     & \leq \|h^n \phi\|_{L^2(Q_T)} \|v_{\Delta_t}^{q-1} - v^{q-1}\|_{L^2(Q_T)} + \|v^{q-1} \phi\|_{L^2(Q_T)} \|h^n-h\|_{L^2(Q_T)}.
 \end{align*}
 Then from \eqref{deribound}, \eqref{equi3}, \eqref{deribound2}, \eqref{prop} and {\bf Step 1} we obtain
 \begin{equation}\label{disti1}
\begin{split}
     &\int_0^{T}\int_{\Omega} \bigg(v_{\Delta_t}^{q-1} \frac{\partial \tilde{v}_{\Delta_t}}{\partial t} - \frac{\partial v^q}{\partial t} v^{q-1}\bigg) \phi~dx \to 0,\\
     &\int_0^T \int_{\Omega} (h^n v_{\Delta_t}^{q-1} - h v^{q-1}) \phi ~dx\to 0 \ \ \text{as} \ \ \Delta_t \to 0.
		\end{split}
 \end{equation}
 From \eqref{equi3} we have $f(x,v_{\Delta_t}) \to f(x,v)$ pointwise and from \eqref{iterfuncbound1} together with \eqref{iterfuncbound2} we have $\int_{\Omega} f(x,v_{\Delta_t}) \phi~dx $ is bounded uniformly in $\Delta_t$. Then by Lebesgue dominated convergence theorem we have
 \begin{equation}\label{disti3}
     \int_0^T \int_{\Omega} (f(x,v_{\Delta_t}) - f(x,v)) \phi ~dx \to 0 \ \ \text{as} \ \ \Delta_t \to 0.
 \end{equation}
 Then finally gathering \eqref{maincon}, \eqref{disti1} and \eqref{disti3}, we conclude by passing to the limits in equation \eqref{passto} that $v$ is weak solution of \eqref{WeakS1}.
\qed
%end{proof}
\begin{remark}
For $q > \frac{N}{2}+1$, if $f$ satisfies $\displaystyle\lim_{s \to 0^+} \frac{f(x,s)}{s^{\alpha}} =0$ where $\alpha > q-1-\frac{N}{2}$ then Theorem \ref{wea} holds. Since $\underline{w} \leq v_n \leq \overline{w}$ then \eqref{require} is in this case
\begin{equation*}
\begin{split}
\int_{\Omega} \left|\frac{f(x,v_n)}{v_n^{q-1}}\right|^2~dx &\leq C_1 \int_{\Omega} \frac{\overline w\,^{2 \alpha}}{\dist^{2(q-1)}(x,\partial\Omega)}~dx + C_2\leq C
\end{split}
\end{equation*}
where $C$ is independent of $n.$ 
\end{remark}
\begin{remark}
All the results in Section \ref{section1}, Section \ref{section2} and Theorem \ref{wea} hold if we replace the assumption $(f_2)$ by $h \geq c >0.$
\end{remark}
%\begin{cor}
%Let $v_1$ and $v_2$ the weak solutions of \eqref{WeakS1} in the sense of Definition \ref{weakso} with initial data $u_0$, $v_0$ such that $ u_0, v_0 \in L^{\infty}(Q_T)^+$ and $h, g\in L^\infty(Q_T)^+$, $h_0, g_0 \in L^{\infty}(\Omega)^+$, $h_0, g_0 \not\equiv 0$ such that $h \geq h_0 >0 ,\ g \geq g_0 >0$ respectively. Then, for any $0\leq t\leq T$, 
%$$\Vert (v_1^q(t)-v^q_2(t))^+\Vert_{L^2(\Omega)}\leq \Vert (u_0-v_0)^+\Vert_{L^2(\Omega)}+ \int_0^t \|r(\gamma)-k(\gamma)\|_{2^+} d\gamma.$$
%\end{cor}
\noindent{\bf Proof of Theorem \ref{contr}:}
For a given function $g$, let $\|g\|_{2^+}\eqdef \|[g]^+\|_{L^2(\Omega)}$. For $z \in \mathcal{D}(\mathcal{R})$ and $r, k \in L^\infty(Q_T)^+$ satisfying assumptions in Theorem \ref{contr}, set \\
 $$\phi(t,s)=\|r(t)-k(s)\|_{2^+} \quad \forall\,(t,s) \in [0,T] \times [0,T], $$ 
for $t \in [-T,T]$
 $$b(t,r,k)= \|u_0^q -z\|_{2^+} + \|v_0^q -z\|_{2^+} + |t| \|{\mathcal R}z\|_{2^+} + \int^{t^+}_{0} \|r(\tau)\|_{2^+} d\tau + \int^{t^-}_{0} \|k(\tau)\|_{2^+} d\tau $$
 and 
 \begin{equation*} 
\psi(t,s) = b(t-s,r,k) +
    \begin{cases}
      \int_0^s \phi(t-s+\tau, \tau) d\tau  & \text{if}\ \  0 \leq s \leq t \leq T   \\
      \int_0^t \phi(\tau,s-t+\tau) d\tau  & \text{if} \ \  0 \leq t \leq s \leq T
    \end{cases}
  \end{equation*}
  is a solution of 
  \begin{equation}\label{psi}
     \left\{
         \begin{alignedat}{2} 
             {}  \dfrac{\partial \psi}{\partial t} (t,s) +\dfrac{\partial \psi}{\partial s} (t,s)   
             & {}=  \phi(t,s)
             && \quad\mbox{ on }\, (t,s) \in [0,T] \times [0,T]\,;\\
             \psi(t,0) & {}= b(t,r,k)
             && \quad\mbox{ on }\, t \in[0,T]\,;\\
             \psi(0,s) & {}= b(-s,r,k)
             && \quad\mbox{ on }\, s \in[0,T]\,.
        \end{alignedat}
     \right.
\end{equation}
Define the following iterative scheme, $u^0=u_0^q$ and for $n \geq 1 , u^n $ is the solution of 
\begin{equation}\label{Itersch}
     \left\{
         \begin{alignedat}{2} 
             {}  \dfrac{u^n-u^{n-1}}{\Delta_t} + \mathcal{R} u^n
             & {}=  h^n 
            && \quad\mbox{ in }\, \Omega \,;
             \\
             u^n & {}= 0
           && \quad\mbox{ on }\,  \partial\Omega \,.
       \end{alignedat}
     \right.
\end{equation}
 Note that the sequence $\{u^n\}_{n=1,2,\dots ,N}$ is well defined. Indeed for $n=1 $ the existence and the uniqueness of $u^1 \in \mathcal{D}(\mathcal{R})$ follows from Corollary \ref{L^infty} with $h = \Delta_t h^1 +u^0 \in L^\infty(\Omega)^+$ and $\lambda =\Delta_t$. Hence by induction we obtain in the same way the existence of the solution $u^n$ for any $n=2,3,\dots ,N$ where $u^n \in \mathcal{D}(\mathcal R).$\\
Moreover let denote by $(u_\epsilon^n)$ the solution of \eqref{Itersch} with $\Delta_t= \epsilon , h=r, r^n= \dfrac{1}{\epsilon}\int_{(n-1)\epsilon}^{n \epsilon} r(\tau,.) d\tau$ and $(u_\eta^m)$ the solution of \eqref{Itersch} with $\Delta_t= \eta , h=k, k^m= \dfrac{1}{\eta} \int_{(m-1)\eta}^{m \eta} k(\tau,.) d\tau$ respectively i.e we have 
\begin{equation}\label{fineq}
     \left\{
         \begin{alignedat}{2} 
             {}  \dfrac{u^n_\epsilon -u^{n-1}_{\epsilon}}{\epsilon} + \mathcal{R} u_{\epsilon}^n
             & {}=  r^n ;
             \\
              {}  \dfrac{u^m_\eta -u^{m-1}_{\eta}}{\eta} +\mathcal{R} u^m_{\eta}
             & {}=  k^m. 
        \end{alignedat}
     \right.
\end{equation}
For $(n,m) \in \mathbb{N}^*$, multiplying the equation in  \eqref{fineq} by $\dfrac{\epsilon \eta}{\epsilon + \eta}$ and then subtracting the two expressions we get,
\begin{align*}
    \dfrac{\eta}{\eta +\epsilon } (u_\epsilon^n-u_\epsilon^{n-1}) &+ \dfrac{\eta \epsilon}{\eta +\epsilon } (\mathcal{R} u_\epsilon^n- \mathcal{R}u_\eta^{m})- \dfrac{\epsilon}{\eta +\epsilon } (u_\eta^m-u_\eta^{m-1}) = \dfrac{\eta \epsilon }{\eta +\epsilon } (r^n- k^m).
\end{align*}
Then we infer that
\begin{equation*}
\begin{split}
    u_{\epsilon}^n-u_\eta^m + \dfrac{\epsilon \eta}{\epsilon + \eta} (\mathcal{R}u_{\epsilon}^n-\mathcal{R}u_\eta^m) &= \dfrac{\epsilon \eta}{\epsilon + \eta} (r^n-k^m)+ \dfrac{\eta}{\epsilon + \eta} (u_{\epsilon}^{n-1}-u_\eta^m) \\
    &+ \dfrac{\epsilon}{\epsilon + \eta} (u_{\epsilon}^n-u_\eta^{m-1}).
\end{split}
\end{equation*}
Let $\Phi_{n,m}^{\epsilon,\eta}= \|u_{\epsilon}^n- u_{\eta}^m\|_{2^+}$ and since $\mathcal{R}$ satisfies \eqref{accretive} and setting $\lambda= \dfrac{\epsilon \eta}{\epsilon +\eta}$, we get 
\begin{align*}
    \Phi_{n,m}^{\epsilon,\eta} &= \|u_{\epsilon}^n- u_{\eta}^m\|_{2^+} \leq \|u_{\epsilon}^n- u_{\eta}^m + \dfrac{\epsilon \eta}{\epsilon + \eta} (\mathcal{R}u_\epsilon^n- \mathcal{R}u_\eta^m)\|_{2^+}\\
    &\leq \dfrac{\epsilon \eta}{\epsilon + \eta} \|r^n-k^m\|_{2^+} + \dfrac{\eta}{\epsilon + \eta} \|u_{\epsilon}^{n-1}-u_\eta^m\|_{2^+} + \dfrac{\epsilon}{\epsilon + \eta} \|u_{\epsilon}^n-u_\eta^{m-1}\|_{2^+}. 
\end{align*}
Then by elementary calculations, we get  
$$\Phi_{n,0}^{\epsilon, \eta} = \|u_\epsilon^n- u_{\eta}\|_{2+} \leq b(t_n, r_{\epsilon}, k_\eta)$$  and
$$\Phi_{0,m}^{\epsilon, \eta} \leq b(-s_m,r_{\epsilon}, k_\eta).$$
Then by using above computations we get ,
  $\Phi_{n,m}^{\epsilon,\eta} \leq \psi_{n,m}^{\epsilon,\eta} $ where $ \psi_{n,m}^{\epsilon ,\eta}$ satisfies  
 $$ \psi_{n,m}^{\epsilon \eta}= \dfrac{\epsilon \eta}{\epsilon + \eta} \|(r^n-k^m)\|_{2^+} + \dfrac{\eta}{\epsilon + \eta} \|\psi_{n-1,m}^{\epsilon,\eta}\|_{2^+} + \dfrac{\epsilon}{\epsilon + \eta} \|\psi_{n,m-1}^{\epsilon,\eta}\|_{2^+} $$ and 
  $ \psi_{n,0}^{\epsilon,\eta} = b(t_n, r_\epsilon, k_\eta)$ and $\psi_{0,m}^{\epsilon,\eta} =b(-s_m,r_\epsilon, k_\eta).$\\
For $(t,s) \in (t_{n-1}, t_n) \times (s_{m-1}, s_m)$, set $\phi^{\epsilon, \eta} (t,s)= \|r_{\epsilon}(t)-k_\eta(s)\|_{2^+}, $ 
\begin{equation*}
\psi^{\epsilon, \eta}= \psi_{n,m}^{\epsilon, \eta},\  b_{\epsilon, \eta}(t,r,k)=  b(t_n, r_\epsilon, k_\eta),\  b_{\epsilon, \eta}(-s,r,k)=  b(-s_m, r_\epsilon, k_\eta).
\end{equation*}
Then by elementary calculations $\psi^{\epsilon, \eta}$ satisfies the following discrete version of \eqref{psi},
  \begin{equation*}
     \left\{
         \begin{alignedat}{2} 
             {}  \dfrac{\psi^{\epsilon, \eta} (t,s)- \psi^{\epsilon, \eta} (t-\epsilon,s) }{\epsilon} + \dfrac{\psi^{\epsilon, \eta} (t,s)- \psi^{\epsilon, \eta} (t,s-\eta) }{\eta}   
             & {}=  \phi^{\epsilon, \eta}(t,s);
             \\
             \psi{\epsilon, \eta}(t,0) & {}= b_{\epsilon, \eta}(t,r,k);
             \\ \psi^{\epsilon, \eta}(0,s) & {}= b_{\epsilon, \eta}(s,r,k).
        \end{alignedat}
     \right.
\end{equation*}
 Since $r_{\epsilon} \to r $ in $L^2(Q_T)$ then $b_{\epsilon, \eta} (., r, k) \to b(.,r,k)$ in $L^{\infty}([0,T])$ and $\phi^{\epsilon, \eta} \to \phi $ in $L^{\infty}([0,T] \times[0,T])$ and we deduce that $ \rho_{\epsilon, \eta}= \|\psi^{\epsilon,\eta}- \psi\|_{L^{\infty}([0,T] \times[0,T])} \to 0 $ (for more details see for instance [\cite{*5}, Chapter 4, Lemma 4.3, page 136] and [\cite{*5}, Chapter 4, Proof of Theorem 4.1, page 138]). Therefore,
 \begin{equation*}
     \|u_\epsilon(t)-u_\eta(s)\|_{2^+}= \Phi^{\epsilon, \eta} (t,s) \leq \psi^{\epsilon, \eta} (t,s) \leq \psi(t,s) +\rho_{\epsilon,\eta}.
 \end{equation*}
 %As we observe from Corollary \ref{L^2} and Remark \ref{conversion} that $u_{\epsilon}(t)= v_{\epsilon}^q(t)$ and %$u_{\eta}(t)= v_{\eta}^q(t)$ where $v_{\epsilon}^q(t)$ and $v_{\eta}^q(t)$ are weak solutions of discretization scheme %\eqref{Itersch1}.
 Since $u_{\epsilon}(t)=v_{\epsilon}^q(t)$ and $u_{\eta}(t)=v_{\eta}^q(t)$, we obtain
  \begin{equation}\label{q}
     \|v^q_\epsilon(t)-v^q_\eta(s)\|_{2^+}= \Phi^{\epsilon, \eta} (t,s) \leq \psi^{\epsilon, \eta} (t,s) \leq \psi(t,s) +\rho_{\epsilon,\eta} .
 \end{equation}
 %and with $t=s, r=k=h, v_0=u_0 :$
 %$$\|v^q_{\epsilon}(t)- v^q_{\eta}(t)\|_{2^+} \leq 2 \|u_0-z\|_{2^+} +\rho_{\epsilon, \eta} $$
 %Since $z$ can be chosen in $\mathcal{D}(\mathcal{R})$ arbitrary close to $u_0$, we deduce that $v_\epsilon$ is a cauchy sequence in $L^2(Q_T)^+.$\\
From Theorem \ref{wea}, $v^q_\epsilon$ and $v_\eta^q$ satisfies $0 <\underline{w} < v_{\epsilon}, v_{\eta}< \overline{w}$ where $\underline{w}, \overline{w}$ are subsolution and supersolution defined in \eqref{subsolu} and \eqref{supprob} and $v^q_\epsilon \to v_1^q$ and $v_\eta^q \to v_2^q$ {\it a.e.} in $\Omega$ where $v_1$ and $v_2$ are weak solutions of \eqref{WeakS1} with initial data $u_0,\ v_0$ respectively. Since $v_\epsilon^q \to v_1^q$ and $v_\eta^q \to v_2^q$ in $L^\infty(0,T;L^2(\Omega))$ and passing to the limit in \eqref{q} as $\epsilon, \eta \to 0$  with $t=s$ we get 
\begin{align*}
\|v_1^q(t)- v_2^q(t)\|_{2^+} &\leq \|v_1^q(t)- v_\epsilon^q(t)\|_{2^+} + \|v_\eta^q(t)- v_2^q(t)\|_{2^+}+ \|v_\epsilon^q(t)- v_\eta^q(t)\|_{2^+}\\
&\leq \|u_0^q-z\|_{2^+} +\|v_0^q-z\|_{2^+}  + \int_0^t \|r(\gamma)-k(\gamma)\|_{2^+} d\gamma.
\end{align*}
Then \eqref{ntk} follows since we can choose $z$ arbitrary close to $v_0^q$ and with $r=h$, $k=g$.
\qed
 
%\begin{equation*}
 %   \int_0^T \int_{\Omega} (v^q) \Psi ~dx ~dt + \int_0^T \int_{\Omega} |\nabla v|^{p(x)-2} |\nabla v| \nabla\bigg(\frac{\Psi}{v^{q-1}}\bigg) ~dx ~dt = \int_0^T \int_{\Omega} h(t,x) \Psi ~dx ~dt + \int_0^T \int_{\Omega} \frac{f(x,v)}{v^{q-1}} \Psi ~dx ~
%\end{equation*}

%%%%%%%%%%%%%%%%%%%%%%%%%%%%%%%%%%%%%%%%%%%%%%%%%%%%%%%%% Application 4 %%%%%%%%%%%%%%%%%%%%%%%%%%%%%%%%%%%%%%%%%%%%%%%%%%%%%%
\section{An application to nonhomogeneous operators}
In this final section, we prove Theorem \ref{ray-strict}. To this aim, we first study the properties of a related energy functional.
Let $m\geq 1$ and $K :\Omega \times \mathbb{R}^N \to \mathbb{R}^+$ be a continuous differentiable function which satisfies the following conditions:
\begin{description}
      \item[(k1)] $K \in C^1(\Omega \times \mathbb{R}^N) \cap C^2(\Omega \times \mathbb{R}^N \backslash \{ 0\}).$
 \end{description}
   \begin{description}
      \item[(k2)]  Ellipticity condition: $\exists \ k_1 \geq 0 \ \text{and}\ \gamma\in (0, \infty) \ \text{such that}$   
   \end{description}
       $$\sum_{i,j=1}^N \dfrac{\partial^2 K}{\partial \xi_i \partial\xi_j}(x,\xi) \eta_i \eta_j \geq \gamma (k_1+|\xi|)^{m-2} |\eta|^2.$$
 %for all $\xi \in \mathbb{R}^N \backslash \{ 0\}$ and $\eta \in \mathbb{R}^N$.
 \begin{description}
      \item[(k3)] Growth condition:  $\exists \ k_2 \geq 0 \ \text{and}\  \Gamma \in (0, \infty) \ \text{such that}$   
   \end{description}
     $$\sum_{i,j=1}^N \left|\dfrac{\partial^2 K}{\partial \xi_i \partial\xi_j}(x,\xi)\right| \leq \Gamma (k_2+|\xi|)^{m-2}$$
for all $\xi \in \mathbb{R}^N \backslash \{ 0\}$ and $\eta \in \mathbb{R}^N$.\\
%For the sake of convenience, we use the convention $K(x,\dfrac{0}{0})=0$ and $K(x,\dfrac{\xi}{0})< \infty $ for all %$x\in \Omega$ and $ \xi \in \mathbb{R}^N$. 
\begin{remark}\label{ggrowth}
From the assumption {\bf(k2)}, it  follows that K is strictly convex and from {\bf(k1)}-{\bf(k3)} there exists some positive constant $\gamma_1$ and $\gamma_2$ with $0<\gamma_1 \leq \gamma_2 < +\infty$ and some nonnegative constants $\Gamma_1$ and $\Gamma_2$ such that 
\begin{equation*}
\gamma_1 |\xi|^m - \Gamma_1 \leq K(x,\xi) \leq \gamma_2 |\xi|^m +\Gamma_2
\end{equation*}
for $x \in \Omega$ and $\xi \in \mathbb{R}^N \backslash \{0\}.$
\end{remark}
%This follows from Taylor's formula combined with assumption $(k2)$ and $(k3)$ which imply 
%$$ \gamma \ominus_q (\xi) \leq K(x,\xi)-K(x,0)= \xi \nabla_{\xi} K(x,0)\leq \Gamma \ominus_q(\xi)$$
%where $\ominus_q(\xi)  \eqdef  \dfrac{1}{q(q-1)} ((k+|\xi|)^q - k^q)-\dfrac{k^{q-1}}{q-1}|\xi|$ for $\xi \in \mathbb{R}^N$.
\noindent Consider the associated functional $\mathcal{J}_m$ defined by  
%$\mathcal{J}_m : W_0^{1,p(x)}(\Omega) \to \mathbb{R}$ such that 
\begin{equation*}
\mathcal{J}_m(u)  \eqdef \int_\Omega \dfrac{|u|^{p(x)}}{p(x)}  K\bigg(x,\dfrac{\nabla u }{u}\bigg)^{\frac{p(x)}{m}} ~dx.
\end{equation*}
for any positive function $u\in  W_0^{1,p(x)}(\Omega)$.
%and \textcolor{blue}{ $\mathcal{J}_q$ is well defined using \eqref{well}.}
%and is well defined using \eqref{well}.
%\begin{remark}\label{stric1}
%An example of function satisfying $\bf (k1)-(k3)$ is $K(x,\xi)=(\epsilon + |\xi|^2)^{m/2}.$
%\end{remark} 
\noindent Now we extend Lemma 2.4 in \cite{*1} as follows:
\begin{thm}\label{thm1}
 Let $K : \Omega \times \mathbb{R}^N \to \mathbb{R}^+$ satisfying {\bf(k1)}-{\bf(k3)} for some $m \in [1, p_-].$ Then, the function ${\mathcal E}:$ $\dot{V}_+^m\cap L^\infty(\Omega) \to\R^+$,  defined by ${\mathcal E}(u)\eqdef\mathcal{J}_m(u^{1/m})$, is ray-strictly convex (even strictly convex if $p(\cdot)\not\equiv m$).
\end{thm}
\begin{proof}
We observe that for $u \in \dot{V}^m_+\cap L^\infty(\Omega)$
$$\mathcal{E}(u)= \int_{\Omega}  \dfrac{1}{p(x)}  \left(u K\left(x, \dfrac{\nabla u}{m u}\right)\right)^{\frac{p(x)}{m}}~dx.$$
 Therefore, since for $1 \leq m \leq p_- ,\ t\to t^{p(x)/m}$ is convex in $\R^+$ (even strictly convex if $p(x)>m$) it is enough to prove that 
 $$\dot{V}_+^m \ni u \to u K\bigg(x,\dfrac{\nabla u}{mu}\bigg)$$ is ray-strictly convex.  To achieve this goal, let $\theta \in (0,1)$ and $ u_1, u_2 \in \dot{V}_+^m$ then by using the strict convexity of $K$ we obtain, for $x \in \Omega$
\begin{align*}
     ((1&-\theta)u_1+\theta u_2) K\bigg(x, \dfrac{(1-\theta)\nabla u_1 + \theta \nabla u_2}{m ((1-\theta) u_1 + \theta u_2)} \bigg) \\
     &= ((1-\theta)u_1+\theta u_2) K\bigg(x, \dfrac{(1-\theta) u_1}{((1-\theta) u_1 + \theta u_2)} \dfrac{\nabla u_1}{m u_1}+ \dfrac{\theta u_2}{((1-\theta) u_1 + \theta u_2)}  \dfrac{\nabla u_2}{m u_2} \bigg)\\
     &\leq ((1-\theta)u_1+\theta u_2) \bigg(\dfrac{(1-\theta) u_1 }{((1-\theta) u_1 + \theta u_2) } K\bigg(x, \dfrac{\nabla u_1}{m u_1}\bigg) \\
     &\ \ \ \ \ \ \ \ \ \ \ \ \ \ \ \ \ \ \ \ \ \ \ \ \ \ \ \ \ \ \ \ \ \ \ \ \ \ \ \ \ \ + \dfrac{\theta u_2}{((1-\theta) u_1 + \theta u_2)} K\bigg(x, \dfrac{\nabla u_2}{m u_2}\bigg) \bigg)\\
     &= (1-\theta) u_1 K\bigg(x, \dfrac{\nabla u_1}{m u_1} \bigg) + \theta u_2 K\bigg(x, \dfrac{\nabla u_2}{m u_2}\bigg).
\end{align*}
The above inequality is always strict unless $\displaystyle\frac{\nabla u_1}{u_1}=\frac{\nabla u_2}{u_2}$, {\it i.e.} $u_1/u_2\equiv\mathrm{Const}$.
\end{proof}
%\begin{remark}\label{stric2}
%Remark \ref{stric1} and Theorem \ref{thm1} implies that $u \to \mathcal{J}_m(u^{1/m})$ is strictly convex on  $%\dot{V}_+^m \cap L^{\infty}(\Omega)$ for 
%$$\mathcal{J}_m(u)= \int_{\Omega} \frac{(|\nabla u|^2+ \epsilon u^2)^{p(x)/2}}{p(x)} ~dx.$$
%\end{remark}
%\noindent Now we prove Theorem \ref{ray-strict}. \\
 \noindent\textbf{Proof of Theorem \ref{ray-strict}:}
Consider the functional
$ \mathcal{J}_{\epsilon} : W_0^{1,p(x)} (\Omega) \to \mathbb{R}$ , defined by
 \begin{equation*}
     \mathcal{J}_{\epsilon} (u) = \int_{\Omega} \dfrac{(|\nabla u|^2 + \epsilon u^2)^{p(x)/2}}{p(x)} ~dx - \int_{\Omega} G(x,u) ~dx
 \end{equation*}
 where the potential $G(x,t)$ defined as
 \begin{equation*}
     G(x,t) =
     \left\{
     \begin{alignedat}{2}
         & \int_0^t g(x,s) ds 
          && \quad\mbox{ if }\, 0 \leq t < \infty  \,;
          \\ 
          & 0
          && \quad\mbox{ if }\, -\infty < t < 0 \,.
          \end{alignedat}
    \right.
\end{equation*}
Assumptions $(f_1)$, $(\tilde{g})$ and Remark \ref{ggrowth} ensure that
$\mathcal{J}_{\epsilon}$ is well defined, coercive and continuous. Then there exists at least one global minimizer of $\mathcal{J}_{\epsilon}$ on $W_0^{1,p(x)}(\Omega)$, say $u_0$. We can easily prove that  $u_0$ is nonnegative and nontrivial. \\
 %\begin{equation}\label{4}
  %   0 \leq g(x,s) \leq  \epsilon s^{p_- -1} + C_{\epsilon} \ \text{for all\ } (x,s) \in \overline{\Omega} \times \mathbb{R}^+
% \end{equation}
% Since $W_0^{1,p(x)} \hookrightarrow L^{p_-}(\Omega),$ and then \eqref{4} gives $\mathcal{J}_{\epsilon}$ is well defined for every function $u \in W_0^{1,p(x)}(\Omega)$ and also
 %$$ \int_{\Omega} \dfrac{(|\nabla u|^2 + \epsilon u^2)^{p(x)/2}}{p(x)} ~dx \geq \int_{\Omega} \frac{|\nabla u|^{p(x)}}{p(x)} ~dx$$
 %Now by using inequality (\ref{4}) and  $W_0^{1,p(x)} \hookrightarrow L^{p_-}(\Omega),$ we obtain
% $$\int_{\Omega} G(x,u(x)) ~dx =\int_{\Omega} \int_{0}^u g(x,u(x)) ~dx \leq \epsilon \int_{\Omega} |u(x)|^{p_-} ~dx + C_{\epsilon} \int_{\Omega} |u(x)| ~dx$$ which implies 
 %$ \mathcal{J}_{\epsilon} : W_0^{1,p(x)} (\Omega) \to \mathbb{R}$ is coercive and $\mathcal{J}_{\epsilon}$ is also continuous then there exists atleast one global minimizer of $\mathcal{J}_{\epsilon}$ on $W_0^{1,p(x)}(\Omega).$\\ 
%Let $v_0$ is a global minimizer for the functional $\mathcal{J}_{\epsilon}$, for any $\phi \in W_0^{1,p(x)}(\Omega)$ %and $t \in \mathbb{R}$ we have $\mathcal{J}_{\epsilon}(v_0+ t\phi) - \mathcal{J}_{\epsilon}(v_0) \geq 0$. Using Taylor expension and passing the limits $t \to 0^+$ and $t \to 0^- $, 
Since ${\mathcal J}_\epsilon$ is differentiable, we deduce that $u_0$ is a weak solution of \eqref{2}. Now from  Theorems \ref{bdd and C^0} and \ref{C^1} in Appendix A, we obtain that any weak solution  $u$ to \eqref{2} belongs to $C^{1, \alpha}(\overline\Omega) $ for some $ \alpha \in (0,1)$ and $u > 0 $ in $\Omega$ and $ \dfrac{\partial u}{\partial \vec n} < 0$ on $\partial\Omega$. Therefore any weak solution belongs to $C^0_d(\overline\Omega)^+$.\\
Now we prove that $u_0$ is the unique weak solution to \eqref{2}.  Let $W\,:\,\dot{V}_+^m \to \R$ defined by
$$W(u)= \mathcal{J}_{\epsilon}(u^{1/m})=\int_{\Omega} \dfrac{(|\nabla (u^{1/m})|^2 + \epsilon (u^{1/m})^2)^{p(x)/2}}{p(x)} ~dx - \int_{\Omega} G(x,u^{1/m}) ~dx.$$ 
The assumption $(\tilde{g})$ together with Theorem \ref{thm1} with $K(x,\xi)=(\epsilon + |\xi|^2)^{m/2}$ imply that $W$ is strictly convex.\\
 Let $u_1 $ a weak solution to \eqref{2}. Then setting $v_0\eqdef u_0^m$, $v_1\eqdef u_1^{m} \in \dot{V}_+^m$ and $t \in [0,1]$, we define $\xi(t)\eqdef J_{\epsilon}(((1-t)v_0+ t v_1)^{1/m})$. Since $u_0$ and $u_1$ belong to $C^0_d(\overline\Omega)^+$, $\xi$ is differentiable  in $[0,1]$. From the convexity of ${\mathcal E}$, we have for any $t\in [0,1]$
\begin{equation}\label{eqn}
\xi'(0)\leq \xi'(t)\leq \xi'(1).
\end{equation}
Since $u_0$ and $u_1$ are weak solutions to \eqref{2}, $\xi'(0)=\xi'(1)=0$ and from \eqref{eqn} we get that $\xi$ is constant which contradicts the strict convexity of ${\mathcal E}$ unless $u_0\equiv u_1$.
%$((1-t)u_1^m+ t u_2^m)^{1/m} \in W_0^{1,p(x)}(\Omega)$
%\begin{align*} 
%\xi(t)\eqdef J_{\epsilon}(((1-t)v_0+ t v_2)^{1/m}) &= W(((1-t)v_1+ t v_2))<(1-t) W(v_1)+ t W(v_2)\\
%&= (1-t)\mathcal{J}_{\epsilon}(u_1)+ t \mathcal{J}_{\epsilon}(u_2) = \min_{u \in W_0^{1,p(x)}(\Omega)} \mathcal{J}%_{\epsilon}(u)
%\end{align*}
%which is a contradiction to the fact that $u_1$ and $u_2$ are global minimizer of $\mathcal{J}_{\epsilon}$ on $W_0^{1,p(x)}(\Omega).$
  %%%%%%%%%%%%%%%%%%%%%%%%%%%%%%%%%%%%%%%%%%%%%%%%%%%%%%   Appendix %%%%%%%%%%%%%%%%%%%%%%%%%%%%%%%%%%%%%%%%%%%%%%%%%%%$
  \qed
 \appendix 
\section{Appendix}
In this section, we recall the following regularity of weak solutions of quasilinear elliptic differential equation
\begin{equation}\label{eq1}
     \left\{
         \begin{alignedat}{2} 
             {} \div  A(x,u, Du) + B(x,u,Du) & {}= 0
             && \quad\mbox{ on }\, \Omega\,;\\
             u & {}= 0
             && \quad\mbox{ on }\, \partial\Omega\,.
        \end{alignedat}
     \right.
\end{equation}
 Now we recall the boundedness and $C^{0, \alpha}(\overline\Omega)$ regularity results of weak solutions of \eqref{eq1} satisfying the following growth conditions:
 \begin{equation}\label{1st}
\begin{split}
     &A(x,u,z) z \geq a_0 |z|^{p(x)} - b|u|^{r(x)} -c; \\
 &    |A(x,u,z)| \leq a_1 |z|^{p(x)-1} + b|u|^{\sigma(x)} +c; \\
 &   |B(x,u,z)| \leq a_2 |z|^{\alpha(x)} + b|u|^{r(x)-1} +c
\end{split}
 \end{equation}
 where $a_0, a_1, a_2, b, c$ are positive constants and $p^*$ is the Sobolev embedding exponent of $p$ and
 \begin{equation}\label{2nd}
     p(x) \leq r(x) < p^*(x), \ \sigma(x)= \dfrac{p(x)-1}{p(x)} r(x) \ \text{and}\  \alpha (x)= \dfrac{r(x)-1}{r(x)} p(x).
\end{equation}
\begin{thm}(\cite{*3}, Theorem 4.1 and Theorem 4.4)\label{bdd and C^0}
Let \eqref{1st}-\eqref{2nd} hold and $p \in \mathcal{P}^{log}(\Omega)$. If $u \in W^{1, p(x)}(\Omega)$ is a weak solution of \eqref{eq1}, then $u \in C^{0, \alpha}(\overline{\Omega)}.$ 
     \end{thm}
Theorem \ref{C^1} below ensures $C^{1, \alpha}(\overline\Omega)$ regularity to weak solutions of \eqref{eq1} under the additional assumptions on $ p,\ A$ and $B$: \\
{\bf Assumptions} $\mathbf{(A_k)}: $  $ A= (A_1, A_2, \dots,A_n) \in C(\overline{\Omega} \times \mathbb{R} \times \mathbb{R}^N, \mathbb{R}^N) .$ For every $(x,u) \in \overline{\Omega} \times \mathbb{R},$ $ A(x,u, .) \in C^1 (\mathbb{R}^N \backslash \{0\}, \mathbb{R}^N)$, there exist a nonnegative constants $k_1, k_2, k_3 \geq 0$, a nonincreasing continuous function $\lambda : [0,\infty) \to (0,\infty)$ and  a nondecreasing continuous function $\Lambda : [0,\infty) \to (0,\infty)$ such that for all $x, x_1, x_2 \in \overline{\Omega},\ u, u_1,u_2 \in \mathbb{R} ,\eta \in \mathbb{R}^N \backslash \{0\}$ and $\xi = (\xi_1, \xi_2, \dots, \xi_n) \in \mathbb{R}^N,$ the following conditions are satisfied 
\begin{equation*}
\begin{split}
   &A(x,u,0)=0,\\
&\sum_{i,j} \dfrac{\partial A_j(x,u,\eta)}{\partial \eta_i}(x,u, \eta) \xi_i \xi_j \geq \lambda(|u|) (k_1+ |\eta|^2)^{(p(x)-2)/2} |\xi|^2,
\\
&\sum_{i,j} \bigg|\dfrac{\partial A_j(x,u,\eta)}{\partial \eta_i}(x,u, \eta)\bigg| \leq \Lambda(|u|) (k_2+ |\eta|^2)^{(p(x)-2)/2} \mbox{ and }
\\
&|A(x_1,u_1, \eta) - A(x_2,u_2, \eta)| \leq \Lambda(\max\{|u_1|, |u_2|\}) (|x_1-x_2|^{\beta_1} +|u_1-u_2|^{\beta_2}) \\
    &\ \ \ \ \ \ \times [(k+|\eta|^2)^{(p(x_1)-2)/2}+ (k+|\eta|^2)^{(p(x_2)-2)/2}] |\eta| (1+ |\log (k_3+ |\eta|^2)|).
\end{split}
\end{equation*}
{\bf Assumption} $\mathbf{(B)}: $ $ B :\overline{\Omega} \times \mathbb{R} \times \mathbb{R}^N \to \mathbb{R},$ the function $B(x,u, \eta)$ is measurable in $x$ and is continuous in $(u ,\eta)$, and 
  \begin{equation*}
      |B(x,u, \eta)| \leq \Lambda(|u|) (1+ |\eta|^{p(x)}), \ \ \ \ \ \ \ \forall (x,u,\eta) \in \overline{\Omega} \times \mathbb{R} \times \mathbb{R}^N.
  \end{equation*}
\begin{thm}(\cite{*2}, Theorem 1.2)\label{C^1}
 Let assumptions $\mathbf{(A_k)}$,  $\mathbf{(B)}$ hold. Assume $p$ belongs to $C^{0, \beta}(\overline{\Omega})$, for some $\beta \in (0,1)$. Suppose that $\Omega$ satisfy $(\mathbf{\Omega}).$ If $u \in W_0^{1,p(x)}(\Omega) \cap L^{\infty}(\Omega)$ is a weak solution of \eqref{eq1}, then $u \in C^{1, \alpha}(\overline{\Omega})$ where $\alpha \in (0,1)$ and $\|u\|_{C^{1,\alpha}(\overline{\Omega})}$ depends upon $ p_{-}, p_{+}, \lambda(M), \Lambda(M), \beta_1, \beta_2, M, \Omega$ where $M \eqdef \|u\|_{L^{\infty}(\Omega)}.$ 
\end{thm}
\noindent 
In the next theorem, we recall some results contained in Lemma $2.1$ of \cite{*31} and Lemma $3.2$ of \cite{*10}. Set $\varrho= \frac{p_-}{2|\Omega|^{1/N} C_0}$ where $C_0$ is the best embedding constant of $W_0^{1,1}(\Omega) \subset L^{\frac{N}{N-1}}(\Omega).$
\begin{thm}\label{regi}
Let $K>0$ and $w_K \in W_0^{1,p(x)}(\Omega) \cap C^{1, \alpha}(\overline\Omega)$ be the weak solution of  \eqref{lambd}.
%\begin{equation}\label{lambd}
  %  \left\{
     %    \begin{alignedat}{2} 
      %   -\Delta_{p(x)} w_\kappa & {} = \kappa
        %    && \quad\mbox{ in }\, \Omega;
          %   \\
           % w_\kappa & {}= 0
           %&& \quad\mbox{ on }\, \partial\Omega;
       % \end{alignedat}
    % \right.
%\end{equation}
Then for $K\geq \varrho$, $\|w_K\|_{L^{\infty}(\Omega)} \leq C_1 K^{1/(p^--1)}$, $w_K(x) \geq C_2 K^{1/(p_+-1+ \varsigma)} \dist(x, \partial\Omega)$ where $\varsigma \in (0,1)$ and for $K < \varrho$, $\|w_{K}\|_{L^{\infty}(\Omega)} \leq C_3 K^{1/(p^+-1)}$ where $C_1, C_2$ and $C_3$ depends upon $p_+, p_-, N,\ \Omega.$ Moreover if $K_1 < K_2$ then $w_{K_1} \leq w_{K_2}.$ 
\end{thm}
\noindent Next we prove a slight extension of Proposition A.2 in \cite{*10}.
\begin{pro}\label{rg1}
Let $p\in C(\bar\Omega)$ and $ q \in (1,p_-]$. Assume $u\in \mathbf W$ satisfying for any $\Psi \in\mathbf W$:
\begin{equation}\label{ws1}
\int_\Omega |\nabla u|^{p(x)-2}\nabla u .\nabla \Psi ~dx=\int_\Omega h u^{q-1} \Psi ~dx
\end{equation}
where $h\in L^2(\Omega)\cap L^{r}(\Omega)$ with $r >\max\{1,\frac{N}{p_-}\}$. Then $u\in L^\infty(\Omega)$.
\end{pro}
\noindent First we prove a regularity lemma.  
\begin{Lem}\label{fsb1}
Let $u\in W_0^{1,p}(\Omega)$ satisfying for any  $B_R$, $R<R_0$, and for all $\sigma\in (0,1)$, and
any $k\geq k_0>0$
\begin{equation*}%\label{equa}
\begin{split}
\int_{A_{k,\sigma R}} |\nabla u|^p ~dx &\leq C\left[ \int_{A_{k, R}}
\left|\frac{u-k}{R(1-\sigma)}\right|^{p^*}~dx+k^\alpha|A_{k,R}|+ |A_{k,R}|^{\frac{p}{p^*}+\varepsilon}\right.\\
&\left.+k^\beta |A_{k,R}|^{\frac{p}{p^*}+\varepsilon}+\left(\int_{A_{k, R}} \left|\frac{u-k}{R(1-\sigma)}\right|^{p^*}~dx\right)^{\frac{p}{p^*}}|A_{k,R}|^\delta\right]
\end{split}
\end{equation*}
where $A_{k,R}=\{x\in B_R\cap \Omega\ |\ u(x)>k\}$, $0<\alpha<p^*=\frac{Np}{N-p}$, $\beta \in (1,p]$ and $\varepsilon,\,\delta>0$. Then $u\in L^\infty(\Omega)$.
\end{Lem}
\begin{proof}
A similar result exists in \cite{FS} or in \cite{GTW} without the term $k^\beta |A_{k,R}|^{\frac{p}{p^*}+\varepsilon}$. For the reader's convenience, we include the complete proof.\\
Let $x_0\in  \overline\Omega$, $B_R$ the ball centred in $x_0$. We define $K_R \eqdef  B_R\cap \Omega$ and we set 
$$r_j=\frac{R}{2}+\frac{R}{2^{j+1}}, \quad \tilde r_j=\frac{r_j+r_{j+1}}{2} \ \mbox{ and } \ k_j=k\left(1-\frac{1}{2^{j+1}} \right) \ \ \mbox{for any $j\in \N$}.$$
Define also
$$I_j=\int_{A_{k_j,r_j}} |u(x)-k_j|^{p^*}~dx \quad \mbox{and} \quad \varphi(t)=\left\{
\begin{array}{l l}
1 & \mbox{if}\ 0\leq t\leq \frac12,\\
0 & \mbox{if}\ t\geq \frac34
\end{array}
\right.\
$$
satisfying $\varphi\in C^1([0,+\infty))$ and $0\leq \varphi \leq 1$. We set $\varphi_j(x)=\varphi\left(\frac{2^{j+1}}{R}(|x|-\frac{R}{2})\right)$. Hence $\varphi_j=1$ on $B_{r_{j+1}}$ and $\varphi_j=0$ on $\R^N\backslash B_{\tilde r_{j+1}}$.\\
We have
\begin{equation*}
\begin{split}
I_{j+1}&=\int_{A_{k_{j+1},r_{j+1}}} |u(x)-k_{j+1}|^{p^*}~dx=\int_{A_{k_{j+1},r_{j+1}}} |u(x)-k_{j+1}|^{p^*}\varphi_j(x)^{p^*}~dx\\
&\leq \int_{K_R} (u(x)-k_{j+1})^+\varphi_j(x))^{p^*}~dx.
\end{split}
\end{equation*}
Since $u\in W^{1,p}_0(\Omega)$, $(u-k_{j+1})^+\varphi_j\in W^{1,p}_0(K_R)$,
\begin{equation*}
\begin{split}
I_{j+1}&{\lesssim} \left(\int_{K_R} |\nabla((u-k_{j+1})^+\varphi_j)|^{p}~dx\right)^{\frac{p^*}{p}}\\
&\lesssim \left(\int_{A_{k_{j+1},\tilde r_{j}}} |\nabla u|^{p}~dx+\int_{A_{k_{j+1},\tilde r_{j}}}(u-k_{j+1})^{p}~dx\right)^{\frac{p^*}{p}}
\end{split}
\end{equation*}
where we use the notation $f\lesssim g$ in the sense there exists a constant $c>0$ such that $f\leq c g$. Since $\tilde r_j<r_j$, we have
\begin{equation}\label{Jh1}
\begin{split}
I_{j+1}&\lesssim \left(2^{jp^*}\int_{A_{k_{j+1},r_{j}}}|u-k_{j+1}|^{p^*}~dx +k_{j+1}^\alpha|A_{k_{j+1},r_{j}}|+|A_{k_{j+1},r_{j}}|^{\frac{p}{p^*}+\varepsilon}\right.\\
&\left.\ \ \ \ \ \ \ +k_{j+1}^\beta|A_{k_{j+1},r_{j}}|^{\frac{p}{p^*}+\varepsilon}+2^{jp}\left(\int_{A_{k_{j+1},r_{j}}}|u-k_{j+1}|^{p^*}~dx\right)^{\frac{p}{p^*}}|A_{k_{j+1},r_{j}}|^\delta\right.\\
&\left.\ \ \ \ \ \ \ \ +{\int_{A_{k_{j+1},r_{j}}}|u-k_{j+1}|^{p^*}~dx}\right)^{\frac{p^*}{p}}.
\end{split}
\end{equation}
Moreover, for any $j$, $k_j\leq k_{j+1}$, this implies
\begin{equation*}
I_j\geq \int_{A_{k_{j+1},r_{j}}}|u-k_{j}|^{p^*}~dx\geq \int_{A_{k_{j+1},r_{j}}}|k_j-k_{j+1}|^{p^*}~dx=|A_{k_{j+1},r_{j}}\|k_{j+1}-k_j|^{p^*}.
\end{equation*}
Then, for any $k>k_0$ and $j\in \N$
$$  |A_{k_{j+1},r_{j}}|+k_{j+1}^{p^*}|A_{k_{j+1},r_{j}}|\lesssim 2^{jp^*}I_j$$
where the constant in the notation depends only on $k_0$, $p$ and $\alpha$. From the previous inequality, we deduce
$$k_{j+1}^{\beta}|A_{k_{j+1},r_{j}}|^{\frac{p}{p^*}+\varepsilon}\leq k_{j+1}^{p+\varepsilon p^*}|A_{k_{j+1},r_{j}}|^{\frac{p}{p^*}+\varepsilon}\lesssim 2^{j(p+\varepsilon p^*)}I_j^{\frac{p}{p^*}+\varepsilon}.$$
Replacing in \eqref{Jh1}, we obtain
\begin{equation}\label{Jh2}
I_{j+1} \lesssim \left(2^{jp^*}I_j+ 2^{j(p+\varepsilon p^*)}I_j^{\frac{p}{p^*}+\varepsilon}+2^{j(p+\delta p^*)}I_j^{\frac{p}{p^*}+\delta}\right)^{\frac{p^*}{p}}.
\end{equation}
Setting $M=\frac{p}{p^*}\max\{p^*,\, p+\varepsilon p^*,\, p+\delta p^*\}$ and $\theta=\min\{1-\frac{p}{p^*},\,\varepsilon,\,\delta\}$ and noting 
$$I_j\leq \int_{K_R} (|u-k_j|^+)^{p^*}~dx\leq \int_{K_R}|u|^{p*}\leq \|u\|_{W^{1,p}_0}^{p*},$$
\eqref{Jh2} becomes
\begin{equation*}
I_{j+1}\lesssim 2^{jM}I_j^{1+\frac{\theta p^*}{p}}
\end{equation*}
where the constant depends on $\|u\|_{W^{1,p}_0}$, $k_0$, $\alpha$ and $p$. We conclude with Lemma 4.7 in Chapter 2 of \cite{LU}.\\
For this it suffices to prove that $I_0$ is small enough. Indeed $u\in L^{p^*}(\Omega)$ implies
$$I_0=\int_{A_{\frac{k}{2},R}} |u-\frac{k}{2}|^{p^*}~dx \to 0 \quad\mbox{ as }\quad k\to \infty.$$
Hence for $k$ large enough, $I_0\leq C^{-\frac1\eta}(2^M)^{-\frac{1}{\eta^2}}$ with $\eta=\frac{\theta p^*}{p}$. Thus $I_j$ converges to $0$ as $j\to +\infty$ and 
$$\int_{A_{k,\frac{R}{2}}} |u-k|^{p^*}~dx=0.$$
We deduce that $u\leq k$ on $K_{\frac{R}{2}}$. In the same way, we prove that $-u\leq k$ on $K_{\frac{R}{2}}$. \\
Since $\overline \Omega$ is compact, we conclude that $u\in L^\infty(\Omega)$.
\end{proof}
\noindent \textbf{Proof of Proposition \ref{rg1}:} We follow the idea of the proof of Theorem 4.1 in \cite{*3}.\\
Let $x_0\in \overline\Omega$, $B_{R}$ the ball of radius $R$ centered in $x_0$ and $K_R \eqdef  \Omega\cap B_R$. We define
$$p^+ \eqdef  \max_{K_{R}} p(x) \quad \mbox{and} \quad p^- \eqdef  \min_{K_{R}} p(x)$$
and we choose $R$ small enough such that $p^+< (p^-)^*$
where $$(p^-)^* \eqdef \left\{\begin{array}{l l}
\frac{Np^-}{N-p^-} & \mbox{if } p^-< N,\\
p^+ +1 & \mbox{if }p^- \geq N.
\end{array}\right.$$ 
Fix $(s,t)\in (\R_+^*)^2$, $t<s<R$ then  $K_t\subset K_s\subset K_R$. Define $\varphi\in
C^\infty(\Omega)$, $0\leq \varphi\leq 1$ such that
$$\varphi=\left\{\begin{array}{l l}
1 & \mbox{in } B_t,\\
0 & \mbox{in }\R^N\backslash B_s
\end{array}\right.$$
{satisfying $|\nabla\varphi|\lesssim 1/(s-t)$}. Let $k\geq 1$, using the same notations as previously $A_{k,\lambda}=\{y\in K_\lambda \ | \ u(y)>k\}$ and taking $\Psi=\varphi^{p^+}(u-k)^+\in W^{1,p(x)}_0(\Omega)$ in \eqref{ws1}, we obtain
\begin{equation}\label{B3}
\begin{split}
\int_{A_{k,s}} |\nabla u|^{p(x)}\varphi^{p^+}\,~dx &+p^+\int_{A_{k,s}} |\nabla u|^{p(x)-2}\nabla u\cdot\nabla \varphi
\varphi^{p^+-1}(u-k)^+\,~dx\\
&\ \ \ \ \ =\int_{A_{k,s}} hu^{q-1}\varphi^{p^+}(u-k)\,~dx.
\end{split}
\end{equation}
{Hence} by Young inequality, for $\epsilon>0$, we have
\begin{equation*}
\begin{split}
p^+\int_{A_{k,s}} |\nabla u|^{p(x)-2}\nabla u\cdot\nabla \varphi \varphi^{p^+-1}(u-k)\,~dx&\leq \varepsilon \int_{A_{k,s}}
|\nabla
u|^{p(x)}\varphi^{(p^+-1)\frac{p(x)}{p(x)-1}}\,~dx\\
&+c\varepsilon^{-1}\int_{A_{k,s}}
(u-k)^{p(x)}|\nabla\varphi|^{p(x)}\,~dx.
\end{split}
\end{equation*}
Since $|\nabla\varphi|\leq c/(s-t)$ and for any $x\in K_R$, $p^+\leq (p^+-1)\frac{p(x)}{p(x)-1}$, we have
$\varphi^{(p^+-1)\frac{p(x)}{p(x)-1}}\leq \varphi^{p^+}$. This implies
\begin{equation}\label{est1}
\begin{split}
p^+\int_{A_{k,s}} |\nabla u|^{p(x)-2}\nabla u.\nabla \varphi \varphi^{p^+-1}(u-k)\,~dx &\leq \varepsilon \int_{A_{k,s}} |\nabla
u|^{p(x)}\varphi^{p^+}\,~dx \\
&+ c\varepsilon^{-1}\int_{A_{k,s}} \left(\frac{u-k}{s-t}\right)^{p(x)}\,~dx.
\end{split}
\end{equation}
Using H\"older inequality, we estimate the right-hand side of \eqref{B3} as follows:
\begin{equation*}
\int_{A_{k,s}} hu^{q-1}\varphi^{p^+}(u-k)\,~dx\leq \|h\|_{L^r}\left(\int_{A_{k,s}}
u^{\frac{r(q-1)}{r-1}}(u-k)^{\frac{r}{r-1}}\,~dx\right)^{\frac{r-1}{r}}.
\end{equation*}
Since $r>\frac{N}{p^-}$, we have $\frac{(p^-)^*}{p^-}\frac{r-1}{r}>1$, applying once again the H\"older inequality and the Young inequality, we
obtain
%\begin{equation*}
%\int_{A_{k,s}} hu^{q-1}\varphi^{p^+}(u-k)\,~dx\leq C \left(\int_{A_{k,s}}
%u^{(q-1)\frac{(p^-)^*}{p^-}}(u-k)^{\frac{(p^-)^*}{p^-}}\,~dx\right)^{\frac{p^-}{(p^-)^*}}|A_{k,s}|^\delta,
%\end{equation*}
%. Using now Young inequality, we obtain
\begin{equation*}
\int_{A_{k,s}} hu^{q-1}\varphi^{p^+}(u-k)\,~dx\lesssim \left(\int_{A_{k,s}}
u^{\frac{q(p^-)^*}{p^-}}\,~dx +\int_{A_{k,s}}(u-k)^{\frac{q(p^-)^*}{p^-}}\,~dx\right)^{\frac{p^-}{(p^-)^*}}|A_{k,s}|^\delta
\end{equation*}
where $\delta= \frac{r-1}{r}-\frac{p^-}{(p^-)^*}>0$.\\
Set $A_{k,s,t}=\{x\in A_{k,s}\ |\ u(x)-k>s-t\}$ and
its complement as $A^c_{k,s,t}$. Now we split the integrals in the right-hand side of \eqref{est3} as follows:
\begin{equation}\label{est3}
\begin{split}
\int_{A_{k,s,t}} \left(\frac{u-k}{s-t}\right)^{\frac{q(p^-)^*}{p^-}}(s-t)^{\frac{q(p^-)^*}{p^-}}\,~dx &+
\int_{A^c_{k,s,t}} \left(\frac{u-k}{s-t}\right)^{\frac{q(p^-)^*}{p^-}}(s-t)^{\frac{q(p^-)^*}{p^-}}\,~dx\\
&\lesssim \int_{A_{k,s}} \left(\frac{u-k}{s-t}\right)^{(p^-)^*}\,~dx+|A_{k,s}| \eqdef \mathcal I
\end{split}
\end{equation}
since $q<p_-$ and we also have
\begin{equation*}
\int_{A_{k,s}} u^{\frac{q(p^-)^*}{p^-}}\,~dx\lesssim \int_{A_{k,s}} (u-k)^{\frac{q(p^-)^*}{p^-}}+k^{\frac{q(p^-)^*}{p^-}}\,~dx\lesssim \mathcal I+ k^{\frac{q(p^-)^*}{p^-}}|A_{k,s}|.
\end{equation*} 
In the same way, the second term in the right-hand side of \eqref{est1} can be estimated as follows:
\begin{equation}\label{est4}
\int_{A_{k,s}\cap A_{k,s,t}} \left(\frac{u-k}{s-t}\right)^{p(x)}\,~dx +\int_{A_{k,s}\cap A^c_{k,s,t}}
\left(\frac{u-k}{s-t}\right)^{p(x)}\,~dx \lesssim \mathcal I.
\end{equation}
Finally plugging \eqref{est1}-\eqref{est4}, we obtain for $\varepsilon$ small enough
\begin{equation*}
\int_{A_{k,s}} |\nabla u|^{p(x)}\varphi^{p^+}\,~dx\lesssim  \mathcal I + |A_{k,s}|^\delta(\mathcal I+ k^{\frac{q(p^-)^*}{p^-}}|A_{k,s}|)^{\frac{p^-}{(p^-)^*}}
\end{equation*}
where the constant depends on $p,\,R$ and $\varepsilon$. Moreover we have
\begin{align*}
(\mathcal I+k^{\frac{q(p^-)^*}{p^-}}|A_{k,s}|)^{\frac{p^-}{(p^-)^*}} &\lesssim \left(\int_{A_{k,s}}
\left(\frac{u-k}{s-t}\right)^{(p^-)^*}\,~dx\right)^{\frac{p^-}{(p^-)^*}}\\
&+|A_{k,s}|^{\frac{p^-}{(p^-)^*}} +k^{q}|A_{k,s}|^{\frac{p^-}{(p^-)^*}}.
\end{align*}
To conclude, using the Young inequality, we obtain the following estimate:
\begin{equation*}
\begin{split}
\int_{A_{k,t}} &|\nabla u|^{p^-}\,~dx\leq\int_{A_{k,s}} |\nabla u|^{p(x)}\varphi^{p^+}\,~dx
\lesssim \int_{A_{k,s}} \left(\frac{u-k}{s-t}\right)^{(p^-)^*}\,~dx + 2|A_{k,s}|\\
&+(1+k^q)|A_{k,s}|^{\frac{p^-}{(p^-)^*}+\delta} +|A_{k,s}|^\delta\left(\int_{A_{k,s}}
\left(\frac{u-k}{s-t}\right)^{(p^-)^*}\,~dx\right)^{\frac{p^-}{(p^-)^*}}.
\end{split}
\end{equation*}
By Lemma \ref{fsb1}, we deduce that $u$ bounded in $\Omega$. 
$\hfill \square$\ \\
\ \\
Combining Theorem 4.1 of \cite{*3} and Proposition \ref{rg1}, we have the following corollary:
%\begin{cor}
%Let $p\in C(\bar\Omega)$ and $q\in (1,p_-)$. Assume $u\in\mathbf W$ satisfying for any $\Psi \in \mathbf W$:
%\begin{equation*}
%\int_\Omega |\nabla u|^{p(x)-2}\nabla u \cdot\nabla \Psi ~dx=\int_\Omega(f(x,u)+hu^{q-1}) \Psi ~dx
%\end{equation*}
%where $f$ verifies $|f(x,t)|\leq c_1+c_2|t|^{s(x)-1}$ with $s\in C(\overline\Omega)$ such that for any $x\in\overline \Omega$, $1<s(x)<p^*(x)$ and $h\in L^2(\Omega)\cap L^{r}(\Omega)$ with $r>\max\{1,\frac{N}{p_-}\}$. Then $u\in L^\infty(\Omega)$.
%\end{cor}
%These results can be extended considering the following nonnegative weak subsolution:
\begin{cor}\label{reg3}
Let $p\in C(\bar\Omega)$ and $q\in (1,p_-]$. Assume $u\in \mathbf W$ and nonnegative satisfying for any $\Psi \in \mathbf W$, $\Psi\geq 0$,
\begin{equation*}
\int_\Omega u^{2q-1}\Psi ~dx+\int_\Omega |\nabla u|^{p(x)-2}\nabla u \cdot\nabla \Psi ~dx\leq\int_\Omega(f(x,u)+hu^{q-1}) \Psi ~dx 
\end{equation*}
where $f$ verifies for any $(x,t)\in \Omega\times \R^+$, $|f(x,t)|\leq c_1+c_2|t|^{s(x)-1}$ with $s\in C(\overline\Omega)$ such that for any $x\in\overline \Omega$, $1<s(x)<p^*(x)$ and $h\in L^2(\Omega)\cap L^{r}(\Omega)$ with $r>\max\{1,\frac{N}{p_-}\}$. Then $u\in L^\infty(\Omega)$.
\end{cor}

\end{document}